\documentclass{amsart}
\usepackage{mathpartir,mathtools,amsmath,amssymb,phonetic,aliascnt,ifmtarg,enumitem,scalefnt}
\usepackage{xcolor}
\definecolor{darkgreen}{rgb}{0,0.45,0} 
\usepackage{hyperref}
\hypersetup{colorlinks,citecolor=darkgreen,linkcolor=darkgreen}
\usepackage{cleveref}
\usepackage[all]{xy}
\usepackage{tikz}
\usetikzlibrary{calc}

\usepackage[backend=bibtex,style=authoryear,citestyle=authoryear-comp,backref=true,backrefstyle=two]{biblatex}
\DeclareFieldFormat{postnote}{#1}
\DeclareFieldFormat{multipostnote}{#1}

\makeatletter
\def\flet#1:{\mathsf{let}\;#1 \@ifnextchar:\@fletdoublecolon\@fletsinglecolon}
\def\fletp#1:{\mathsf{let}'\;#1 \@ifnextchar:\@fletdoublecolon\@fletsinglecolon}
\def\@fletdoublecolon:=#1in{\Coloneqq #1\;\mathsf{in}\;}
\def\@fletsinglecolon=#1in{\coloneqq #1\;\mathsf{in}\;}
\def\type{\mathsf{Type}}
\def\inr{\mathsf{inr}}
\def\inl{\mathsf{inl}}

\def\refl{\mathsf{refl}}
\def\prop{\mathsf{Prop}}

\def\base{\mathsf{base}}
\def\floop{\mathsf{loop}}

\def\id{\mathsf{id}}
\def\ap{\mathsf{ap}}
\def\J{\mathsf{J}}
\def\disc{\mathsf{Disc}}
\def\discprop{\mathsf{DiscProp}}
\def\isdisc{\mathsf{isdiscrete}}
\def\codisc{\mathsf{Codisc}}
\def\iscodisc{\mathsf{iscodiscrete}}
\def\codiscprop{\mathsf{CodiscProp}}
\def\code{\mathsf{code}}
\def\encode{\mathsf{encode}}
\def\decode{\mathsf{decode}}
\def\const{\mathsf{const}}
\def\oo{\ensuremath{\infty}}
\def\fib{\mathsf{fib}}
\def\proj{\mathsf{pr}}
\def\inbox{\mathsf{in}_{\Box}}

\newcommand{\jdeq}{\equiv}      % An equality judgment
\newcommand{\defeq}{\vcentcolon\equiv}  % A judgmental equality currently being defined

\def\shc{\sim\hspace{-5pt}:\hspace{3pt}}
\def\@shc:{\shc}
\let\@tilde~
\def~{\@ifnextchar:\@shc\@tilde}

\newcommand{\shape}{\ensuremath{\mathord{\raisebox{0.5pt}{\text{\rm\esh}}}}}

\newcommand{\T}{\mathcal{T}}
\newcommand{\bR}{\R} % {\mathbf{R}}

\def\oo{\ensuremath{\infty}}
\let\xto\xrightarrow
\def\apart{\mathbin{\#}}
\let\toto\rightrightarrows

\crefformat{section}{\S#2#1#3}
\Crefformat{section}{Section~#2#1#3}
\crefrangeformat{section}{\S\S#3#1#4--#5#2#6}
\Crefrangeformat{section}{Sections~#3#1#4--#5#2#6}
\crefmultiformat{section}{\S\S#2#1#3}{ and~#2#1#3}{, #2#1#3}{ and~#2#1#3}
\Crefmultiformat{section}{Sections~#2#1#3}{ and~#2#1#3}{, #2#1#3}{ and~#2#1#3}
\crefrangemultiformat{section}{\S\S#3#1#4--#5#2#6}{ and~#3#1#4--#5#2#6}{, #3#1#4--#5#2#6}{ and~#3#1#4--#5#2#6}
\Crefrangemultiformat{section}{Sections~#3#1#4--#5#2#6}{ and~#3#1#4--#5#2#6}{, #3#1#4--#5#2#6}{ and~#3#1#4--#5#2#6}

\crefformat{figure}{Figure~#2#1#3}
\Crefformat{figure}{Figure~#2#1#3}

\makeatletter
\def\prd#1{\@ifnextchar\bgroup{\prd@parens{#1}}{%
    \@ifnextchar\sm{\prd@parens{#1}\@eatsm}{%
    \@ifnextchar\prd{\prd@parens{#1}\@eatprd}{%
    \@ifnextchar\;{\prd@parens{#1}\@eatsemicolonspace}{%
    \@ifnextchar\\{\prd@parens{#1}\@eatlinebreak}{%
    \@ifnextchar\narrowbreak{\prd@parens{#1}\@eatnarrowbreak}{%
      \prd@noparens{#1}}}}}}}}
\def\prd@parens#1{\@ifnextchar\bgroup%
  {\mathchoice{\@dprd{#1}}{\@tprd{#1}}{\@tprd{#1}}{\@tprd{#1}}\prd@parens}%
  {\@ifnextchar\sm%
    {\mathchoice{\@dprd{#1}}{\@tprd{#1}}{\@tprd{#1}}{\@tprd{#1}}\@eatsm}%
    {\mathchoice{\@dprd{#1}}{\@tprd{#1}}{\@tprd{#1}}{\@tprd{#1}}}}}
\def\@eatsm\sm{\sm@parens}
\def\prd@noparens#1{\mathchoice{\@dprd@noparens{#1}}{\@tprd{#1}}{\@tprd{#1}}{\@tprd{#1}}}
\def\lprd#1{\@ifnextchar\bgroup{\@lprd{#1}\lprd}{\@@lprd{#1}}}
\def\@lprd#1{\mathchoice{{\textstyle\prod}}{\prod}{\prod}{\prod}({\textstyle #1})\;}
\def\@@lprd#1{\mathchoice{{\textstyle\prod}}{\prod}{\prod}{\prod}({\textstyle #1}),\ }
\def\tprd#1{\@tprd{#1}\@ifnextchar\bgroup{\tprd}{}}
\def\@tprd#1{\mathchoice{{\textstyle\prod_{(#1)}}}{\prod_{(#1)}}{\prod_{(#1)}}{\prod_{(#1)}}}
\def\dprd#1{\@dprd{#1}\@ifnextchar\bgroup{\dprd}{}}
\def\@dprd#1{\prod_{(#1)}\,}
\def\@dprd@noparens#1{\prod_{#1}\,}

\def\@eatnarrowbreak\narrowbreak{%
  \@ifnextchar\prd{\narrowbreak\@eatprd}{%
    \@ifnextchar\sm{\narrowbreak\@eatsm}{%
      \narrowbreak}}}
\def\@eatlinebreak\\{%
  \@ifnextchar\prd{\\\@eatprd}{%
    \@ifnextchar\sm{\\\@eatsm}{%
      \\}}}
\def\@eatsemicolonspace\;{%
  \@ifnextchar\prd{\;\@eatprd}{%
    \@ifnextchar\sm{\;\@eatsm}{%
      \;}}}

\def\lam#1{{\lambda}\@lamarg#1:\@endlamarg\@ifnextchar\bgroup{.\,\lam}{.\,}}
\def\@lamarg#1:#2\@endlamarg{\if\relax\detokenize{#2}\relax #1\else\@lamvar{\@lameatcolon#2},#1\@endlamvar\fi}
\def\@lamvar#1,#2\@endlamvar{(#2\,{:}\,#1)}
\def\@lameatcolon#1:{#1}

\def\lamu#1{{\lambda}\@lamuarg#1:\@endlamuarg\@ifnextchar\bgroup{.\,\lamu}{.\,}}
\def\@lamuarg#1:#2\@endlamuarg{#1}

\def\fall#1{\forall (#1)\@ifnextchar\bgroup{.\,\fall}{.\,}}

\def\exis#1{\exists (#1)\@ifnextchar\bgroup{.\,\exis}{.\,}}

\def\sm#1{\@ifnextchar\bgroup{\sm@parens{#1}}{%
    \@ifnextchar\prd{\sm@parens{#1}\@eatprd}{%
    \@ifnextchar\sm{\sm@parens{#1}\@eatsm}{%
    \@ifnextchar\;{\sm@parens{#1}\@eatsemicolonspace}{%
    \@ifnextchar\\{\sm@parens{#1}\@eatlinebreak}{%
    \@ifnextchar\narrowbreak{\sm@parens{#1}\@eatnarrowbreak}{%
        \sm@noparens{#1}}}}}}}}
\def\sm@parens#1{\@ifnextchar\bgroup%
  {\mathchoice{\@dsm{#1}}{\@tsm{#1}}{\@tsm{#1}}{\@tsm{#1}}\sm@parens}%
  {\@ifnextchar\prd%
    {\mathchoice{\@dsm{#1}}{\@tsm{#1}}{\@tsm{#1}}{\@tsm{#1}}\@eatprd}%
    {\mathchoice{\@dsm{#1}}{\@tsm{#1}}{\@tsm{#1}}{\@tsm{#1}}}}}
\def\@eatprd\prd{\prd@parens}
\def\sm@noparens#1{\mathchoice{\@dsm@noparens{#1}}{\@tsm{#1}}{\@tsm{#1}}{\@tsm{#1}}}
\def\lsm#1{\@ifnextchar\bgroup{\@lsm{#1}\lsm}{\@@lsm{#1}}}
\def\@lsm#1{\mathchoice{{\textstyle\sum}}{\sum}{\sum}{\sum}({\textstyle #1})\;}
\def\@@lsm#1{\mathchoice{{\textstyle\sum}}{\sum}{\sum}{\sum}({\textstyle #1}),\ }
\def\tsm#1{\@tsm{#1}\@ifnextchar\bgroup{\tsm}{}}
\def\@tsm#1{\mathchoice{{\textstyle\sum_{(#1)}}}{\sum_{(#1)}}{\sum_{(#1)}}{\sum_{(#1)}}}
\def\dsm#1{\@dsm{#1}\@ifnextchar\bgroup{\dsm}{}}
\def\@dsm#1{\sum_{(#1)}\,}
\def\@dsm@noparens#1{\sum_{#1}\,}

\makeatother

\let\types\vdash

\def\istrue{\;\mathsf{true}}

\newcommand{\blank}{\mathord{\hspace{1pt}\text{--}\hspace{1pt}}}

\newcommand{\trunc}[2]{\mathopen{}\left\Vert #2\right\Vert_{#1}\mathclose{}}

\newcommand{\tproj}[3][]{\mathopen{}\left|#3\right|_{#2}^{#1}\mathclose{}}
\newcommand{\tprojf}[2][]{|\blank|_{#2}^{#1}}

\newcommand{\brck}[1]{\trunc{}{#1}}

\newcommand{\unit}{\ensuremath{\mathbf{1}}}
\newcommand{\ttt}{\ensuremath{\mathsf{tt}}}

\newcommand{\bool}{\ensuremath{\mathbf{2}}}
\newcommand{\btrue}{{1_{\bool}}}
\newcommand{\bfalse}{{0_{\bool}}}

\newcommand{\N}{\ensuremath{\mathbb{N}}}
\newcommand{\Z}{\ensuremath{\mathbb{Z}}}
\newcommand{\Q}{\ensuremath{\mathbb{Q}}}
\newcommand{\R}{\ensuremath{\mathbb{R}}}

\newcommand{\coeq}{\mathsf{coeq}}
\newcommand{\topcirc}{\ensuremath{\mathbb{S}^1}}
\newcommand{\topsphere}{\ensuremath{\mathbb{S}^2}}
\newcommand{\hocirc}{\ensuremath{S^1}}
\newcommand{\topdisc}{\ensuremath{\mathbb{D}^2}}
\newcommand{\hosphere}{\ensuremath{S^2}}

\newcommand{\C}{\mathcal{C}}
\newcommand{\Sh}{\mathrm{Sh}}

\newcommand{\sharpf}{(\blank)^\sharp}
\newcommand{\flatf}{(\blank)_\flat}

\newcommand{\mprop}{proposition}

\newcommand{\mysc}[1]{{\scalefont{0.76}#1}}

\newtheorem{thm}{Theorem}[section]
\crefname{thm}{Theorem}{Theorems}
\newtheorem{thmua}[thm]{Theorem \mysc{\{UA\}}}
\crefname{thmua}{Theorem}{Theorems}

\crefname{thmlem}{Theorem}{Theorems}
\newtheorem{thmlemse}[thm]{Theorem \mysc{\{LEM,$\sharp\emptyset$\}}}
\crefname{thmlemse}{Theorem}{Theorems}

\crefname{thmse}{Theorem}{Theorems}
\newtheorem{thmlemr1t}[thm]{Theorem \mysc{\{LEM,C1,T\}}}
\crefname{thmlemr1t}{Theorem}{Theorems}
\newtheorem{thmac}[thm]{Theorem \mysc{\{AC\}}}
\crefname{thmac}{Theorem}{Theorems}
\newtheorem{thmccorlem}[thm]{Theorem \mysc{\{AC$_{\mathbb{N}}$ or LEM\}}}
\crefname{thmccorlem}{Theorem}{Theorems}
\newtheorem{thmr0}[thm]{Theorem \mysc{\{C0\}}}
\crefname{thmr0}{Theorem}{Theorems}
\newtheorem{thmlemr1}[thm]{Theorem \mysc{\{LEM,C1\}}}
\crefname{thmlemr1}{Theorem}{Theorems}
\newtheorem{thmr1}[thm]{Theorem \mysc{\{C1\}}}
\crefname{thmr1}{Theorem}{Theorems}
\newtheorem{thmr2}[thm]{Theorem \mysc{\{C2\}}}
\crefname{thmr2}{Theorem}{Theorems}
\newtheorem{thmlemr2}[thm]{Theorem \mysc{\{LEM,C2\}}}
\crefname{thmlemr2}{Theorem}{Theorems}
\newtheorem{thmr3}[thm]{Theorem \mysc{\{R$\flat$\}}}
\crefname{thmr3}{Theorem}{Theorems}
\newtheorem{thmuar3}[thm]{Theorem \mysc{\{UA,R$\flat$\}}}
\crefname{thmuar3}{Theorem}{Theorems}
\newtheorem{thmlemr3}[thm]{Theorem \mysc{\{LEM,R$\flat$\}}}
\crefname{thmlemr3}{Theorem}{Theorems}
\newtheorem{thmlemr3t}[thm]{Theorem \mysc{\{LEM,R$\flat$,T\}}}
\crefname{thmlemr3t}{Theorem}{Theorems}
\newtheorem{thmr3t}[thm]{Theorem \mysc{\{R$\flat$,T\}}}
\crefname{thmr3t}{Theorem}{Theorems}
\newtheorem{thmualemr3t}[thm]{Theorem \mysc{\{UA,LEM,R$\flat$,T\}}}
\crefname{thmualemr3t}{Theorem}{Theorems}

\crefname{thmuar3t}{Theorem}{Theorems}
\crefname{thmr4}{Theorem}{Theorems}
\newaliascnt{lem}{thm}
\newtheorem{lem}[lem]{Lemma}
\aliascntresetthe{lem}
\crefname{lem}{Lemma}{Lemmas}
\newtheorem{lemlem}[lem]{Lemma \mysc{\{LEM\}}}
\crefname{lemlem}{Lemma}{Lemmas}

\crefname{lemt}{Lemma}{Lemmas}
\newtheorem{lemr0}[lem]{Lemma \mysc{\{C0\}}}
\crefname{lemr0}{Lemma}{Lemmas}
\newtheorem{lemlemr1}[lem]{Lemma \mysc{\{LEM,C1\}}}
\crefname{lemlemr1}{Lemma}{Lemmas}
\newtheorem{lemlemr1t}[lem]{Lemma \mysc{\{LEM,C1,T\}}}
\crefname{lemlemr1t}{Lemma}{Lemmas}
\newtheorem{lemr1}[lem]{Lemma \mysc{\{C1\}}}
\crefname{lemr1}{Lemma}{Lemmas}
\newtheorem{lemr3}[lem]{Lemma \mysc{\{R$\flat$\}}}
\crefname{lemr3}{Lemma}{Lemmas}
\newtheorem{lemuar3}[lem]{Lemma \mysc{\{UA,R$\flat$\}}}
\crefname{lemuar3}{Lemma}{Lemmas}

\crefname{lemr3t}{Lemma}{Lemmas}
\newaliascnt{cor}{thm}
\newtheorem{cor}[cor]{Corollary}
\aliascntresetthe{cor}
\crefname{cor}{Corollary}{Corollaries}
\newtheorem{corac}[cor]{Corollary \mysc{\{AC\}}}
\crefname{corac}{Corollary}{Corollaries}
\newtheorem{corccorlem}[cor]{Corollary \mysc{\{AC$_{\mathbb{N}}$ or LEM\}}}
\crefname{corccorlem}{Corollary}{Corollaries}
\newtheorem{corr0}[cor]{Corollary \mysc{\{C0\}}}
\crefname{corr0}{Corollary}{Corollaries}
\newtheorem{corr1}[cor]{Corollary \mysc{\{C1\}}}
\crefname{corr1}{Corollary}{Corollaries}
\newtheorem{corccorlemr1}[cor]{Corollary \mysc{\{AC$_{\mathbb{N}}$ or LEM, C1\}}}
\crefname{corccorlemr1}{Corollary}{Corollaries}
\newtheorem{corr3}[cor]{Corollary \mysc{\{R$\flat$\}}}
\crefname{corr3}{Corollary}{Corollaries}

\crefname{corr3t}{Corollary}{Corollaries}

\crefname{coruar3}{Corollary}{Corollaries}
\newtheorem{corua}[cor]{Corollary \mysc{\{UA\}}}
\crefname{corua}{Corollary}{Corollaries}
\newtheorem{corlemr1}[cor]{Corollary \mysc{\{LEM,C1\}}}
\crefname{corlemr1}{Corollary}{Corollaries}
\newtheorem{corlemr2}[cor]{Corollary \mysc{\{LEM,C2\}}}
\crefname{corlemr2}{Corollary}{Corollaries}

\crefname{corlemr3}{Corollary}{Corollaries}
\newtheorem{corlemr3t}[cor]{Corollary \mysc{\{LEM,R$\flat$,T\}}}
\crefname{corlemr3t}{Corollary}{Corollaries}
\newtheorem{corualemr3t}[cor]{Corollary \mysc{\{UA,LEM,R$\flat$,T\}}}
\crefname{corualemr3t}{Corollary}{Corollaries}
\newtheorem{coracr1}[cor]{Corollary \mysc{\{AC,C1\}}}
\crefname{coracr1}{Corollary}{Corollaries}
\theoremstyle{definition}
\newaliascnt{rmk}{thm}
\newtheorem{rmk}[rmk]{Remark}
\aliascntresetthe{rmk}
\crefname{rmk}{Remark}{Remarks}
\newaliascnt{defn}{thm}
\newtheorem{defn}[defn]{Definition}
\aliascntresetthe{defn}
\crefname{defn}{Definition}{Definitions}
\newtheorem{defnr0}[defn]{Definition \mysc{\{C0\}}}
\crefname{defnr0}{Definition}{Definitions}

\newaliascnt{eg}{thm}
\newtheorem{eg}[eg]{Example}
\aliascntresetthe{eg}
\crefname{eg}{Example}{Examples}

\newtheorem*{namedthm}{\protect\namedthmname}
\newcounter{namedthm}
\makeatletter
\newenvironment{named}[1]
  {\def\namedthmname{#1}%
   \refstepcounter{namedthm}%
   \begin{namedthm}\def\@currentlabel{#1}}
  {\end{namedthm}}
\makeatother

\numberwithin{equation}{section}

\newenvironment{proof*}{\begin{proof}}{\end{proof}}

\makeatletter
\def\mscsname#1{#1}\def\ns{ }
\makeatother

\title[Real-cohesive HoTT]{Brouwer's fixed-point theorem in real-cohesive homotopy type theory}
\author[Michael Shulman]{\mscsname{Michael}\ns\mscsname{Shulman}}
  \thanks{This material is based on research sponsored by The United States Air Force Research Laboratory under agreement number FA9550-15-1-0053.  The U.S. Government is authorized to reproduce and distribute reprints for Governmental purposes notwithstanding any copyright notation thereon.  The views and conclusions contained herein are those of the author and should not be interpreted as necessarily representing the official policies or endorsements, either expressed or implied, of the United States Air Force Research Laboratory, the U.S. Government, or Carnegie Mellon University.}

\address{University of San Diego, San Diego, CA, 92110, USA}

\usepackage{perpage}
\MakePerPage{footnote}

\begin{document}

\begin{abstract}
  We combine Homotopy Type Theory with axiomatic cohesion, expressing the latter internally with a version of ``adjoint logic'' in which the discretization and codiscretization modalities are characterized using a judgmental formalism of ``crisp variables''.
  This yields type theories that we call ``spatial'' and ``cohesive'', in which the types can be viewed as having independent topological and homotopical structure.
  These type theories can then be used to study formally the process by which topology gives rise to homotopy theory (the ``fundamental $\infty$-groupoid'' or ``shape''), disentangling the ``identifications'' of Homotopy Type Theory from the ``continuous paths'' of topology.
  In a further refinement called ``real-cohesion'', the shape is determined by continuous maps from the real numbers, as in classical algebraic topology.
  This enables us to reproduce formally some of the classical applications of homotopy theory to topology.
  As an example, we prove Brouwer's fixed-point theorem.
\end{abstract}

\maketitle

\section{Introduction}
\label{sec:introduction}

\subsection*{On spaces, types, and \oo-groupoids}
\label{sec:spaces-types-oogpds}

Homotopy type theory~\parencite{hottbook} is an emerging field that connects homotopy theory and higher category theory with constructive type theory.
The homotopy-theoretic semantics for type theory~\parencite{aw:htpy-idtype,klv:ssetmodel} enables us to view the types in type theory as homotopical objects.
This motivates new rules and axioms for type theory, such as Voevodsky's univalence axiom and higher inductive types, which allow us to do \emph{synthetic homotopy theory}.
That is, we can construct homotopical objects in type theory and prove theorems about them there, such as calculating homotopy groups of spheres~\parencite[Chapter 8]{hottbook}.
The homotopy-theoretic semantics imply that such formal theorems automatically yield proofs of corresponding classical results.

However, it is important to understand that at least on the surface, a theorem in synthetic homotopy theory is a  different statement about different objects than a theorem of the same name in classical homotopy theory.
For instance, in the synthetic theorem that $\pi_1(\hocirc)=\Z$~\parencite{ls:pi1s1}:
\begin{itemize}
\item $\hocirc$ is a \emph{higher inductive type}, freely generated by a basepoint and a loop.
\item $\pi_1(X)$ is the 0-truncation of the loop space $\Omega(X)$, where the 0-truncation is a higher inductive type that ``kills all information above dimension 0''.
\item The loop space $\Omega(X)$ involves ``paths'' that are an essentially undefined primitive notion, like ``point'' and ``line'' in axiomatic geometry.
  They are given meaning by the rules governing them (namely, those of Martin-L\"of's intensional identity type).
\end{itemize}
This should be contrasted with the meanings of the same words in the theorem of classical homotopy theory that is denoted $\pi_1(\topcirc)=\Z$:
\begin{itemize}
\item $\topcirc$ is the topological space $\{ (x,y)\in \R\times\R \mid x^2+y^2=1 \}$.
\item $\pi_1(X)$ is the set of path-components of the loop space $\Omega(X)$.
\item The loop space $\Omega(X)$ (and the notion of path-component) involves paths that are \emph{defined} as continuous functions $[0,1] \to X$ out of the topological unit interval.
\end{itemize}

\begin{rmk}
  In contrast to \textcite{hottbook}, I will use $\topcirc$ for the topological circle $\{ (x,y)\in \R\times\R \mid x^2+y^2=1 \}$, hoping that the font will serve as a mnemonic for its relationship to the real numbers $\R$.
  This leaves $\hocirc$ for the higher inductive one.
\end{rmk}

The connection between these worlds that justifies using the same terminology in both cases can be described as follows.
\begin{enumerate}
\item In homotopy theory we study objects that may be called \emph{homotopy spaces} or \emph{\oo-groupoids}, which have objects, identifications between those objects, higher identifications between identifications, and so on.
\item A topological space $X$ gives rise to an \oo-groupoid, traditionally called its \emph{fundamental \oo-groupoid}, whose objects are the points of $X$, whose identifications are continuous paths in $X$, whose higher identifications are continuous homotopies in $X$, and so on.
\item Type theory admits semantics in \oo-groupoids.
\item Finally, the fundamental \oo-groupoid of the topological circle $\topcirc$ is the same \oo-groupoid that furnishes semantics for the higher inductive circle $\hocirc$ in type theory (and similarly for paths, truncations, and so on).
\end{enumerate}

A classical algebraic topologist would not bring up \oo-groupoids, of course, but just talk about topological spaces directly.
However, I find it clarifying to do so, because it enables us to distinguish ways in which topological spaces and \oo-groupoids behave differently.
Specifically, a topological space can be defined up to homeomorphism, whereas an \oo-groupoid is (at least the way I use the word) only ever defined up to homotopy equivalence.
For instance, as topological spaces, a cylinder is distinct from a M\"{o}bius strip (e.g.\ one is orientable and one is not); but they have the same fundamental \oo-groupoid (which is, in fact, also the same as that of \topcirc).

This is important because the homotopy-theoretic semantics of type theory lands in \oo-groupoids, and \emph{not} in topological spaces.%
\footnote{To be completely precise, at least with current technology~\parencite{klv:ssetmodel} it lands in simplicial sets, which are a different model for \oo-groupoids.
Like a topological space and unlike an \oo-groupoid, a simplicial set can be defined up to isomorphism rather than equivalence; but the notion of isomorphism for simplicial sets is totally different from the notion of isomorphism for topological spaces.
Moreover, at most a very small amount of this extra strictness of simplicial sets beyond \oo-groupoids is visible to type theory, and conjecturally none at all.}
Thus, inside of type theory we should think of types as \oo-groupoids and not as topological spaces.
In particular, although it is common in homotopy type theory to use terminology borrowed from topology such as ``path'' and ``circle'', these words have \textit{a priori} nothing to do with their topological versions \emph{which can also be defined inside of type theory}.
That is, since type theory is rich enough to encode all of mathematics, we can define in it the real numbers \R\ and thus also the topological circle $\topcirc = \{ (x,y):\R\times\R \mid x^2+y^2=1 \}$; but this type is quite different from the higher inductive circle $\hocirc$.
For instance, the former is a set (0-truncated) with infinitely many distinct points, while the latter is not a set and has ``only one point'' (technically, it is 0-connected).

This frequently causes confusion among newcomers to homotopy type theory, who struggle to understand the meaning of ``path'' because it both is, and is not, like the topological concept after which it is named.
For this reason among others, in this paper I will say \emph{identification} or \emph{equality} rather than ``path'' when speaking of the synthetic notion (i.e.\ elements of identity types).

More importantly, however, this means there is something missing from homotopy type theory, because there is more to algebraic topology than the study of \oo-groupoids in their own right.
The process by which they arise from topological spaces (the fundamental \oo-groupoid) is also important.
This often becomes clear when we are concerned with \emph{applications} of algebraic topology to fields such as geometry and physics, which have no \emph{intrinsic} interest in homotopy theory.
As a simple example, consider the standard homotopy-theoretic proof of Brouwer's fixed-point theorem:

\begin{thm}
  Let $\topdisc$ denote the topological disc $\{ (x,y) \in \R^2 \mid x^2 + y^2 \le 1 \}$.
  Then any continuous map $f:\topdisc \to \topdisc$ has a fixed point.
\end{thm}
\begin{proof}
  Suppose $f:\topdisc \to \topdisc$ is continuous with no fixed point.
  For any point $z\in\topdisc$, draw the ray from $f(z)$ to $z$ (which is well-defined since they are distinct) and keep going until you hit the boundary of $\topdisc$ (see \cref{fig:bfp}).
  \begin{figure}
    \centering
    \begin{tikzpicture}[>=stealth]
      \draw (0,0) circle (2);
      \node[circle,fill,inner sep=1pt,label=below:$f(z)$] (fz) at (-1,-.5) {};
      \node[circle,fill,inner sep=1pt,label=$r(z)$] (rz) at (0,2) {};
      \node[circle,fill,inner sep=1pt,label=left:$z$] (z) at ($(fz)!0.3!(rz)$) {};
      \draw[->] (fz) -- (rz);
    \end{tikzpicture}
    \caption{A retraction from a self-map}
    \label{fig:bfp}
  \end{figure}
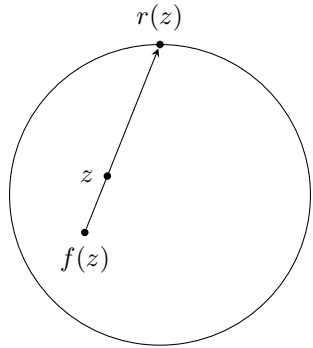
  Call that point of intersection $r(z)$; then $r$ defines a continuous map from $\topdisc$ to its boundary, which is the topological circle $\topcirc$.
  Moreover, if $z\in\topcirc$ to begin with, then there is no need to keep going once we get from $f(z)$ to $z$, so that $r(z)=z$ in this case.
  Thus, $\topcirc$ would be a retract of $\topdisc$.
  It follows that $\pi_1(\topcirc)$ would be a retract of $\pi_1(\topdisc)$; but the latter is trivial while the former is not.
\end{proof}

Note how crucial it is that $\topcirc$ is a topological space, with points that can be specified exactly and not just up to homotopy.
Indeed, the statement of the theorem makes no reference to homotopy theory.

With present technology, this theorem is completely inaccessible to homotopy type theory.
Of course, the use of proof by contradiction requires the law of excluded middle, but this can be assumed, as discussed by~\textcite[Chapter 3]{hottbook}.
The real problem is that inside of type theory, we have no way to relate $\topcirc$ to $\hocirc$, so that we cannot make the jump into homotopy theory in the last step.
(We could, of course, repeat the \emph{classical} definition of $\pi_1$ as a quotient of a set of continuous paths, and likewise the classical proof that $\pi_1(\topcirc)\cong\Z$, but then we would have gained nothing from homotopy type theory.)

What we need is some way to define the fundamental \oo-groupoid inside of type theory.
The most natural way to interpret this is that given a topological space $X$ in type theory, we want a \emph{type} whose \emph{synthetic} paths (i.e.\ elements of its identity type) are the \emph{topological} paths in $X$ (i.e.\ continuous maps $[0,1] \to X$), and so on.
It is natural to try to construct such a thing as a higher inductive type, but there are at least two problems with this.
Firstly, it seems that we would need infinitely many constructors to handle paths of all dimension, which at present we do not know how to do in the finitary syntax of type theory.
And secondly, if we succeed in doing it, it is not immediately clear how we would show that this construction takes $\topcirc$ to $\hocirc$ without essentially copying again the classical proof, thereby negating any benefit from synthetic homotopy theory.

\subsection*{Combining synthetic topology and synthetic homotopy theory}

In this paper we will take a different approach: \emph{we make the topology synthetic, as well as the homotopy theory}.
That is, in addition to carrying a structure of synthetic \oo-groupoid, with identifications and higher identifications, every type also carries a structure of \emph{synthetic topological space}, with some sort of ``cohesion'' attaching nearby elements together.
(I am being deliberately vague here; I don't want to demand that a synthetic topology is necessarily specified in terms of open sets the way a classical topology usually is.)

It is well known that type theory admits topological semantics, in which types are interpreted by topological spaces or other similar objects.
Indeed, this has been known for much longer than the homotopical interpretations.
Moreover, just as type theory can be enhanced with homotopically motivated axioms such as univalence, it can be enhanced with topologically motivated axioms such as Browerian continuity principles or an axiomatic ``Sierpinski space'', leading to \emph{synthetic topology}~\parencite{escardo:syntop-datatypes,escardo:top-hoil,taylor:lamcra,bl:met-syntop}.

What we will do is to \emph{combine} these two kinds of interpretation, so that each type will have \emph{both} an \oo-groupoid structure \emph{and} a topological structure.
In contrast to the situation in classical algebraic topology, where the same structure is viewed at different times either as topological or homotopical, in our world the topological and homotopical structures will be unrelated.
The objects of our motivating model will be \emph{topological \oo-groupoids}, which can be thought of as \oo-groupoids together with topologies on their points, identifications, higher identifications, and so on, making all of the operations continuous.
Here are some examples to get the reader's intuition up to speed.
\begin{itemize}
\item An ordinary topological space yields a topological \oo-groupoid with no nontrivial identifications, in the same way that a set yields an ordinary \oo-groupoid.
\item An ordinary \oo-groupoid can be given the discrete topology (in all dimensions, i.e.\ on points, identifications, and so on).
\item An ordinary \oo-groupoid can also be given the \emph{indiscrete} topology.
\item Recall that an ordinary group $G$ can be ``delooped'' to an \oo-groupoid $K(G,1)$, with one point $b$ and $G$ forming the identifications from $b$ to itself.  If $G$ is a \emph{topological} group, then we have an analogous $K(G,1)$ that is a \emph{topological} \oo-groupoid, but remembering the topology on $G$ as the space of identifications.%
  \footnote{This is different from what a classical algebraic topologist means by ``$BG$'' for a topological group $G$.  The latter is the \emph{ordinary} $\infty$-groupoid with one object whose identifications are points of $G$, whose 2-identifications are paths in $G$, etc.---i.e.\ the delooping of the fundamental \oo-groupoid of $G$.} %  For us that would be $K(\shape G,1)$.
\end{itemize}
Note that here it starts to matter even more that we distinguish between paths and equalities: a topological \oo-groupoid has both paths \emph{and} equalities; and it can also have equalities between paths, and paths between equalities.\footnote{Modulo dealing correctly with endpoints, equalities between paths and paths between equalities are actually the same, but are different from paths between paths and from equalities between equalities.}

At this point we are presented with several problems.
\begin{enumerate}
\item How should we define precisely what we mean by a ``topological \oo-groupoid'' (in set-theoretic foundations, to serve as semantics for type theory)?\label{item:q1}
\item What new rules or axioms can we bring into type theory, motivated by such a model, that will enable us to relate $\topcirc$ to $\hocirc$?
  Continuity principles and the Sierpinski space are not well-adapted to saying interesting things about types that are not 0-truncated.\label{item:q2}
\item The classical proof of Brouwer's fixed-point theorem uses the law of excluded middle.
  But unlike homotopical interpretations, topological interpretations are generally incompatible with excluded middle; how can we resolve this apparent contradiction?\label{item:q3}
\end{enumerate}
Fortunately, all three have the same solution.
Let us start with~\ref{item:q3}.
The problem with excluded middle in topological models is that it may not hold \emph{continuously}: e.g.\ a subspace $U\subseteq X$ and its complement $X\setminus U$ will together contain all the points of $X$, but the map from their coproduct $U + (X\setminus U)$ to $X$ does not have a \emph{continuous} section because the topology on the domain is different.
Thus, we cannot expect to have the full law of excluded middle in a topological model.

However, we can recover a form of excluded middle if we enhance our model and our type theory with the ability to talk about \emph{discontinuous} functions in addition to continuous ones.
Then it will be consistent to say something like ``discontinuously, for all $P$ we have $P\vee \neg P$''.
There are at least three ways to add discontinuous functions to a world where everything is continuous: we could add a new basic notion of ``discontinuous map'', we could add an operation that retopologizes a space discretely (since every function out of a discrete space is continuous), or we could dually add an operation that retopologizes a space indiscretely (since every function into an indiscrete space is continuous).
In fact, we will do all three; we will denote the two retopologization operations by $\flat$ and $\sharp$, respectively.
Categorically, $\flat$ is a coreflection and $\sharp$ is a reflection; in type theory we call them \emph{modalities}.

Now let us move on to question~\ref{item:q2}.
Here our answer comes from \textcite{lawvere:cohesion} \parencite[see also][]{lm:ac-cohesive}, who has proposed an axiomatization of ``categories of cohesive spaces'' based entirely on adjoint functors and their properties.
This is convenient because it is easy to generalize to homotopy theory (regarded as $(\oo,1)$-category theory), since we have well-behaved notions of adjoint \oo-functors.
Lawvere has considered many axioms, but the basic setup is a string of four adjoint functors:
\begin{center}
  \begin{tikzpicture}
    \node (S) at (-0.5,0) {sets};
    \node (C) at (-0.5,2) {cohesive spaces};
    \draw[->] (0,1.7) -- node[auto] {$p_*$} (0,0.3);
    \draw[->] (1,0.3) -- node[auto,swap] {$p^!$} (1,1.7);
    \draw[->] (-1,0.3) -- node[auto,swap] {$p^*$} (-1,1.7);
    \draw[->] (-2,1.7) -- node[auto] {$p_!$} (-2,0.3);
    \node at (.7,1) {$\dashv$};
    \node at (-.3,1) {$\dashv$};
    \node at (-1.3,1) {$\dashv$};
  \end{tikzpicture}
\end{center}
in which $p^*$ (hence also $p^!$) is fully faithful.
The idea is that $p_*$ finds the underlying set of a space, $p^*$ equips a set with the discrete topology, and $p^!$ equips it with the indiscrete topology.
(We will return to the meaning of $p_!$ below.)
It follows that $p^* p_*$ is a coreflection into the subcategory of discrete spaces while $p^! p_*$ is a reflection into the subcategory of indiscrete spaces; thus they coincide with the modalities $\flat$ and $\sharp$ mentioned above.
We can \oo-ize this whole setup by simply replacing categories with $(\oo,1)$-categories, ``sets'' with ``\oo-groupoids'', and adjoint functors with adjoint \oo-functors; this has been studied extensively by \textcite{schreiber:dcct}.

Finally, there is a topos-theoretic side to Lawvere's work which informs our answer to question~\ref{item:q1}.
If the category of cohesive spaces is to be a topos, then the existence of this adjoint string says that it must be a \emph{local and locally connected} topos.
And we know how to construct local and locally connected toposes, by taking categories of sheaves on Grothendieck sites with special properties
\parencite[see e.g][C3.3.10 and C3.6.3(d)]{ptj:elephant}.
Thus, it stands to reason that by taking \oo-sheaves on similar sorts of sites, we should be able to construct \emph{cohesive $(\oo,1)$-toposes}, and this is in fact the case, as shown by \textcite{schreiber:dcct}.
The objects of the particular cohesive $(\oo,1)$-topos we are interested in are known as \emph{topological \oo-groupoids} or \emph{topological stacks}.

\subsection*{Real-cohesive homotopy type theory}

At this point we may appear to have wandered rather far from our original motivating problem, namely that we need an ``internal'' way to construct the fundamental \oo-groupoid.
In fact, however, we have snuck up behind it and are about ready to pounce.
The secret lies with Lawvere's fourth functor $p_!$.
In his 1-categorical context, it constructs the set of connected components of a space, since a map from any space to a discrete one must be constant on every component.
However, when we generalize to the $(\oo,1)$-categorical context, the analogous $p_!$ \emph{is the fundamental \oo-groupoid functor}!
Unfortunately, to explain how this comes about would take too much space here.
Some insight may be gained from the proof of \cref{thm:circ-circ} below; for more details see \textcite[Proposition 4.3.32]{schreiber:dcct} or \textcite[\S3]{carchedi:hotyorb}.

Just as $p^* p_*$ and $p^! p_*$ induce the operations $\flat$ and $\sharp$ on the type theory of cohesive spaces, the composite $p^* p_!$ induces a third modality, a reflection into \emph{discrete} spaces that we call the \emph{shape} and denote $\shape$.
(See \cref{rmk:shape} for discussion of the name ``shape''.)
With $\shape$ in our type theory, we can at least \emph{state} the desired relation between the circles: we should have $\shape\topcirc = \hocirc$.
The existence of $\shape$ alone does not suffice to \emph{prove} this, however; what we need is an additional axiom guaranteeing that $\shape$ is in fact constructed out of continuous paths indexed by intervals in \R, as we expect.
To be precise, we assert that $\shape$ is the ``internal localization'' at $\R$, as in~\textcite[Proposition 3.9.4]{schreiber:dcct} and~\textcite[Proposition 8.3]{dug:uht}.

At last, we have a type theory in which we can reproduce the classical proof of Brouwer's fixed-point theorem, using synthetic homotopy theory in the appropriate place.
We perform the topological part of the proof in type theory, using classical axioms such as our modified form of excluded middle where appropriate.
Then at the end we apply $\shape$, use our calculation that $\shape \topcirc = \hocirc$, and then appeal to the synthetic proof of $\pi_1(\hocirc)=\Z$ to reach a contradiction.

The type theory formed by the addition of $\shape$ to $\flat$ and $\sharp$ is called \emph{cohesive homotopy type theory}, and was already sketched by~\textcite{ss:qgftchtt}.
Our theory is a further enhancement of this which I call \emph{real-cohesive homotopy type theory}, due to the special role played by the real numbers in defining $\shape$ using continuous paths.
As we will see, the simple axiom of real-cohesion has many striking consequences, especially combined with our modified classical axioms.
In addition to Brouwer's fixed-point theorem, it implies versions of the Intermediate Value Theorem, and characterizes the internal functions $\R\to\R$ as being precisely those that are $\varepsilon$-$\delta$ continuous in the usual sense.
The present paper is a formal development of just enough real-cohesive homotopy type theory to enable the proof of the Brouwer fixed-point theorem, which we give in \cref{sec:bfp}.
%(Not everything in the paper is literally necessary for this proof, but I have tried to minimize digressions to keep the length less unmanageable.)

\subsection*{The problem of comonadic modalities}

It should be noted before proceeding that there are in fact many subtly different ways to represent the modalities $\shape$, $\flat$, and $\sharp$ in type theory.
The problems in trying to do this arise from the fact that everything in ordinary type theory happens in an arbitrary context, which means that all type-theoretic operations must correspond to category-theoretic operations that apply in each slice category and are stable under pullback.
For instance, the na\"\i{}ve type-theoretic notion of ``reflective subuniverse''~\parencite[Definition 7.7.1]{hottbook} corresponds categorically not to a mere reflective subcategory, but to a ``reflective subfibration'' of the self-indexing.

For $\sharp$ this is not a problem, since it is a left exact reflector (being a composite of two right adjoints), and any left-exact-reflective subcategory admits a canonical extension to a reflective subfibration.
For $\shape$, which is a reflector but not left exact, it is more of an issue; however, as long as $\shape$ can be defined by an internal localization --- which categorically means that our topos is not just locally connected but \emph{stably} locally connected \parencite[see][]{johnstone:punctual-lc} --- it can be extended to a reflective subfibration.
(The category-theoretic side of such modalities will be studied elsewhere; presently we are concerned with their type-theoretic manifestation, with the category theory providing only motivation.
Internally in type theory, modalities of this sort were sketched briefly by \textcite[\S7.7]{hottbook} and will be studied in detail by \textcite{rss:modalities}.)

The real problem arises with the modality $\flat$, which is a \emph{coreflector}, and \emph{cannot} be extended to a ``coreflective subfibration''.
In fact, one can prove a ``no-go theorem'' (\cref{thm:no-go}) that the only coreflective subfibrations are those of the form $\Box_U (A) = A\times U$ for a subterminal object $U$, and $\flat$ is certainly not of this form.
Thus, we must find some way to represent $\flat$ that prevents it from being applied in arbitrary contexts.

\textcite{ss:qgftchtt} did this by using $\sharp\type$.
We cannot have $\flat$ as an operation $\type\to\type$ on the universe itself, but we can have it as an operation $\sharp\type \to \sharp\type$, i.e.\ a ``discontinuous'' self-map of the universe.
This choice had the advantage of being formalizable in existing proof assistants, but the disadvantage that it required developing large amounts of theory of $\sharp\type$ as ``the external category of types'' \parencite[which was mostly omitted by][]{ss:qgftchtt}.
Even worse, there is no way to ``escape'' from $\sharp\type$, but it should be the case that if we have a particular type defined in the empty context, such as $\N$, we can define $\flat\N$ as an actual type, rather than just an element of $\sharp\type$.
Finally, in \textcite{ss:qgftchtt} all the modalities were obtained by asserting axioms, but for good formal behavior of a type theory it is preferable to use rules rather than axioms, since axioms interfere with canonicity.

For these reasons, in this paper we take a different approach, based on the ``judgmental reconstruction'' of modal logic by \textcite{pd:modal}.
We modify our base type theory by introducing a new sort of hypothesis $x::A$, which we call \emph{crisp}, intended to mark whether dependence on a variable is continuous or discontinuous.
Based on this, we can then introduce $\flat$ and $\sharp$ using the usual sort of rules for type constructors: formation, introduction, elimination, and computation.

Our rules for $\flat$ are almost exactly those of \textcite{pd:modal} for $\Box$.
If $B$ is a type depending only on crisp hypotheses we can form $\flat B$, and similarly if a term $M:B$ depends only on crisp hypotheses we can form $M^\flat : \flat B$, while an element of $\flat B$ can always be assumed to be of this form.
However, although our $\sharp$ on its own behaves like the $\bigcirc$ of \textcite{pd:modal}, its rules are quite different (in particular, we have no judgment ``$M ~: A$'' like in \textcite{pd:modal}) and make it automatically right adjoint to $\flat$.
The resulting type theory, which requires no axioms yet, we call \emph{spatial type theory}; we expect that it corresponds semantically to \emph{local toposes}, which have $p^*\dashv p_* \dashv p^!$ but no $p_!$.
It includes as a fragment the ``modal logic of local toposes'' of \textcite{ab:ax-local} and \textcite{abs:lrt-modal-comput-ea}.

In particular, the rules of $\flat$ stipulate that it can only be applied in a context of purely discontinuous dependence, circumventing the no-go theorem.
In addition to resolving all the above problems with $\sharp\type$, this has the additional advantage that we can apply $\flat$ not only in a totally empty context, but in a context of other \emph{discrete} spaces.
Categorically speaking, this means we treat the topos of cohesive spaces as \emph{indexed over} the base topos via $p^*$, which is exactly the right thing to do for any geometric morphism \parencites[see][B3.1.2]{ptj:elephant}[and][Theorem 3.4.20]{schreiber:dcct}.

\begin{rmk}\label{rmk:modes}
  One might argue that it would be better to include a separate kind of type representing objects in the base category, and type constructors representing the functors $p^*$, $p_*$, etc.\ rather than the monads and comonads they induce.
  This would be in line with the ``judgmental deconstruction'' of~\textcite{jr:modal} that decomposes $\Box$ and $\bigcirc$ into pairs of adjoint functors.
  It could be that this is indeed better; in fact, the current rules for $\flat$ and $\sharp$ were deduced from a generalization of \textcite{jr:modal} developed by \textcite{ls:1var-adjoint-logic}.

  However, for now I have chosen to go with a simpler type theory in which there is only one kind of type.
  Learning to keep track of the two kinds of variables in spatial type theory is hard enough.
  Moreover, the theory of \textcite{ls:1var-adjoint-logic} is general enough to include this case as well: unlike $\Box$ and $\bigcirc$, our modalities $\flat$ and $\sharp$ are type-theoretically sensible in their own right without needing to be deconstructed.
  There is also a philosophical attractiveness to the position that \emph{all} types have topological structure, even if that structure happens to be discrete or codiscrete.
\end{rmk}

\begin{rmk}\label{rmk:admissible}
  It is natural to wonder why we introduce crisp hypotheses $x::A$, when semantically they are essentially just ordinary hypotheses $u:\flat A$.
  Why not simply restrict the formation of $\flat B$ and $M^\flat$ to contexts containing only variables of the form $u:\flat A$?

  This might be simpler in some respects, but it would be bizarre in others.
  The official type-theoretic reason is that, as noted by \textcite[\S4.2]{pd:modal}, ``substitution would no longer be an admissible rule''.
  In less fancy language, that means that we would not be able to determine the validity of an expression by plain syntactic analysis, or else we would have to introduce a complicated calculus of explicit substitutions.

  For instance, if $M^\flat$ could be formed in a context of variables $u:\flat A$, then we could define a function $f\defeq (\lam{x} x^\flat):\flat B \to \flat \flat B$, since a variable $x:\flat B$ has the correct form to allow us to write $x^\flat : \flat\flat B$.
  Now suppose we also had a function $g:A\to \flat B$.
  Then the composite $f \circ g$ would reduce to $\lam{a} g(a)^\flat$; but $g(a)$ does \emph{not} have the correct form to allow us to apply $(\blank)^\flat$ to it!
  So if we allowed $g(a)^\flat$ as a valid expression, we would be unable to determine its validity by inspection, but would have to ``guess'' that it was obtained by substituting $g(a)$ into the valid $x^\flat$.
  Otherwise, we would have to write instead $x^\flat[g(a)/x]$, where the substitution $[g(a)/x]$ is not, as usual, a \emph{defined operation} on expressions, but a basic part of the grammar of expressions (an ``explicit substitution'').

  The formalism of crisp variables avoids both horns of this dilemma.
  In this setup, $f:\flat B \to \flat\flat B$ must be defined as $\lam{x} \flet u^\flat := x in u^\flat{}^\flat$ (see \cref{sec:flat}), and when we compose it with $g$ we get $\lam{a} \flet u^\flat := g(a) in u^\flat{}^\flat$.
  Now we are no longer applying $(\blank)^\flat$ directly to $g(a)$, but to a crisp variable that it has been destructed into; thus the \emph{syntactic} invariant that indicates a correct application of $(\blank)^\flat$ is preserved.
  % (Note that, as in \cref{thm:discrete-coreflective-0}, the uniqueness principle does not reduce this to $g(a)^\flat$ precisely because it doesn't make sense.)

  With all that said, once the basic lemmas about $\flat$ are proven using crisp variables, in most cases we will be able to blur the distinction between crisp elements of $A$ and elements of $\flat A$, at least when working informally.
\end{rmk}

Now we have to choose how to represent $\shape$.
One possibility would be to add further new judgment forms so that it could be characterized similarly to $\flat$ and $\sharp$.
The general theory of \textcite{ls:1var-adjoint-logic} immediately suggests how to do this, but the judgmental structure of the resulting theory would be substantially more complicated.
Moreover, semantically it seems likely to correspond to indexing the base topos over the cohesive topos by means of $p_!$, which is at least an odd thing to do.
We would also need to take special care that the rules didn't force $\shape$ to be left exact (whereas our rules for $\flat$ and $\sharp$ do automatically make both of them left exact), since in our desired models it is not.

For these reasons, I have chosen to introduce $\shape$ axiomatically.
However, rather than asserting the modality itself as an axiom, we assert a set of ``generators'' for it, enabling the modality itself to be constructed by a higher inductive localization.
This has some technical advantages, % --- e.g.\ it is freer of universe-polymorphism worries ---
and it also leads naturally into the stronger axiom of real-cohesion, where we can simply assert that \R\ alone is a sufficient generator for $\shape$.
It seems highly unlikely that the special position of \R\ in real-cohesion could be obtained by rules rather than an axiom, so we lose nothing in that regard by making $\shape$ axiomatic already.

It should be clear from this discussion, however, that there is nothing sacred about the particular choices made in this paper.
Most of the length of this paper is devoted to \emph{proving} from our rules and axioms that $\shape$, $\flat$, and $\sharp$ satisfy the correct relationships (which~\textcite{ss:qgftchtt} just \emph{assumed} to be true).
Specifically, we need to know that $\sharp$ and $\shape$ are reflectors (\cref{thm:codiscrete-reflective,thm:discrete-reflective}) and $\flat$ is a coreflector (\cref{thm:discrete-coreflective}), that $\sharp$ is left exact (\cref{thm:sharp-pullback}), that $\shape\dashv\flat\dashv\sharp$ (\cref{thm:flat-sharp-adj,thm:shape-flat-adj}), that $\shape$ preserves finite products % (\cref{thm:shape-prod})
and reflects into the same subcategory that $\flat$ coreflects into, that the images of $\flat$ and $\sharp$ are equivalent (\cref{thm:co-disc-equiv}) by an equivalence that identifies $\flat$ with $\sharp$ (\cref{thm:co-disc-equiv-gamma}), and that $\shape$ is generated by $\R$ (\ref{ax:r3}).
Once these facts are established somehow, the contentful theorems (such as $\shape\topcirc = \hocirc$ and the Brouwer fixed-point theorem) are essentially independent of the underlying type-theoretic machinery.
I had to fix some particular choice of that machinery in order to write the paper, but this choice made may very well turn out not to be the optimal one.

Moreover, the behavior of ``spatial type theory'' as a type theory has not yet been studied.
It looks somewhat reasonable, in that $\sharp$ and $\flat$ have introduction and elimination rules of roughly the usual sort.
But for it to be really respectable as a type theory one should prove various standard theorems about it, such as normalization and canonicity, and this has not been done.
(\textcite{ls:1var-adjoint-logic} did this for a more general class of ``adjoint logic'' type theories, but only under the fairly severe simplifying assumptions of no dependent types and only one variable in the context.
\textcite{lsr:multi} allows multi-variable contexts, but still no dependent types, and has a rather fiddlier adequacy theorem.)
Eventually, one might hope to implement such a type theory in a proof assistant.

\subsection*{Outline of the paper}

This paper has three parts; but rather than coming in sequential order, one of the parts is interleaved through the other two.
The first part, consisting of \crefrange{sec:spatial-type-theory}{sec:topos-models}, develops spatial type theory.
It begins with the judgmental structure (\cref{sec:spatial-type-theory}), then moves on to $\sharp$ (\cref{sec:codisc}) and then $\flat$ (\crefrange{sec:flat}{sec:disc}).
In \cref{sec:topos-models} we sketch briefly the intended categorical models of this theory, and use them to motivate the axioms to be introduced in the third part.

The second, interleaved, part, consisting of \cref{sec:lem,sec:lem-2,sec:axiom-choice,sec:axiom-choice-2,sec:ac-2}, studies classical axioms that can be added to spatial type theory without destroying its topological content, such as the modified law of excluded middle mentioned above.
It is interleaved because the formulation and proof of various classical principles requires more and more of the base type theory, but at the same time provides useful motivation for the same.
In \cref{sec:lem} we state the restricted law of excluded middle, and decide that in order to state a similarly restricted axiom of choice we need $\sharp$.
After introducing $\sharp$, in \cref{sec:lem-2} we revisit excluded middle in some other guises made possible by $\sharp$, and in \cref{sec:axiom-choice} we formulate the axiom of choice, but note that there ought to be simpler versions available if we had $\flat$.
Thus, after introducing $\flat$, in \cref{sec:axiom-choice-2} we discuss these versions and their applications.
Finally, after stating the axioms of real-cohesion, in \cref{sec:ac-2} we show that these axioms strengthen our classicality axioms, and state one further such axiom.

The third part, consisting of \cref{sec:real-cohesion,sec:shape,sec:pieces-have-points,sec:continuity,sec:bfp}, adds additional axioms to spatial type theory that make it cohesive and real-cohesive, and applies them to prove the Brouwer fixed-point theorem.
In \cref{sec:real-cohesion} we state the cohesion and real-cohesion axioms in generator form and deduce some consequences, and in \cref{sec:shape} we construct $\shape$ as a higher inductive type.
\Cref{sec:pieces-have-points} relates our axioms back to the motivating topos-theoretic notions.
In \cref{sec:continuity} we note some implications of real-cohesion for synthetic topology that don't require any higher homotopy, including some versions of the Intermediate Value Theorem and a version of ``Brouwer's theorem'' that all functions $\R\to\R$ are continuous.
Finally, in \cref{sec:bfp} we bring everything together to prove the Brouwer fixed-point theorem by mimicking the classical proof, invoking synthetic homotopy theory and our modified classicality axioms in the appropriate places.
We also similarly prove a constructive variant asserting only the existence of \emph{approximate} fixed points, but not requiring any classicality axioms.

\subsection*{On axioms and notation}

Because we are developing a hierarchy of type theories (spatial $<$ cohesive $<$ real-cohesive), and using classical axioms in some places, we keep careful track of which axioms are necessary for which theorems.
In addition to our classicality axioms (\mysc{LEM}, \mysc{AC}, and \mysc{T}) and (real-)cohesion axioms (\mysc{C0}, \mysc{C1}, \mysc{C2}, and \mysc{R$\flat$}), we track uses of the full univalence axiom (\mysc{UA}).
We do need \mysc{UA} for our proof of the Brouwer fixed-point theorem, but most of the theory does not require it, and hence should be valid in a 1-topos model in addition to an $(\oo,1)$-topos model.
We will, however, use without comment the axioms of function extensionality, propositional resizing, and univalence restricted to propositions (``propositional extensionality''), since these are all valid in a 1-topos.

We will generally adhere to the notation of \textcite{hottbook}.
For example, we write $x=y$ for the identity type, $M\jdeq N$ for a judgmental equality, and $a\defeq P$ if $a$ is currently being defined to equal $P$.
However, as remarked above, we will call elements of $x=y$ ``identifications'' or ``equalities'' rather than ``paths'', and we write $\hocirc$ for the higher inductive circle to reserve $\topcirc$ for the topological one.
We will also call $(-1)$-truncated types simply \emph{propositions} rather than ``mere propositions'', as is common outside of \textcite{hottbook}.
Similarly, we pronounce the $(-1)$-truncated existential (as in \textcite{hottbook}, $\exists_{x:A} P(x)$ means $\brck{\sm{x:A}P(x)}$) simply as ``there exists'', rather than the ``there merely exists'' of \textcite{hottbook}.
This also requires avoiding other terminology such as ``propositional equality'' and ``propositional uniqueness rule'' since the notions in question are not $(-1)$-truncated; for this purpose we introduce the new adjective \emph{typal} (to contrast with \emph{judgmental}).

\subsection*{Vistas}

The Brouwer fixed-point theorem is not itself of central importance; it serves mainly as a convenient test case.
The real point is the development of spatial, cohesive, and real-cohesive type theory.
Some possibilities for further applications include the following.
\begin{itemize}
\item Real-cohesion should also imply other applications of homotopy theory to topology, such as the Lefschetz fixed-point theorem and the hairy ball theorem.
\item In cohesive homotopy type theory, a topological group appears as an ordinary internal (0-truncated) group, which can be delooped as in \textcite{lf:emspaces}.
  These deloopings are classifying spaces for topological principal bundles up to isomorphism (rather than the up-to-homotopy-equivalence classifying spaces of classical algebraic topology); the classifying maps live in ``continuous cohomology theories''.
\item Matrix groups and their deloopings are important in classical algebraic topology, but as \oo-groupoids they are hard to define in ordinary homotopy type theory.
  In real-cohesive homotopy type theory, we can obtain them as shapes of ordinary set-level definitions.
  The same is true for other classical objects such as higher Hopf fibrations.
\item There are other local toposes than cohesive and real-cohesive ones, such as Sierpinski cones, the topological topos of \textcite{ptj:topological-topos}, and
  % For instance, any geometric morphism can be factored through a local one, which encodes the original geometric morphism by the modalities on it.
  % Thus, spatial type theory could be used as a type-theoretic language for arbitrary geometric morphisms (although there are probably better ways to achieve that, along the lines of \cref{rmk:modes}).
  the relative realizability topos of \textcite{abs:lrt-modal-comput-ea}.
  Each of these should motivate different axioms.
\item Similarly, there are cohesive toposes other than real-cohesive ones.
  Some, which play an important role in \textcite{schreiber:dcct}, encode \emph{smoothness} rather than continuity;
  these satisfy a version of \ref{ax:r3} with $\R$ replaced by a type of ``smooth reals''.
  Other interesting examples include the \emph{global homotopy theory} of \textcite{rezk:global-cohesion}, the \oo-topos of simplicial \oo-groupoids, and the tangent \oo-toposes of \textcite{lurie:ha}.
\end{itemize}

\subsection*{On theft and honest toil}

One of the advantages of type theory is that it enables us to ``add structure by failing to rule it out''.
For instance, because type theory admits topological models, if we define (say) a \emph{group} in type theory, as long as we didn't use any principles that are invalid topologically, we have automatically defined a \emph{topological group} as well.
Homotopy type theory applies the same effect to \oo-groupoids: as long as we don't use any principles like UIP, everything we do in type theory is automatically homotopical.%
\footnote{We do have to be careful that this happens in the right way.
  For instance, the na\"\i{}ve definition of ``group'' doesn't give the correct \oo-groupoidal notion, but there is a different one that does.}
This is especially valuable because working directly with \oo-groupoids can be combinatorially complicated and require a lot of background in algebraic topology and category theory.

Spatial and cohesive type theory extend this advantage even further: not only does type theory apply automatically to spaces and to \oo-groupoids, it also applies to \emph{spatial \oo-groupoids}, which have both structures at once.
Moreover, even if we didn't have any categorical models in mind to specify what a ``spatial \oo-groupoid'' might mean, we could still undertake to study them in type theory by simply \emph{combining} the axioms and principles that pertain to spatial models (such as continuity principles and/or modalities) and homotopical models (such as univalence and HITs).
Except for \cref{sec:topos-models,sec:pieces-have-points}, the present paper can be read from such a perspective.

It is true that mathematical honesty may demand that such models eventually be produced, to ensure relative consistency and allow our ``synthetic'' results to be translated into classical theorems.
However, once \emph{someone} has undertaken that toil \parencite[in our case, the toil includes][as well as future work that remains to be done]{lurie:higher-topoi,schreiber:dcct,gk:univlcc,shulman:intext-codisc,shulman:rsfib-factsys,klv:ssetmodel,lw:localuniv,ls:hits}, we can continue to work in the type theory without needing to think about or understand the model.
% Russell's famous association fallacy notwithstanding, there is no guilt in such an approach.
This principle should also apply to other possible enhancements of spatial type theory.
%, including minor variations such as smooth cohesion and infinitesimal cohesion, but also more major ones that incorporate other structures such as computability or global equivariance.

\subsection*{Acknowledgments}

This paper owes a lot to many people.
My understanding of cohesive $(\oo,1)$-toposes developed over the course of many discussions with Urs Schreiber, who also helped greatly with understanding the categorical semantics of modalities (this joint work is represented in \textcite{ss:qgftchtt}, which this paper draws heavily on).
Other participants at the $n$Forum and $n$-Category Caf\'e have also been very helpful, particularly Zhen Lin.
Andrej Bauer and Mart\'\i{}n Escard\'o patiently explained to me the relationship between different classicality axioms and some existing theories of synthetic topology, and Bas Spitters contributed the proof of \cref{thm:axt-fails}.
Urs Schreiber, Bas Spitters, Egbert Rijke, and Mart\'in Escard\'o also gave useful feedback on drafts.
Finally, Dan Licata explained \textcite{pd:modal} to me, did most of the work on our joint paper~\parencite*{ls:1var-adjoint-logic} generalizing \textcite{jr:modal} that led to the spatial type theory presented here, and has been generally indispensable in teaching me to sound at least vaguely like a type theorist.

% \part{Spatial type theory}
% \label{sec:stt}

\section{Judgments of spatial type theory}
\label{sec:spatial-type-theory}

\addtocontents{toc}{\protect\setcounter{tocdepth}{2}}

As explained in the introduction, we take the point of view that \emph{all types have spatial structure, independently of their \oo-groupoid structure}, and concomitantly \emph{all ordinary type-theoretic constructions are continuous}.
% Moreover, we take the stronger position that \emph{this spatial structure is the only source of non-classicality}.
% That is, the \emph{only} reason classical principles such as LEM and AC fail is because they fail to hold \emph{continuously}.
Now, it does happen sometimes in mathematics that we want to talk about things that are \emph{not} continuous.
Our main target in this paper --- the Brouwer fixed-point theorem --- is a prime example: fixed points of functions cannot be selected \emph{continuously} with respect to the space of functions.
For this reason, we augment our type theory with technology that enables us to talk about possibly-discontinuous constructions as well as continuous ones.

Recall that the basic judgment of type theory is $a:A$ for some specific type $A$, expressing that $a$ is a point of the type $A$.
This judgment is used in two ways: for some specific \emph{expression} $a$ we can judge it to be true, or for some \emph{variable} $x$ we can \emph{suppose} it to be true.
For instance, in a statement like ``whenever $x:A$ then $f(x):B$'' we see both uses: $x$ is a variable hypothesized to be of type $A$, and under this hypothesis, $f(x)$ is an expression judged to have type $B$.
It is a distinguishing feature of type theory that whenever $x:A$ is a \emph{hypothesis}, $x$ must be a \emph{variable}; it makes no sense to say ``suppose that $f(x):B$'' (in contrast to the behavior of the membership predicate $\in$ of set theory).

In a world where types are spaces, all ordinary type-theoretic constructions are continuous; thus a statement like ``whenever $x:A$ then $f(x):B$'' means that $f(x)$ depends \emph{continuously} on $x$.
We now augment our type theory with new features to express the possibility that this dependence may \emph{not} be continuous.
In general, we might imagine doing this either by modifying the hypothesis $x:A$ or the conclusion $f(x):B$.
Since continuous functions are also discontinuous,\footnote{We will use ``discontinuous'' to mean ``not necessarily continuous'', by analogy with other phrases such as ``noncommutative ring''.} the variance of implication means that if we modify $x:A$, we would have to replace it by a \emph{stronger} judgment (i.e.\ one that implies $x:A$), whereas if we modify the conclusion $f(x):B$, we would have to replace it by a \emph{weaker} judgment (i.e.\ one that is implied by $f(x):B$).

We take the first route.
Following \textcite{pd:modal}, we denote the stronger judgment by
\[x::A\]
When we hypothesize $x::A$, this means that we allow ourselves to perform constructions and proofs using $x$ that may not respect the topology of $A$ (but they still must respect the topology of any other ordinary hypotheses $y:C$).
In other words, \emph{having} $x::A$ is \emph{more} powerful than having $x:A$; and thus, oppositely, \emph{proving} something with a hypothesis of $x::A$ says \emph{less} than proving it with a hypothesis $x:A$.

We refer to $x::A$ as a \textbf{crisp} hypothesis, and $x$ as a crisp variable.%
\footnote{The word ``crisp'' is chosen as having a connotation somewhat similar to ``discrete'' (i.e.\ not flabby or deformable) and being relatively free of other mathematical meanings.
Its one existing mathematical usage that I am aware of is in fuzzy-set theory, where it refers to ordinary (non-fuzzy) sets; this is actually closely related to our usage, since many categories of fuzzy sets are quasitoposes \parencite[Chapter 8]{wyler:quasitopoi} in which the crisp sets are what we would call the codiscrete objects, so that they also ``forget about the cohesion''.}
(\textcite{pd:modal}, whose focus is on logic, use the word \emph{valid}.)
A crisp variable can always be used as an ordinary one, but also (potentially) in other ways.
If we don't want to give a name to the variable, we may say that $A$ \textbf{holds crisply}.
If necessary, for contrast we will refer to an ordinary hypothesis or variable $x:A$ as \textbf{cohesive}.

For consistency, we require that all variables appearing in the type of a crisp hypothesis are themselves crisp.
Moreover, a crisp variable can only be substituted by expressions involving only other crisp variables.
If $M:A$ is an expression other than a variable, we will say that $M$ is crisp, and sometimes write $M::A$, if the only \emph{variables} it contains are crisp.
Thus, as for \textcite{pd:modal} the notion of \emph{crisp conclusion} is not a basic part of the theory, but rather is defined to mean an ordinary conclusion that depends only upon crisp hypotheses.

All the ordinary rules of type theory ($\prod$-types, $\sum$-types, $=$-types, $W$-types, HITs) are imported into our theory \emph{only in the world of cohesive variables}.
Applications of these rules can involve dependence on ``unaffected'' crisp variables; but the variables that are manipulated in the rules, as well as their conclusions, are always cohesive.
(We will eventually prove, however, that in some cases this can be worked around.)

At a formal level, what we are doing is separating our \emph{context} into two pieces: first a crisp one, then a cohesive one, maintaining the restriction that the type of each variable can only contain variables occurring to its left.
We would therefore write our judgments as $\Delta\mid \Gamma \types \mathcal{J}$.
The previous paragraph means that the ordinary rules of type theory leave the crisp context $\Delta$ untouched; for example, the rules for $\prod$-types are shown in \cref{fig:prod}.
\begin{figure}
  \centering
  \begin{mathpar}
  \inferrule{\Delta\mid\Gamma \types A:\type \\ \Delta\mid\Gamma,x:A\types B:\type}
    {\Delta\mid\Gamma\types{\tprd{x:A}B}:\type}
\and
  \inferrule
  {\Delta\mid\Gamma,x:A\types {b}:{B}}
  {\Delta\mid\Gamma\types{\lam{x} b}:{\tprd{x:A} B}}
\and
  \inferrule
  {\Delta\mid\Gamma\types {f}:{\tprd{x:A} B} \\ \Delta\mid\Gamma\types{a}:{A}}
  {\Delta\mid\Gamma\types{f(a)}:{B[a/x]}}
\and
  \inferrule
  {\Delta\mid\Gamma,x:A\types{b}:{B} \\ \Delta\mid\Gamma\types{a}:{A}}
  {\Delta\mid\Gamma\types {(\lam{x} b)(a)} \jdeq {b[a/x]}}
\and
  \inferrule
  {\Delta\mid\Gamma\types {f}:{\tprd{x:A} B}}
  {\Delta\mid\Gamma\types {f}\jdeq {(\lam{x}f(x))}}
  \end{mathpar}
  \caption{$\prod$-types in spatial type theory}
  \label{fig:prod}
\end{figure}
In what follows, we will attempt to describe all the rules of spatial type theory in words, but we will also include some formal syntax of this sort for precision.
After introducing the rules, however, we will reason almost entirely informally in the style of \textcite{hottbook}.

We will have occasion to use various \emph{axioms} (including axioms from \textcite{hottbook} such as function extensionality and univalence), so it is worth noting how these interact with our split contexts.
Ordinarily in type theory an \emph{axiom} means an assumed element of some type.
There are two ways of thinking about such axioms in type theory: they can be premise-free rules added to the theory, or they can be simply additional assumptions in the context that are never discharged.
According to the first approach, an axiomatic $a:A$ in our theory would be given by the rule
\[ \inferrule{\ }{\Delta\mid\Gamma\types a:A} \]
Note that because this rule is valid even if $\Gamma$ is empty, it implies unavoidably that the term $a:A$ is \emph{crisp}.
The second approach is more flexible: we can add assumptions to either the crisp context $\Delta$ or the cohesive one $\Gamma$.
However, \emph{in this paper, whenever we speak of an axiom, it is to be understood that it is crisp}.

We will also consider some ``axioms'' that are not stated as assumed elements of a fixed type, e.g.\ because they quantify over crisp variables.
These can, however, always be formulated as ``unjustified rules''; an example will occur momentarily.
Moreover, once we introduce $\flat$, all such axioms can be reformulated in the usual way (modulo potential issues with universes, which we ignore).

\subsection{The law of excluded middle}
\label{sec:lem}

With the tool of crisp hypotheses in hand, we can express a version of the law of excluded middle that does not contradict topological models.

\begin{named}{Axiom LEM}(Crisp excluded middle)\label{ax:lem}
  For any crisp $P::\prop$, we have $P \vee \neg P$.
\end{named}

\begin{rmk}\label{rmk:entailment}
  In ordinary mathematics, we blur the distinction between a \emph{hypothetical conclusion} (or \emph{entailment}) and an \emph{implication} (or \emph{function}).
  That is, when we write ``if $x:A$ then $f(x):B(x)$'', we might mean that we have constructed $f(x):B(x)$ under the hypothesis of $x:A$, or we might mean that we have an element $f:\prd{x:A} B(x)$.
  These two statements are intimately connected by the introduction and elimination rules for $\prod$, so no confusion results from failing to distinguish between them.

  However, when the hypothesis is crisp, we have as yet no type-former corresponding to $\prod$: we cannot abstract over a crisp variable.
  Thus, statements such as crisp excluded middle \emph{must} be read as hypothetical conclusions.
  In type-theoretic language, we have to read \ref{ax:lem} as an ``unjustified rule'':
  \[ \inferrule{\Delta\mid\cdot\types P:\prop}{\Delta\mid\cdot \types \mathsf{lem}_P : P\vee \neg P} \]
  rather than an assumed element of some fixed type.
  (We write ``$\Delta\mid\cdot$'' to indicate that the cohesive context is empty.)

  Fortunately, in \cref{sec:codisc,sec:flat} we will introduce new type formers that will allow us to mostly go back to ignoring this distinction.
  See \cref{rmk:admissible} for an explanation of why we do this in two steps.
\end{rmk}

To understand why this law of excluded middle is sensible, we need to think a bit about what a \emph{predicate} $P:A\to\prop$ looks like topologically.
In the classical category of topological spaces, the monomorphisms are the continuous injections; they need \emph{not} necessarily be subspace inclusions.
Therefore, given a subobject (i.e.\ a mono) $m:[P]\hookrightarrow A$, there is a difference between saying that every \emph{point} of $A$ is in $[P]$ (which is to say that the injection $m$ is actually a bijection) and saying that $m$ is a homeomorphism.
We regard the judgments ``$P(x)$ for all $x::A$'' and ``$P(x)$ for all $x:A$'' as expressing this same dichotomy in our type theory.
The latter says that $P(x)$ holds for all $x$ \emph{continuously} as a function of $x$, so that our monomorphism $m$ must have a section, and hence be a homeomorphism.

Now if $P:A\to\prop$, its pointwise negation $\neg P:A\to \prop$ is, essentially by definition, the \emph{largest subobject of $A$ disjoint from $P$}.
Topologically, this means the corresponding mono $[\neg P] \hookrightarrow A$ should contain exactly those points of $A$ that are not in $[P]$,  but also that $[\neg P] \hookrightarrow A$ \emph{must be a subspace inclusion} --- for the inclusion of the subspace determined by those points is certainly disjoint from $[P]$ and hence must be contained in $[\neg P]$ as a subobject of $A$.

It should be clear now that the ordinary LEM cannot hold, for it would assert that for any mono $m:[P]\hookrightarrow A$, the space $A$ is the disjoint union of the space $[P]$ and the subspace determined by its complement.
This is not true in general even if $[P]$ is a subspace inclusion.
However, since $[P]$ and $[\neg P]$ together do contain all the \emph{points} of $A$, it is sensible to assert that for any $x:A$ we have $P(x) \vee \neg P(x)$ \emph{crisply}, and this is what the crisp law of excluded middle gives us.

\begin{rmk}
  It is natural to wonder how we can express internally in type theory the property that a subobject is a subspace inclusion.
  The above analysis suggests that the double negation $\neg\neg P : A\to \prop$ should be the subspace containing the same points as $P$, so that $P$ would itself be a subspace if and only if $\prd{x:A} \neg\neg P(x)\to P(x)$.
  In \cref{sec:codisc} we will introduce another way to express this property, and in \cref{sec:lem-2} we will prove (using our axioms) that they are equivalent.
\end{rmk}

\begin{rmk}
  The assertion that every \emph{subspace} has a complementary subspace (such that their disjoint union is the whole space) translates to $\prd{P:\prop} \neg P \vee \neg\neg P$.
  Thus, the latter assertion, which is sometimes called \emph{de Morgan's law} (since it is equivalent to the one direction of de Morgan's laws that is not constructively valid, $\neg (P\wedge Q) \to (\neg P \vee \neg Q)$), is not acceptable for us either.
\end{rmk}

We can try to perform a similar analysis of the axiom of choice.
There is one subtlety: we must realize that the ``surjections'' of type theory correspond topologically to \emph{regular epimorphisms} --- which is to say, quotient maps --- rather than arbitrary epimorphisms.
One way to see that this must be so is to recall that any surjection between sets in type theory can be proven to be a quotient of its domain by some equivalence relation, which is true of regular epis of spaces but not ordinary ones.
Another way is to recall that in type theory we have a (surjection, embedding) factorization system, and topologically the left class corresponding to the monomorphisms is the quotient maps.

Now supposing that $Y \to X$ is a quotient map, of course there may not be any continuous section of it.
However, if the only source of non-classicality is topology, it should have a \emph{discontinuous} section.
Thus we may hope to be able to formulate this principle in our type theory.
However, if we inspect the usual formulation of the axiom of choice from \textcite[Theorem 3.2.2]{hottbook}:
\[ \left(\prd{x:X} \brck{Y(x)} \right) \to \brck{\prd{x:X} Y(x)} \]
it is not clear how to do this with our current tools.
What we want is for the section $f:\prd{x:X} Y(x)$ on the right to be discontinuous.
But at present we can only indicate discontinuity by a conclusion $f(x):Y(x)$ with a crisp hypothesis $x::X$, and we cannot apply a $\brck{\blank}$ \emph{outside} such a judgment.
(The outer $\brck{\blank}$ is, of course, necessary for consistency with univalence.)

In the next section, therefore, we introduce a new type former $\sharp$ that ``internalizes'' such judgments, allowing us to solve this problem (and many others).

\section{The $\sharp$ modality and codiscreteness}
\label{sec:codisc}
\addtocontents{toc}{\protect\setcounter{tocdepth}{1}}

Topologically, the function-type $\prd{x:X} Y(x)$ represents the space of \emph{continuous} functions; but as we saw in the last section, we sometimes want to internalize the notion of a discontinuous function.
Since every function out of a discrete topological space is continuous, and every function into a codiscrete space is continuous, to make a discontinuous function continuous all we have to do is retopologize its domain discretely or its codomain codiscretely.
We represent these ``retopologizing'' constructions by \emph{modalities} called $\flat$ and $\sharp$, which ``reify'' judgments involving crisp hypotheses.

% \subsection{The rules for $\sharp$}
% \label{sec:sharp-modality}

In this section we study $\sharp$ first, as it is somewhat simpler and stands on its own better.
Described in English, the rules are as follows; a type-theoretic presentation is shown in \cref{fig:sharp}.
\begin{itemize}
\item For any $A:\type$ there is a type $\sharp A$.
\item If $a: A$, then $a^\sharp : \sharp A$.
\item All crisp variables appearing in $a$ can become cohesive ones in $a^\sharp$.
  In other words, when we write $(\blank)^\sharp$, inside the $(\blank)$ we are free to use all variables introduced outside of it as if they were crisp.
  In yet other words, when we ``hit a $(\blank)^\sharp$'' while parsing an expression, all variables currently in the context become crisp.
  This expresses the fact that all functions into a codiscrete space are continuous.
\item In the same way, all crisp variables appearing in a type $A$ can become cohesive ones in $\sharp A$.
  This expresses the fact that the space of codiscrete types is itself codiscrete (which may not be obvious, but is true; we will explain it topologically in \cref{rmk:codiscrete-universe}).
\item From any \emph{crisp} element $x::\sharp A$, we can extract an element $x_\sharp : A$.
  Of course, since $x$ is crisp, $x_\sharp$ is also crisp.
\item A computation rule, which says $(a^\sharp)_\sharp \jdeq a$.
  This requires $(a^\sharp)_\sharp$ to be well-typed, which means that only crisp variables in the ambient context can occur in $a$.
  (The above rule says that when writing $a^\sharp$, cohesive variables in the ambient context can be used as crisp \emph{inside} $a$; but such variables remain cohesive in the term $a^\sharp$, potentially preventing us from writing $(a^\sharp)_\sharp$.)
\item A uniqueness rule, which says $(a_\sharp)^\sharp \jdeq a$.
  This likewise requires both sides to be well-typed, but now the potential problem lies with the right-hand side.\footnote{Note that if we view this rule as an ``expansion'' $a \to_\eta (a_\sharp)^\sharp$ rather than a reduction, as is common for uniqueness rules, then it shares with the computation rule the property that it can be applied to any well-typed input $a:\sharp A$.}
  In writing $(a_\sharp)^\sharp$ the cohesive variables in the context might be treated as crisp inside $a_\sharp$ and hence inside $a$, so that $a$ itself might not be well-typed outside of $(\blank)^\sharp$.
\end{itemize}

\begin{figure}
  \centering
  \begin{mathpar}
    \inferrule{\Delta,\Gamma\mid\cdot \types A:\type}{\Delta\mid\Gamma \types \sharp A : \type}\and
    \inferrule{\Delta,\Gamma\mid\cdot \types M:A}{\Delta\mid\Gamma \types M^\sharp : \sharp A}\and
    \inferrule{\Delta\mid\cdot \types M:\sharp A}{\Delta\mid\Gamma \types M_\sharp : A}\and
    \inferrule{\Delta\mid\cdot \types M:A}{\Delta\mid\Gamma \types (M^\sharp)_\sharp \jdeq M : A}\and
    \inferrule{\Delta\mid\Gamma \types M:\sharp A}{\Delta\mid\Gamma \types (M_\sharp)^\sharp \jdeq M : \sharp A}
  \end{mathpar}
  \caption{Rules for $\sharp$}
  \label{fig:sharp}
\end{figure}

\begin{rmk}\label{thm:negative}
  $\sharp$ is a \emph{negative} type former.
  This means, roughly, that its introduction rule(s) are chosen to match its elimination rule(s) rather than vice versa.
  For instance, function types are negative because to introduce a function we use $\lambda$-abstraction, which essentially means we have to say what will happen when we apply the eliminator (application) on any input.
  In the case of $\sharp$, the elimination rule says that we can get to $A$ if we have a \emph{crisp} element of $\sharp A$; thus to introduce an element of $\sharp A$, it suffices to give an element of $A$ under the assumption that everything is crisp.

  Negative type formers tend to require (judgmental) uniqueness rules (whereas for positive ones, a typal uniqueness rule\footnote{Recall that we say ``typal'' to mean that a rule is witnessed by an inhabitant of the identity type rather than being a judgmental equality, avoiding the older term ``propositional'' since it may not be a ``proposition'' in the $(-1)$-truncated sense.} is usually provable).
  This holds true for $\sharp$.
\end{rmk}

As our first application, we can make $\sharp$ into a functor: given $f:A\to B$, we define $\sharp f : \sharp A \to \sharp B$ by
\[ \sharp f (u) \defeq f(u_\sharp)^\sharp. \]
Note how the type-checking works: inside the $(\blank)^\sharp$, the variable $u::\sharp A$ is crisp, so we can write $u_\sharp : A$ and apply $f$ to it.
We can check functoriality: if also $g:B\to C$, then
\begin{align*}
  \sharp g (\sharp f(u)) &\jdeq g((f(u_\sharp)^\sharp)_\sharp)^\sharp\\
  &\jdeq g(f(u_\sharp))^\sharp\\
  &\jdeq \sharp (gf)(u).
\end{align*}
In addition, for any $A:\type$, we have a map $\sharpf : A \to \sharp A$, and we can check that this is natural:
\begin{align*}
  \sharp f(x^\sharp) &\jdeq f((x^\sharp)_\sharp)^\sharp\\
  &\jdeq f(x)^\sharp.
\end{align*}
We will often omit parentheses around iterated sub- and superscripted $\sharp$s (and, later, $\flat$s), maintaining only their order.
Thus, for instance, $g((f(u_\sharp)^\sharp)_\sharp)^\sharp$ becomes instead $g(f(u_\sharp)^\sharp{}_\sharp)^\sharp$, and $f((x^\sharp)_\sharp)^\sharp$ becomes $f(x^\sharp{}_\sharp)^\sharp$.
This avoids proliferation of parentheses.

% \subsection{Codiscreteness and $\sharp$ as a modality}
% \label{sec:sharp-as-modality}

Now, recall that our intent was that $\sharp A$ would be $A$ retopologized codiscretely.
Thus, a type should be called \emph{codiscrete} just when this operation does nothing to it.

\begin{defn}
  A type $A:\type$ is \textbf{codiscrete} if $\sharpf : A \to \sharp A$ is an equivalence.
\end{defn}

The rules for $\sharp$ transfer over to theorems about codiscrete types.
For instance, we have:

\begin{thm}\label{thm:codicsrete-ind}
  If $P$ is codiscrete, then when constructing an element of $P$ we may assume that all variables in the context are crisp.
  % We refer to this principle as \textbf{codiscrete induction}.
\end{thm}
\begin{proof}
  Since $P$ is codiscrete, we have an inverse $r : \sharp P \to P$.
  Thus, it suffices to construct an element of $\sharp P$.
  But now we can apply $\sharpf$, which allows us to assume all variables are crisp.
\end{proof}

\begin{thm}\label{thm:sharp-ind}
  If $P : \sharp A \to \type$ is such that each $P(v)$ is codiscrete, and we have $f:\prd{x:A} P(x^\sharp)$, then we have $g:\prd{v:\sharp A} P(v)$ such that $g(x^\sharp) = f(x)$ for all $x:A$.
  We refer to this principle as \textbf{$\sharp$-induction}.
\end{thm}
\begin{proof}
  Since each $P(v)$ is codiscrete, we have inverses $r_v : \sharp P(v) \to P(v)$.
  Thus, to construct $g$ it will suffice to construct for each $v:\sharp A$ an element of $\sharp P(v)$, and for this it suffices to construct an element of $P(v)$ assuming a crisp $v::\sharp A$.
  Now of course we have $v_\sharp : A$, and hence $f(v_\sharp) : P(v_\sharp{}^\sharp) \jdeq P(v)$, as desired.
  In symbols,
  \[ g(v) \defeq r_v(f(v_\sharp)^\sharp). \]
  Finally, we have
  \[ g(x^\sharp) \jdeq r_v(f(x^\sharp{}_\sharp)^\sharp) \jdeq r_v(f(x)^\sharp) = f(x) \]
  using the computation rule for $\sharp$ and the fact that $r_v$ is an inverse of $\sharpf$.
\end{proof}

We now show that the codiscrete types form a \emph{reflective subuniverse} in the sense of \textcite[\S7.7]{hottbook}.

\begin{thm}\label{thm:codiscrete-sharp}
  For any $A:\type$, the type $\sharp A$ is codiscrete.
\end{thm}
\begin{proof}
  Given $z:\sharp\sharp A$, to define an element of $\sharp A$ it suffices to define an element of $A$ assuming $z$ is crisp.
  But then we have $z_\sharp{}_\sharp : A$.
  In other words,
  \[ (\lam{z} z_\sharp{}_\sharp{}^\sharp) : \sharp\sharp A \to \sharp A. \]
  Note that we cannot reduce $z_\sharp{}_\sharp{}^\sharp$ to $z_\sharp$ using the uniqueness rule, since $z_\sharp$ is not well-typed for an arbitrary (cohesive) $z$.

  We now show that this map is inverse to $\sharpf:\sharp A \to \sharp\sharp A$.
  In one direction, if $z$ is $y^\sharp$ for $y:\sharp A$, we have
  \[ y^\sharp{}_\sharp{}_\sharp{}^\sharp \jdeq y_\sharp{}^\sharp \jdeq y \]
  using the computation rule followed by the uniqueness rule.
  Note that we could not do this computation in the other order, because $y^\sharp{}_\sharp$ is not well-typed on its own: $y$ is not crisp, so neither is $y^\sharp$.
  In the other direction, for any $z:\sharp\sharp A$ we have
  \[ z_\sharp{}_\sharp{}^\sharp{}^\sharp \jdeq z_\sharp{}^\sharp \jdeq z \]
  using the uniqueness rule twice;
  the first use is valid because $z_\sharp$ \emph{is} well-typed inside the outer $\sharpf$.
\end{proof}

\begin{thm}\label{thm:codiscrete-reflective}
  If $B:\type$ is codiscrete, then for any $A:\type$, precomposition with $\sharpf : A \to \sharp A$ is an equivalence
  \[ (\sharp A \to B) \simeq (A\to B). \]
  % Thus, the codiscrete types are a reflective subuniverse.
  More generally, if $B:\sharp A\to \type$ has each $B(v)$ codiscrete, then precomposition with $\sharpf$ is an equivalence
  \[ \left(\tprd{v:\sharp A} B(v)\right) \simeq \left(\tprd{x:A} B(x^\sharp)\right). \]
\end{thm}
\begin{proof}
  We prove the more general statement.
  Since each $B(v)$ is codiscrete, postcomposition with $\sharpf:B(v)\to \sharp B(v)$ is an equivalence.
  Thus, by the 2-out-of-3 property of equivalences, it will suffice to show that the precomposition map
  \[ \left(\tprd{v:\sharp A} \sharp B(v)\right) \to \left(\tprd{x:A} \sharp B(x^\sharp)\right) \]
  is an equivalence.
  Now in the opposite direction, given $f:\prd{x:A} \sharp B(x^\sharp)$, we can construct $g:\prd{v:\sharp A} \sharp B(v)$ by
  \[ g(v) \defeq f(v_\sharp)_\sharp{}^\sharp. \]
  Note that this has the form of the left-hand-side of the uniqueness rule, but that rule doesn't apply: ``$f(v_\sharp)$'' would not be well-typed since $v$ is not a crisp variable.
  On the other hand, it has the correct type $\sharp B(v)$ \emph{because} the uniqueness rule says $v_\sharp{}^\sharp \jdeq v$ since $v$ is crisp inside the $\sharpf$, hence $f(v_\sharp) : \sharp B(v_\sharp{}^\sharp)\jdeq \sharp B(v)$.

  If we precompose $g$ with $\sharpf : A \to \sharp A$, we get
  \begin{equation*}
    g(y^\sharp) \jdeq f(y^\sharp{}_\sharp)_\sharp{}^\sharp
    \jdeq f(y)_\sharp{}^\sharp
    \jdeq f(y)
  \end{equation*}
  where now we have been able to apply the uniqueness rule since $f(y)$ is well-typed on its own.
  On the other hand, if we start with $h:\sharp A \to \sharp B$, precompose it with $\sharpf$, and then extend the result back to $\sharp A$ as above, we get the function sending $v:\sharp A$ to
  \begin{equation*}
    h(v_\sharp{}^\sharp)_\sharp{}^\sharp
    \jdeq h(v)_\sharp{}^\sharp
    \jdeq h(v)
  \end{equation*}
  using the uniqueness rule twice.
  Applying function extensionality on both sides completes the proof.
\end{proof}

Combining this with \cref{thm:sharp-ind}, by \textcite[Theorem 7.7.4]{hottbook} we see that the codiscrete types are actually a \emph{modality} in the sense of \textcite[\S7.7]{hottbook}.
However, in the present paper we are using the word ``modality'' more generally to include coreflectors such as $\flat$ in addition to reflectors such as $\sharp$, so we will refer to modalities in the sense of \textcite[\S7.7]{hottbook} as \textbf{monadic modalities}.

The fact that $\sharp$ is a monadic modality formally implies many useful consequences.
In particular, essentially all properties of the $n$-truncation from \textcite[Chapter 7]{hottbook} that don't refer to more than one value of $n$ are true for all monadic modalities, and hence in particular for $\sharp$.
Monadic modalities will be studied further by \textcite{rss:modalities}; here we list some of the main results, specialized to $\sharp$.

\begin{enumerate}
\item A type $A$ is codiscrete if and only if $\sharpf: A \to \sharp A$ admits a retraction.
\item If $A,B,C$ are codiscrete, and we have $x,y:A$ and functions $f:B\to A$ and $g:C\to A$, then the following types are also codiscrete:
  \begin{mathpar}
    \unit \and
    A\times B\and
    x =_A y\and
    \mathsf{fib}_f(x) \and
    B\times_A C \and
    A\simeq B\and
    \textsf{is-}n\textsf{-type}(A)
  \end{mathpar}
\item If $A$ is any type and $P:A\to\type$ is such that each $P(x)$ is codiscrete, then $\tprd{x:A} P(x)$ is codiscrete.
  If in addition $A$ is codiscrete, then $\tsm{x:A} P(x)$ is also codiscrete.
\item For any $A,B$, the canonical map $\sharp(A\times B) \to \sharp A \times \sharp B$ is an equivalence.
\item If $A$ is a \mprop{}, so is $\sharp A$.
\end{enumerate}

% \subsection{$\sharp$ is left exact}
% \label{sec:sharp-left-exact}

In fact, $\sharp$ is even a \emph{left exact} monadic modality, i.e.\ it preserves pullbacks.
To show this, we begin with an ``encode-decode'' characterization of the identity types of $\sharp A$.

\begin{thm}\label{thm:path-sharp}
  For any $x,y:A$, we have an equivalence $(x^\sharp = y^\sharp) \simeq \sharp(x=y)$ such that the following triangle commutes:
  \begin{equation}\label{eq:path-sharp-tri}
    \vcenter{\xymatrix@R=1pc{
        & (x^\sharp = y^\sharp) \ar[dd]^{\simeq}\\
        (x=y) \ar[ur]^{\ap_{\sharpf}} \ar[dr]_{\sharpf} \\
        & \sharp(x=y)
      }}
  \end{equation}
\end{thm}
\begin{proof*}
  We define $\code : \sharp A \to \sharp A \to \type$ by
  \[ \code(u,v) \defeq \sharp (u_\sharp = v_\sharp). \]
  Here we use the fact that variables may be assumed crisp inside $\sharp(\blank)$ as well as inside $\sharpf$.
  Note that for $x,y:A$ we have $\code(x^\sharp,y^\sharp) \jdeq \sharp(x=y)$.

  We also have $\mathsf{r} : \prd{u:\sharp A} \code(u,u)$.
  For by \cref{thm:codiscrete-sharp} $\code(u,u)$ is codiscrete; thus by $\sharp$-induction the goal reduces to $\prd{x:A} \sharp(x=x)$, which is inhabited by $\lam{x} \refl_x{}^\sharp$.

  Now we define
  \[ \encode : \prd{u,v:\sharp A} (u=v) \to \code(u,v) \]
  in the usual way, $\encode(u,v,p) \defeq p_* (\mathsf{r}(u))$.
  To define
  \[ \decode : \prd{u,v:\sharp A} \code(u,v) \to (u=v) \]
  we proceed as follows.
  Since $\sharp A$ is codiscrete, so is $u=v$, and hence so are $\code(u,v) \to (u=v)$ and $\prd{v:\sharp A} \code(u,v) \to (u=v)$.
  Thus we can apply $\sharp$-induction twice, reducing the goal to $\prd{x,y:A} \code(x^\sharp,y^\sharp) \to (x^\sharp = y^\sharp)$.
  But $\code(x^\sharp,y^\sharp) \jdeq \sharp(x=y)$, so we can use $\sharp$-induction again followed by $\ap_{\sharpf}$.

  It now suffices to show that $\encode \circ \decode$ is the identity, since for each $u$ that will exhibit $\sm{v:\sharp A} \code(u,v)$ as a retract of the contractible $\sm{v:\sharp A} (u=v)$, hence itself contractible.
  Now the goal $\encode(u,v,\decode(u,v,c)) = c$ is an equality in $\code(u,v)$, which is codiscrete by \cref{thm:codiscrete-sharp}.
  Thus, the goal is also codiscrete, so we can do $\sharp$-induction again, replacing $u$ and $v$ by $x^\sharp$ and $y^\sharp$ respectively, and then $c:\sharp(x=y)$ by $p^\sharp$ for $p:x=y$, so that the goal becomes $(\ap_{\sharpf}(p))_*((\refl_{x})^\sharp) = p^\sharp$.
  Finally, an Id-induction on $p$ completes the proof of equivalence.

  To see that~\eqref{eq:path-sharp-tri} commutes, we do Id-induction on $p$ again, then compute
  \[\encode(x^\sharp,x^\sharp,\ap_{\sharp}(\refl_x)) \jdeq \encode(x^\sharp,x^\sharp,\refl_{x^\sharp}) \jdeq (\refl_{x^\sharp})_*(\mathsf{r}(x^\sharp)) \jdeq \mathsf{r}(x^\sharp) = \refl_x{}^\sharp.\qedhere \]
\end{proof*}

We remark in passing that this implies that $\sharp$ is ``functorial on homotopies'' as well.
Given $f,g:A\to B$ and $H:f\sim g$, i.e.\ $H:\prd{x:A} f(x)=g(x)$, for any $u:\sharp A$ we have $H(u_\sharp)^\sharp : \sharp(f(u_\sharp) = g(u_\sharp))$, which by \cref{thm:path-sharp} can be decoded to give $f(u_\sharp)^\sharp = g(u_\sharp)^\sharp$, i.e.\ $\sharp f(u) = \sharp g(u)$.
We could also go on to construct higher-dimensional aspects of a ``coherent \oo-functor'' structure on $\sharp$.

Returning to left-exactness of $\sharp$, this is actually implied formally by the conclusion of \cref{thm:path-sharp} \parencite[see][]{rss:modalities,hottcoq}.
However, using the special properties of $\sharp$ we can give a more direct proof, beginning with a sense in which $\sharp$ ``preserves $\sum$s''.

\begin{lem}\label{thm:sharp-sigma}
  % $\sharp$ preserves $\sum$-types, in the following sense:
  For any $A:\type$ and $B:A\to \type$, we have
  \[ \sharp\left(\sm{x:A} B(x)\right) \;\simeq\; \sm{u:\sharp A} \sharp B(u_\sharp). \]
\end{lem}
(The right-hand side makes sense because inside $\sharp(\blank)$ we may use $u$ as crisp.)
\begin{proof}
  % By \cref{thm:codiscrete-sigma},
  The right-hand side is codiscrete, so we can construct a map from left to right by $\sharp$-induction, sending $(a,b)$ to $(a^\sharp,b^\sharp)$.
  (Note that if $u$ is $a^\sharp$, then $\sharp B(a^\sharp{}_\sharp) \jdeq \sharp B(a)$, so this is well-typed.)
  In the other direction it suffices to construct an element of $\sm{x:A} B(x)$ supposing crisp hypotheses $u::\sharp A$ and $v::\sharp B(u_\sharp)$;
  but then we have $(u_\sharp, v_\sharp) : \sm{x:A} B(x)$.
  Finally, since both sides are codiscrete, both round-trip composites can be compared to identities with further $\sharp$-inductions.
  % ; we leave the details to the reader.
\end{proof}

\begin{thm}\label{thm:sharp-pullback}
  $\sharp$ preserves pullbacks.
\end{thm}
\begin{proof}
  Given $f:A\to C$ and $g:B\to C$, recall that the pullback $A\times_C B$ is $\sm{x:A}{y:B} (f(x)=g(y))$.
  By \cref{thm:sharp-sigma}, we have
  \[ \sharp(A\times_C B) \;\simeq \; \sm{u:\sharp A}{v:\sharp B} \sharp (f(u_\sharp) = g(v_\sharp)). \]
  while
  \[ \sharp A \times_{\sharp C} \sharp B \;\simeq \; \sm{u:\sharp A}{v:\sharp B} (\sharp f(u) = \sharp g(v)) \]
  Thus, it will suffice to prove that for all $u:\sharp A$ and $v:\sharp B$ we have
  \[ \sharp (f(u_\sharp) = g(v_\sharp)) \;\simeq\; (\sharp f(u) = \sharp g(v)). \]
  But by definition of the functoriality of $\sharp$, the right-hand side is $f(u_\sharp)^\sharp = g(v_\sharp)^\sharp$.
  Thus the result follows from \cref{thm:path-sharp}.
\end{proof}

Left exact monadic modalities also have many nice features, for most of which we refer to \textcite{rss:modalities} (and which are formalized in \textcite{hottcoq}).
One particularly important one is the following:

\begin{cor}\label{thm:codiscrete-ntypes}
  $\sharp$ preserves $n$-types for all $n$.
\end{cor}
\begin{proof}
  We induct on $n$.
  The base case $n=-2$ (i.e.\ $\sharp\unit \simeq \unit$) is true for any monadic modality.
  Thus, suppose $\sharp$ preserves $n$-types for some $n$, let $A$ be an $(n+1)$-type, and let $u,v:\sharp A$.
  Then $(u=v)$ is codiscrete, hence so is the \mprop{} ``$(u=v)$ is an $n$-type'', so we can use $\sharp$-induction on $u$ and $v$.
  But by \cref{thm:path-sharp}, for $x,y:A$ we have $(x^\sharp = y^\sharp)\simeq \sharp(x=y)$, which is an $n$-type by the inductive hypothesis.
\end{proof}

\textcite{rss:modalities} and \textcite{hottcoq}
show that if a left-exact monadic modality is also \emph{accessible} in a technical sense (defined there; see also \cref{sec:shape}), then the universe of modal types is modal.
We can prove this for $\sharp$ without accessibility, because of our formation rule for $\sharp A$.
The marker \{\mysc{UA}\} means that we need the full univalence axiom (although of course for the subuniverse of \mprop{}s we would need only propositional extensionality).

\begin{thmua}\label{thm:codiscrete-universe}
  The universe $\codisc \defeq \tsm{A:\type} \iscodisc(A)$ of codiscrete types is codiscrete.
\end{thmua}
\begin{proof}
  Suppose $X : \sharp \codisc$; we want to construct an element of $\codisc$ (to define a retraction).
  The first component of it will be $\sharp$ of some type, and for the purpose of constructing that type we may assume $X$ is crisp.
  Thus, we have $X_\sharp : \codisc$, which we can destruct further as $(A,c)$ for $A:\type$ and $c:\iscodisc(A)$.
  Now we use $A$ to be the type required.
  Explicitly, from $X : \sharp \codisc$ we have constructed
  \[ \sharp(\mathsf{pr}_1(X_\sharp)) : \type \]
  By \cref{thm:codiscrete-sharp}, this type is codiscrete, so gives an element of $\codisc$.

  To show that this is a retraction, if $X$ is of the form $Y^\sharp$, then we have
  \[ \sharp(\mathsf{pr}_1(Y^\sharp{}_\sharp)) \jdeq \sharp(\mathsf{pr}_1(Y))\]
  which is equivalent to $\mathsf{pr}_1(Y)$, since by $\mathsf{pr}_2(Y)$ the latter is codiscrete.
  This completes the proof since codiscreteness is a \mprop{}.
\end{proof}

It follows that we can interpret the formal syntax of ordinary type theory (with only one kind of context) entirely into the codiscrete types, with $\codisc$ playing the role of \emph{the} universe, obtaining a model with $\prod$-types, $\sum$-types, a unit type, and identity types.
We can also obtain some (higher) inductive types in this model, at least up to homotopy \parencite[in the sense of][\S5.5]{hottbook}, simply by applying $\sharp$ to their ordinary versions; this includes binary sums, the empty type, the natural numbers, $n$-truncations, suspensions, and more general colimits.

This doesn't appear to work for general W-types, however; if $W$ is generated by a constructor $\mathsf{sup}:\prd{a:A} (B(a) \to W) \to W$, there isn't even any obvious way to define a function $\prd{a:A} (B(a) \to \sharp W) \to \sharp W$.
It may be tempting to write ``$\lam{a}{b} \mathsf{sup}(a,\lam{y} b(y)_\sharp )^\sharp$'', but this is ill-typed: although the outer $\sharpf$ makes all variables in the \emph{current} context crisp, the inner $\lambda y$ binds a \emph{new}, \emph{cohesive} variable $y$, so that $b(y)$ is not crisp and we cannot apply $(\blank)_\sharp$ to it.

We could work around this if $\sharp$ were accessible, by adding the appropriate localization constructors to any HIT.
Accessibility is a perfectly reasonable axiom, since it holds in all known examples.
But in this paper we have no need for it, and moreover we want the base spatial type theory of this part (as opposed to the real-cohesive extension of it in the next part) to be free of axioms.

As an example of translating type theory into the codiscrete universe, let us define the Dedekind real numbers codiscretely.
Since the types $\Z$ and $\Q$ are abstractly equivalent to $\N$, and $\sharp\N$ is the natural numbers in the codiscrete world, it follows that the codiscrete versions of $\Z$ and $\Q$ will be equivalent to $\sharp\Z$ and $\sharp\Q$ respectively.
Now if we expand out the usual definition of the Dedekind reals $\R$ into the basic type constructors, we obtain the type shown first in \cref{fig:R}.
\begin{figure}
  \centering
  \begin{align*}
  \R &\defeq \tsm{L,U:\Q \to \prop}
   \brck{\tsm{q:\Q} L(q)}\times \brck{\tsm{r:\Q} U(r)}\\
  &\land \left(\tprd{q:\Q} L(q) \leftrightarrow \brck{\tsm{r : \Q} (q < r) \times L(r)}\right)\\
  &\land \left(\tprd{r:\Q} U(r) \leftrightarrow \brck{\tsm{q : \Q} U(q) \times (q < r)}\right)\\
  &\land \left(\tprd{q:\Q} (L(q) \times U(q)) \to \emptyset\right)
  \land \left(\tprd{q,r:\Q} (q < r) \rightarrow \brck{L(q) + U(r)}\right)\\
  \R' &\defeq \tsm{L,U:\sharp\Q \to \codiscprop}
   \sharp\brck{\tsm{q:\sharp\Q} L(q)}\times \sharp\brck{\tsm{r:\sharp\Q} U(r)}\\
  &\land \left(\tprd{q:\sharp\Q} L(q) \leftrightarrow \sharp\brck{\tsm{r : \sharp\Q} (q <' r) \times L(r)}\right)\\
  &\land \left(\tprd{r:\sharp\Q} U(r) \leftrightarrow \sharp\brck{\tsm{q : \sharp\Q} U(q) \times (q <' r)}\right)\\
  &\land \left(\tprd{q:\sharp\Q} (L(q) \times U(q)) \to \sharp\emptyset \right)
  \land \left(\tprd{q,r:\sharp\Q} (q <' r) \rightarrow \sharp\brck{\sharp(L(q) + U(r))}\right)
\end{align*}
\caption{The Dedekind reals and their codiscrete version}
\label{fig:R}
\end{figure}
Thus, their codiscrete version should be the type $\R'$ shown second in \cref{fig:R}.
In contrast to $\Z$ and $\Q$, it is not at all obvious that $\R'$ has anything to do with $\sharp\R$; in the former all intermediate constructions are discontinuous, whereas in the latter we do everything continuously and then at the very end forget the resulting topology.

We will see later on that they \emph{are} in fact equivalent (\cref{thm:sharp-R}), but this is a very special property of $\R$.
In most cases, the ``codiscrete version'' of a type $A$ is very different from $\sharp A$.
For instance, $\sharp(\R\to\R)$ should be the set of \emph{continuous} functions $\R\to\R$ topologized codiscretely, whereas $\R'\to\R'$ is the set of \emph{discontinuous} functions.

\begin{rmk}\label{rmk:sharp-overall}
  The overall conclusion of this section so far can be summarized as ``$\sharp$ is a left-exact monadic modality'' (plus \cref{thm:codiscrete-universe}, which could be obtained from accessibility).
  This is a purely cohesive statement, i.e.\ it can be stated without any need for crisp variables.
  In \textcite{ss:qgftchtt}, using ordinary type theory without crispness, we took this statement as the axiomatic definition of $\sharp$.
  The benefit of instead deriving it from rules involving crisp variables is that, as we will see in \cref{sec:flat}, $\sharp$ then automatically has the correct relationship to $\flat$.
\end{rmk}

\addtocontents{toc}{\protect\setcounter{tocdepth}{2}}
\subsection{Excluded middle revisited}
\label{sec:lem-2}

With $\sharp$ in hand, we can reformulate the crisp law of excluded middle (\ref{ax:lem}) as a single term (modulo universe polymorphism).

\begin{lemlem}[The sharp law of excluded middle]\label{thm:sharp-lem}
  $\prd{P:\prop} \sharp(P\vee\neg P)$.
\end{lemlem}
\begin{proof}
  Assume $P:\prop$.
  By the introduction rule for $\sharp$, it suffices to prove $P\vee\neg P$ assuming a crisp $P::\prop$.
  But this is exactly \ref{ax:lem}.
\end{proof}

Let us now return to the topological viewpoint, where codiscrete types have the codiscrete topology.
More generally, if $P:A\to \type$ and each $P(x)$ is codiscrete, then $\sm{x:A} P(x)$ must have the ``final topology'' induced by the first projection, i.e.\ the coarsest possible topology such that the first projection is continuous.
To see this, note that sections $\prd{x:A} P(x)$ are equivalent to sections $\prd{x:A} \sharp P(x)$, and hence to conclusions $f(x) : P(x)$ under crisp hypothesis $x::A$; but by our topological gloss, these are supposed to be simply discontinuous sections.
Thus, the sections $\prd{x:A} P(x)$ take no account of any topology on $P$, so it must be that $P$ has the final topology.

\begin{rmk}\label{rmk:codiscrete-universe}
  This topological gloss explains why the universe of codiscrete types must itself be codiscrete.
  A map into $A$ whose domain has the final topology is uniquely specified by a map into the underlying set (or \oo-groupoid) of $A$, i.e.\ it takes no account of the topology of $A$.
  Thus, a family of codiscrete types over $A$ should be the same as such a family over $\sharp A$, which is to say that any map $P:A\to\codisc$ must factor through $\sharp A$; thus $\codisc$ must be codiscrete.
\end{rmk}

In particular, to say that a predicate $P:A\to \prop$ is pointwise codiscrete must mean that it is a \emph{subspace}, with the induced topology.
We have already seen in \cref{sec:lem} a different type-theoretic property that ought also to characterize subspaces, namely that $P$ is $\neg\neg$-closed.
It turns out that one simple assumption beyond \ref{ax:lem} guarantees that these two properties coincide.
\parencite[See also][Corollary 4.5.]{lm:ac-cohesive}

\begin{named}{Axiom $\sharp\emptyset$}\label{ax:dense}
  $\emptyset$ is codiscrete.
\end{named}

\begin{thmlemse}\label{thm:codiscrete-notnot}\ 
  \begin{enumerate}
  \item A \mprop{} $P$ is codiscrete if and only if $\neg\neg P \to P$.\label{item:snn1}
  \item For any \mprop{} $P$ we have $\sharp P \simeq \neg\neg P$.\label{item:snn2}
  \item For any type $A$ we have $\sharp\brck{A} \simeq \neg\neg A$.\label{item:snn3}
  \end{enumerate}
\end{thmlemse}
\begin{proof}
  We prove~\ref{item:snn1} first.
  Since $P\to \neg\neg P$ always, we have $\neg\neg P \to P$ if and only if $\neg\neg P = P$.
  Now since $\emptyset$ is codiscrete, % by \cref{thm:codiscrete-pi},
  so is $\neg\neg P$ for any $P$. % is always codiscrete.
  Thus if $\neg\neg P = P$, then $P$ is also codiscrete.

  Conversely, suppose $P$ is codiscrete, and suppose that $\neg\neg P$; we must show $P$.
  By crisp excluded middle, we have $\sharp (P\vee \neg P)$.
  Since our goal $P$ is codiscrete, by $\sharp$-induction we have $P\vee \neg P$.
  But since $\neg\neg P$, it must be that $P$.
  (Note that each of the two assumptions \ref{ax:lem} and \ref{ax:dense} is used in exactly one direction of this equivalence.)

  Now~\ref{item:snn2} follows because $\sharp$ and $\neg\neg$ both reflect a \mprop{} into the sub-poset of codiscrete \mprop{}s.
  Likewise,~\ref{item:snn3} follows because $\sharp\brck{\blank}$ and $\neg\neg$ both reflect an arbitrary type into that same poset (since $\neg\neg A$ is always a \mprop{}).
\end{proof}

In particular, in this case the statement of \cref{thm:sharp-lem} is equivalent to
\[ \prd{P:\prop} \neg\neg(P\vee\neg P) \]
which is a provable statement in type theory.
Of course, we used \ref{ax:lem} in proving \cref{thm:codiscrete-notnot}, so this doesn't in any way absolve us from having to assume \ref{ax:lem}.
But it is nice to see that the cohesive version of excluded middle implied by our \ref{ax:lem} is just the one that is already provable in constructive type theory.

\subsection{The axiom of choice}
\label{sec:axiom-choice}

We can now revisit the axiom of choice.
First, however, let us consider all the ways of saying that a family $P:A\to\type$ is ``inhabited'', of which we now have \emph{four}.
Here is a list of them with their intended topological meanings.
\begin{enumerate}
\item $\prd{x:A} P(x)$ says that $\proj_1 : (\sm{x:A} P(x)) \to A$ has a continuous section.\label{item:inh1}
\item $\prd{x:A} \sharp P(x)$ says that this projection has a \emph{discontinuous} section.\label{item:inh2}
\item $\prd{x:A} \brck{P(x)}$ says that this projection is a quotient map.\label{item:inh3}
\item $\prd{x:A} \sharp\brck{P(x)}$ says that this projection is surjective on points.\\
  (Assuming \ref{ax:lem}, this is equivalent to $\prd{x:A} \neg\neg{P(x)}$.)\label{item:inh4}
\end{enumerate}
These topological glosses are mainly for our intuition, but in some cases we can prove precise versions of them in type theory.
For instance, statement~\ref{item:inh1} makes sense internally and is obviously true.
Regarding~\ref{item:inh2}, we can prove:

\begin{lem}\label{thm:sharp-section}
  $\prd{x:A} \sharp P(x)$ if and only if $\sharp\proj_1 : \sharp(\sm{x:A} P(x)) \to \sharp A$ has a section.
\end{lem}
\begin{proof}
  By $\sharp$-induction $\prd{x:A} \sharp P(x)$ is equivalent to $\prd{x:\sharp A} \sharp P(x_\sharp)$.
  This, in turn, is equivalent to saying that $\proj_1 : (\sm{x:\sharp A} \sharp P(x_\sharp)) \to \sharp A$ has a section.
  But by \cref{thm:sharp-sigma}, $\sharp(\sm{x:A} P(x)) \simeq \sm{x:\sharp A} \sharp P(x_\sharp)$, and it is easy to see that the projections match up.
\end{proof}

As for~\ref{item:inh3}, if $A$ and $P$ are sets we can consider it justified by~\cite[Theorem 10.1.5]{hottbook} (surjections of sets are regular epimorphisms).
The general case requires a more refined notion of ``quotient map'', but is also true when suitably formulated \parencite[see][]{kraus:nonrec-hit,vandoorn:proptrunc-nonrec,rijke:join}.

Finally, regarding~\ref{item:inh4} we can say the following.

\begin{lem}\label{thm:sharp-surj}
  $\prd{x:A} \sharp\brck{P(x)}$ if and only if $\sharp\proj_1 : \sharp(\sm{x:A} P(x)) \to \sharp A$ is ``codiscretely surjective'', i.e.\ for all $x:\sharp A$ we have $\sharp\brck{\fib_{\sharp\proj_1}(x)}$.
\end{lem}

Note that ``codiscretely surjective'' is what we get by interpreting ``surjective'' in the universe of codiscrete types, as discussed at the end of \cref{sec:codisc}.
Thus, it is the appropriate way to say that a map of codiscrete types ``is surjective'' in the untopologized world.

\begin{proof}
  By $\sharp$-induction $\prd{x:A} \sharp \brck{P(x)}$ is equivalent to $\prd{x:\sharp A} \sharp \brck{P(x_\sharp)}$.
  Thus, it will suffice to show that for any $x:\sharp A$ we have $\sharp\brck{P(x_\sharp)} \leftrightarrow \sharp\brck{\fib_{\sharp\proj_1}(x)}$.

  Suppose first that $q:\sharp\brck{P(x_\sharp)}$; we will show $\sharp\brck{\fib_{\sharp\proj_1}(x)}$.
  We apply $\sharpf$, making our goal $\brck{\fib_{\sharp\proj_1}(x)}$ and our hypotheses crisp, allowing us to form $q_\sharp : \brck{P(x_\sharp)}$.
  Now by functoriality of $\brck{\blank}$, it suffices to assume $p:P(x_\sharp)$ and prove $\fib_{\sharp\proj_1}(x)$.
  But then $(x_\sharp,p)^\sharp:\sharp(\sm{x:A} P(x))$, and
  \[ \sharp\proj_1((x_\sharp,p)^\sharp) \jdeq \proj_1(x_\sharp,p)^\sharp \jdeq x_\sharp{}^\sharp \jdeq x. \]

  Next suppose that $z:\sharp\brck{\fib_{\sharp\proj_1}(x)}$; we will show $\sharp\brck{P(x_\sharp)}$.
  Since $\sharp$ takes values in codiscrete types, it suffices to show $\sharp\sharp\brck{P(x_\sharp)}$.
  Thus, we can apply $\sharpf$ to make our hypotheses crisp while keeping the goal $\sharp\brck{P(x_\sharp)}$.
  Now we can form $z_\sharp : \brck{\fib_{\sharp\proj_1}(x)}$, and since $\sharp\brck{P(x_\sharp)}$ is a \mprop{} we can destruct $z_\sharp$ to obtain $q:\sharp(\sm{x:A} P(x))$ and $\sharp\proj_1(q)=x$.
  Next, by $\sharp$-induction we can assume $q$ is $(y,p)^\sharp$ for some $y:A$ and $p:P(y)$.
  But $\sharp\proj_1((y,p)^\sharp) \jdeq \proj_1(y,p)^\sharp \jdeq y^\sharp$, so we have $y^\sharp = x$.
  And since $x$ is crisp we have $x\jdeq x_\sharp{}^\sharp$, so $y^\sharp = x_\sharp{}^\sharp$.
  But by \cref{thm:path-sharp} this type is equivalent to $\sharp (y = x_\sharp)$.
  Thus, with another $\sharp$-induction we obtain $s:y=x_\sharp$, so we can transport $p$ along $s$ to get an element of $P(x_\sharp)$, and finally map it into $\sharp\brck{P(x_\sharp)}$.
\end{proof}

\begin{cor}\label{thm:emb-sharp-eqv}
  If $f:A\to B$ is an embedding and $\prd{x:B} \sharp\fib_f(x)$, then $\sharp f$ is an equivalence.
\end{cor}
\begin{proof}
  Since $f$ is an embedding, so is $\sharp f$.
  Moreover, $\sharp f$ has codiscrete fibers.
  Thus, \cref{thm:sharp-surj} is sufficient to ensure that it is an equivalence.
\end{proof}

Now, since we cannot expect every surjection to have a continuous section, the conclusion of our axiom of choice should be $\prd{x:A} \sharp P(x)$.
Therefore, it ought to suffice as a hypothesis that it be surjective on points, i.e.\ $\prd{x:A} \sharp\brck{P(x)}$.
Of course, the conclusion must be further $(-1)$-truncated outside the $\prod$, to be consistent with univalence.
This might lead us to state the axiom of choice as
\begin{equation}
  \left(\prd{x:A} \sharp\brck{P(x)}\right) \to \brck{\prd{x:A} \sharp P(x)}.\label{eq:wrong-ac}
\end{equation}
However, I believe this is not yet quite right either.
Remember that the entire statement is also parametrized by $A$ and $P$.
Thus,~\eqref{eq:wrong-ac} is actually asserting that one subobject of $\sm{A:\type}(A\to \type)$ is contained in another.
But the first subobject is codiscrete --- i.e.\ has the subspace topology --- whereas the second is \emph{not}, at least not by definition: $\brck{\blank}$ doesn't preserve codiscreteness.
Moreover, the reverse implication always holds; thus~\eqref{eq:wrong-ac} is actually asserting that these two subobjects agree, and hence that the one on the right is a subspace even though we haven't forced it to be.
This seems unreasonable (although I do not currently have a concrete counterexample).
Thus, instead we take the following:

\begin{named}{Axiom AC}(Sharp axiom of choice)\label{ax:shac}\label{ax:ac}
  For any set $A$ and type family $P:A\to \type$, we have
  \begin{equation}
    \left(\prd{x:A} \sharp\brck{P(x)}\right) \to \sharp\brck{\prd{x:A} \sharp P(x)}.\label{eq:ac}
  \end{equation}
\end{named}

This is admittedly quite a mouthful, and one is left wondering whether it is good for anything.
Are there simpler special cases?
What about, for instance, the \emph{countable} axiom of choice, where $A=\N$; surely there shouldn't be any topology on $\N$ to get in the way.
Thus, we now turn to making sense of what it means to have a \emph{discrete} topology.

\section{The $\flat$ modality}
\label{sec:flat}

\addtocontents{toc}{\protect\setcounter{tocdepth}{1}}
\subsection{A no-go theorem}
\label{sec:no-go}

Discreteness is encoded by the modality $\flat$ that is dual to $\sharp$.
Whereas $\sharp$ is a reflector into the subcategory of codiscrete types, $\flat$ is a coreflector into the subcategory of discrete types.
However, unlike $\sharp$, the behavior of $\flat$ cannot be stated as ordinary type-theoretic rules or axioms, because the resulting pullback-stability would cause it to degenerate.
This is the conclusion of the following ``no-go theorem''.

\begin{thm}\label{thm:no-go}
  Suppose we have the following data:
  \begin{enumerate}
  \item A predicate $\inbox : \type\to\prop$ that is invariant under equivalence, i.e.\ $(A\simeq B) \to \inbox(A) \to \inbox(B)$.
    (This condition is, of course, automatic with univalence.)
  \item An operation $\Box: \type\to\type$, such that $\inbox(\Box(A))$ for all $A$.
  \item For each $A:\type$, a function $\varepsilon_A : \Box A \to A$.
  \item If $\inbox(B)$, then postcomposition with $\varepsilon_A$ is an equivalence $(B\to\Box A) \simeq (B\to A)$.\label{item:nogoh4}
  \end{enumerate}
  Then there exists $U:\prop$ such that for all $A$ we have
  \begin{enumerate}[label=(\alph*)]
  \item $\inbox(A) \leftrightarrow (A\to U)$ and\label{item:nogo1}
  \item $\Box A \simeq (A\times U)$\label{item:nogo2}
  \end{enumerate}
\end{thm}
\begin{proof}
  Firstly, we observe that if $A$ is a \mprop{}, so is $\Box A$.
  For if $x,y:\Box A$, then we have two functions $\const_x, \const_y : \Box A \to \Box A$ whose composites with $\varepsilon_A$ are equal.
  %(since $A$ is a \mprop{}).
  Hence $\const_x = \const_y$ and thus (evaluating them at $x$, say) $x=y$.

  Secondly, we observe that if $\varepsilon_A : \Box A \to A$ has a section $s:A\to \Box A$, then $\inbox(A)$.
  It suffices to show that $s \circ \varepsilon_A = \id_{\Box A}$; but these are again two maps $\Box A \to \Box A$ having the same composite with $\varepsilon_A$.
  % (This is just the dual of \cref{thm:codiscrete-retraction}.)

  Thirdly, we can extend $\Box$ to a functor such that $\varepsilon$ is a natural transformation from $\Box$ to the identity.
  (This is also the dual of facts we observed in \cref{sec:codisc} for $\sharp$.)

  Now let $U \defeq \Box \unit$.
  First we prove~\ref{item:nogo1}.
  On one hand, if $\inbox(A)$, then $A\to \unit$ factors through $U$.
  On the other hand, suppose $f:A\to U$; we will construct a section $s:A\to\Box A$.
  For any $x:A$ we have $\const_x : \unit \to A$, and hence by functoriality $\Box\const_x : \Box \unit \to \Box A$.
  But $f:A\to \Box \unit$, so we have $\Box\const_x(f(x)):\Box A$; we define this to be $s(x)$.
  Now by naturality of $\varepsilon$, we have
  \[\varepsilon_A(s(x)) = \varepsilon_A(\Box\const_x(f(x))) = \const_x(\varepsilon_\unit(f(x))) = x.\]
  Thus, $\inbox(A)$; this completes the proof of \ref{item:nogo1}.

  Now we prove~\ref{item:nogo2}.
  By~\ref{item:nogo1}, we have $\inbox(A\times U)$; thus the projection $\proj_1 : A\times U\to A$ factors through $\varepsilon_A$ by some $g:A\times U\to \Box A$.
  On the other hand, since $\inbox(\Box A)$ we have $f:\Box A \to U$, and thus $(\varepsilon_A,f):\Box A \to A\times U$.
  The composite in one direction is
  \[ (\varepsilon_A,f) \circ g = (\varepsilon_A \circ g,f \circ g) = (\proj_1, \proj_2) = \id_{A\times U} \]
  where $\varepsilon_A \circ g= \id_A$ by definition of $g$, while $f\circ g = \id_U$ since $U$ is a \mprop{}.
  The composite in the other direction is a map $\Box A \to \Box A$ whose composite with $\varepsilon_A$ is $\varepsilon_A$, hence it must be the identity.
  This shows~\ref{item:nogo2}.
\end{proof}

We certainly don't expect $\flat$ to be of the form $(\blank\times U)$, so this means that we cannot describe $\flat$ as ordinary type-theoretic rules or axioms.
The central point of \cref{thm:no-go} is that the universal property~\ref{item:nogoh4} can be applied in an arbitrary context; this is what enables us to construct a section $A\to\Box A$ from a map $A\to \Box\unit$.
Thus, in order to describe a more general coreflection, we need to restrict the context in which it and its universal property can be applied.
In our current setup, an obvious restriction to try is that the context should include only crisp variables.
This is what we will do.

\subsection{The rules for $\flat$}
\label{sec:rules-flat}

We now describe our rules for $\flat$ in English; a type-theoretic presentation is shown in \cref{fig:flat}.
\begin{itemize}
\item For any \emph{crisp} type $A::\type$, there is a type $\flat A$.
\item For any \emph{crisp} $a::A$, we have $a^\flat : \flat A$.
\item An element of $\flat A$ may always be assumed to be of the form $u^\flat$ for some crisp variable $u::A$.
  This is called \textbf{$\flat$-induction}.
  Syntactically, if $C$ is a type depending on $x:\flat A$, and $N:C[u^\flat/x]$ is an expression depending on a crisp variable $u::A$, while we have an element $M:\flat A$, we have an induced element
  \[ (\flet u^\flat := M in N) : C[M/x]. \]

  Here the notation $C[M/x]$ indicates that the expression $M$ is substituted for the variable $x$ in $C$.
  \textcite{hottbook} used function notation when discussing induction principles, but for $\flat$ we cannot since we do not have a basic notion of ``function of a crisp variable'' (indeed, one of the purposes of $\flat$ is to provide such a notion); thus we have to talk about substituting into expressions rather than evaluating functions.
\item The expected computation rule: $(\flet u^\flat := M^\flat in N) \jdeq N[M/u]$
  if $M::A$ is crisp.
\end{itemize}
It is crucial that we cannot apply $\flat$ to types containing cohesive variables.
This reflects the fact that, by \cref{thm:no-go}, ``retopologizing discretely'' is not a \emph{continuous} operation on the space of spaces.
(It \emph{is} possible to do this for ``retopologizing codiscretely'', so that $\sharp$ requires no such restriction.
Indeed, by \cref{thm:codiscrete-universe} the type of codiscrete types is codiscrete, so any function into it is automatically continuous.)

\begin{figure}
  \centering
  \begin{mathpar}
    \inferrule{\Delta\mid\cdot \types A:\type}{\Delta\mid\Gamma \types \flat A :\type} \and
    \inferrule{\Delta\mid\cdot \types M:A}{\Delta\mid\Gamma \types M^\flat: \flat A} \and
    \inferrule{\Delta\mid\Gamma,x:\flat A \types C:\type \\ \Delta\mid\Gamma \types M:\flat A \\ \Delta,u::A\mid\Gamma \types N:C[u^\flat/x]}{\Delta\mid\Gamma \types (\flet u^\flat := M in N):C[M/x]} \and
    \inferrule{\Delta\mid\Gamma,x:\flat A \types C:\type \\ \Delta\mid\cdot \types M:A \\ \Delta,u::A\mid\Gamma \types N:C[u^\flat/x]}{\Delta\mid\Gamma \types(\flet u^\flat := M^\flat in N) \jdeq N[M/u] \;:\: C[M^\flat/x]}
  \end{mathpar}
  \caption{Rules for $\flat$}
  \label{fig:flat}
\end{figure}

\begin{rmk}
  Strictly speaking, the syntax ``$\flet u^\flat := M in N$'' should also notate the type family $C$, since that is not generally uniquely inferrable.
  However, when working informally this will not matter much.
\end{rmk}

\begin{rmk}
  Dually to $\sharp$ (see \cref{thm:negative}), $\flat$ is a \emph{positive} type former.\footnote{Well, at least in the usual sense that this adjective is used for dependent type theories.
    In the sequent calculus of \textcite{ls:1var-adjoint-logic}, $\flat$ is (conjecturally) positive in the precise sense of ``focusing'', but our presentation of $\flat$ breaks some of this positivity.
    However, this often happens to positive types when translating from sequent calculus into dependent type theory.}
  This means, roughly, that its elimination rule is chosen to match its introduction rule(s).
  For instance, the coproduct $A+B$ is positive because to define a (dependent) function \emph{out} of it, we use case analysis, which essentially means we have to specify the values of that function on all the ways to construct an element of $A+B$ (namely via $\inl$ or $\inr$).
  In the case of $\flat$, the introduction rule says that the way to get an element of $\flat A$ is to have a crisp element of $A$; thus the elimination rule says that to define a function out of $\flat A$ it suffices to assume we have a crisp element of $A$.

  In contrast to negative type formers, whose uniqueness principles tend to be judgmental rules, positive type formers tend to have uniqueness principles that are only typal, but can be proven from their elimination and computation rules.
  As we will now see, this is also the case for $\flat$.
\end{rmk}

\begin{lem}[Uniqueness principle for $\flat$]\label{thm:flat-eta}
  Let $A::\type$, $C:\flat A \to \type$, and $f:\prd{x:\flat A} C(x)$.
  For any $x:\flat A$ we have
  \[ (\flet u^\flat := x in f(u^\flat)) = f(x). \]
\end{lem}
\begin{proof*}
  Let $P(x) \defeq ((\flet u^\flat := x in f(u^\flat)) = f(x))$; then $P:\flat A \to \type$.
  Thus, in proving $\prd{x:\flat A} P(x)$ we can use $\flat$-induction.
  But when $x$ is $v^\flat$, we have $P(v^\flat) \jdeq (f(v^\flat) = f(v^\flat))$, using the computation rule for $\flat$ on the left-hand side, and this is trivial by reflexivity.
  In symbols, we have
  \[ (\flet v^\flat := x in \refl_{f(v^\flat)}) \;:\; (\flet u^\flat := x in f(u^\flat)) = f(x).\qedhere \]
\end{proof*}

As with $\sharp$, we can make $\flat$ into a functor, at least on crisp functions: given $A,B::\type$ and $f::A\to B$, we define $\flat f : \flat A \to \flat B$ by
\begin{equation}
  \flat f(x) \defeq (\flet u^\flat := x in f(u)^\flat).\label{eq:flat-f}
\end{equation}
We might be tempted to write instead $(\flet u^\flat := x in f(u))^\flat$, but this is invalid because the inside of $(\blank)^\flat$ cannot contain cohesive variables such as $x$.
For the same reason, we require not only $A$ and $B$, but also $f$, to be crisp.
In particular, we cannot deduce from this a map $(A\to B) \to (\flat A \to \flat B)$.
The best we can do is a map $\flat (A\to B) \to \flat (\flat A \to \flat B)$, defined by $\flat$-induction followed by~\eqref{eq:flat-f}.

% \subsection{Commuting conversions}
% \label{sec:comm-conv-crisp}

When we try to verify the functoriality of this operation, we see that we need some additional lemmas.
Given $A,B,C::\type$ and $f::A\to B$ and $g::B\to C$, in addition to~\eqref{eq:flat-f} we have
\begin{equation}
  \flat g(y) \defeq (\flet v^\flat := y in g(v)^\flat)\label{eq:flat-g}
\end{equation}
and thus for any $x:\flat A$, we have
\[ \flat g(\flat f(x)) \jdeq \big(\flet v^\flat := (\flet u^\flat := x in f(u)^\flat) in g(v)^\flat\big) \]
but on the right-hand side, what are we to do with this ``nested $\mathsf{let}$''?
We have to apply the following lemma, called a \emph{(typal) commuting conversion}.

\begin{lem}[Commuting \textsf{let} with itself]\label{thm:commute-let}
  Whenever both sides typecheck, we have
  % For any expressions $M:\flat A$ and $N:\flat B$ and $P:C[v^\flat/y]$, where
  % \begin{itemize}
  % \item $N$ may depend on $u::A$,
  % \item $P$ may depend on $v::B$, and
  % \item $C$ may depend on $y:\flat B$,
  % \end{itemize}
  % but no other dependencies on $x,y,u,v$ are allowed, we have
  \[ \big(\flet v^\flat := (\flet u^\flat := M in N) in P\big) = \big(\flet u^\flat := M in (\flet v^\flat := N in P)\big). \]
  Moreover, if the left-hand side typechecks, so does the right-hand side, while if the right-hand side typechecks and $u$ does not occur in $P$, then the left-hand side typechecks.
\end{lem}
\begin{proof}
  % Note that, like $\flat$-induction and $\flat$-computation, we have to state this using expressions and substitution rather than functions and application, because of the possible dependence on crisp variables.
  % However, because $M$ does not depend on any of the crisp variables introduced (though $N$ and $P$ may), we can apply $\flat$-induction to it.
  One proof is to do $\flat$-induction on $M$:
  when $M$ is $w^\flat$ for $w::A$, both sides of the goal reduce to $(\flet v^\flat := N[w/u] in P)$, so we can apply reflexivity.
  Another is to apply the uniqueness principle in reverse, followed by the computation rule:
  \begin{align*}
    &\hspace{2.8cm} \flet v^\flat := (\flet u^\flat := M in N) in P\\
    &= \flet w^\flat := M in (\flet v^\flat := (\flet u^\flat := w^\flat in N) in P)\\
    &\jdeq \flet w^\flat := M in (\flet v^\flat := N[w/u] in P)\\
    &\jdeq \flet u^\flat := M in (\flet v^\flat := N in P)
  \end{align*}
  The second proof makes the typechecking conditions clearer, as these rewriting steps can always be performed forwards, whereas to perform them backwards we need to know that $u$ does not occur in $P$.
\end{proof}

\begin{rmk}\label{rmk:derivable-rule}
  In \cref{rmk:entailment} we mentioned the difference between implication and entailment, whose type-theoretic syntaxes are $\Gamma\types f:A\to B$ and $\Gamma,x:A \types f(x):B$ respectively.
  Both of these, however, are single ``hypothetical judgments'', whose hypotheses consist of a context of variables (either ``$\Gamma$'' or ``$\Gamma,x:A$'').
  By contrast, \cref{thm:commute-let} cannot be expressed as a \emph{single} hypothetical judgment, because $M$, $N$, and $P$ must be ``metavariables'' denoting \emph{expressions} rather than ordinary variables.
  This should be clear from the condition ``$u$ does not occur in $P$'', which would be nonsensical if $P$ were a variable.
  In type-theoretic language, \cref{thm:commute-let} is a \emph{derivable rule} rather than a single judgment, because among its \emph{hypotheses} are other \emph{hypothetical judgments}; \cref{fig:uniq-comm} shows it written out in type-theoretic syntax.
  We will try to avoid such ``meta-statements'' as much as possible, but in this section and the next we will need a few more of them.
\end{rmk}

\begin{figure}
  \centering
  \begin{mathpar}
    \inferrule*[right=(\ref{thm:commute-let})]{                 % TODO: Can this one allow more dependence?
      \Delta\mid\cdot \types A : \type \\
      \Delta\mid\cdot \types B : \type \\
      \Delta\mid\Gamma,y:\flat B \types C : \type \\
      \Delta\mid\Gamma \types M : \flat A\\
      \Delta,u::A\mid\Gamma \types N : \flat B \\
      \Delta,v::B\mid\Gamma \types P:C[v^\flat/y]}
    {\Delta\mid\Gamma \types \big(\flet v^\flat := (\flet u^\flat := M in N) in P\big) = \big(\flet u^\flat := M in (\flet v^\flat := N in P)\big)}
    \and
    \inferrule*[right=(\ref{thm:commute-app})]{\Delta\mid\cdot \types A : \type \\
      \Delta\mid\Gamma,x:\flat A \types B : \type \\
      \Delta\mid\Gamma,x:\flat A, y:B \types C : \type \\
      \Delta\mid\cdot \types M : \flat A \\
      \Delta,u::A \mid\Gamma \types N : B[u^\flat/x] \\
      \Delta\mid\Gamma \types f:\tprd{x:\flat A}{y:B} C}
    {\Delta\mid\Gamma \types f\big(M,(\flet u^\flat := M in N)\big) = \big(\flet u^\flat := M in f(u^\flat,N)\big)}
  \end{mathpar}
  \caption{Commuting conversions}
  \label{fig:uniq-comm}
\end{figure}

With \cref{thm:commute-let}, we can now verify the functoriality of $\flat$:
\begin{align*}
  \flat g(\flat f(x))
  &\jdeq \flet v^\flat := (\flet u^\flat := x in f(u)^\flat) in g(v)^\flat\\
  &= \flet u^\flat := x in (\flet v^\flat := f(u)^\flat in g(v)^\flat)\\
  &\jdeq \flet u^\flat := x in g(f(u))^\flat\\
  &\jdeq \flat(g\circ f)(x).
\end{align*}
Functoriality on identities is easier, using the uniqueness principle:
\begin{equation*}
  \flat(\id_A)(x)
  \jdeq (\flet u^\flat := x in u^\flat)
  = x
\end{equation*}
Note that unlike the functoriality of $\sharp$, these are typal rather than judgmental equalities.
Thus, they should actually be viewed as the first level of a ``coherent \oo-functor''; later on we will construct the next level by showing that $\flat$ is also ``functorial on homotopies''.

Continuing with the analogy to $\sharp$, in place of the ``unit'' $\sharpf : A \to \sharp A$ we have a ``counit'' $\flatf : \flat A \to A$, which is defined for any (crisp!)\ $A::\type$ by
\[x_\flat \defeq (\flet u^\flat := x in u). \]
The computation rule tells us that for $v::A$ we have
\begin{equation}
  v^\flat{}_\flat \jdeq (\flet u^\flat := v^\flat in u) \jdeq v.\label{eq:beta-flatf}
\end{equation}
We expect $\flatf$ to be a natural transformation (and eventually a ``coherent \oo-natural transformation'') from $\flat$ to the identity functor.
To verify this, let $A,B::\type$ and $f::A\to B$ and fix any $x:\flat A$; then we would like to compute
\begin{alignat*}{2}
  (\flat f(x))_\flat
  &\jdeq (\flet u^\flat := x in f(u)^\flat)_\flat\\
  &= \flet u^\flat := x in f(u)^\flat{}_\flat
  &\qquad\text{(\cref{thm:commute-let})}\\
  &\jdeq \flet u^\flat := x in f(u)\\
  &= f(\flet u^\flat := x in u)
  &\qquad\text{(?)}\\
  &\jdeq f(x_\flat).
\end{alignat*}
To fill in the step marked ?, we need another commuting conversion.

\begin{lem}[Commuting \textsf{let} with function application]\label{thm:commute-app}
  Suppose given types $A$, $B$, and $C$, where $A$ is crisp, $B$ may depend on $x:\flat A$, and $C$ may depend on $x:\flat A$ and $y:B$.
  Also suppose we have expressions $M:\flat A$ and $N:B[u^\flat/x]$ and $f:\prd{x:\flat A}{y:B} C$, where $N$ may depend on $u::A$.
  Then we may conclude
  \[ f\big(M,(\flet u^\flat := M in N)\big) = \big(\flet u^\flat := M in f(u^\flat,N)\big). \]
\end{lem}
Like \cref{thm:commute-let}, this is formally a ``derivable rule''; in \cref{fig:uniq-comm} we have written it out in type-theoretic syntax.
\begin{proof*}
  As before, $\flat$-induction on $M$ reduces both sides to $f(w^\flat,N[w/u])$; or we can expand and contract:
  \begin{align*}
    &\hspace{2.7cm} f\big(M,(\flet u^\flat := M in N)\big)\\
    &= \flet v^\flat := M in f\big(v^\flat,(\flet u^\flat := v^\flat in N)\big)\\
    &\jdeq \flet v^\flat := M in f(v^\flat,N[v/u])\\
    &\jdeq \flet u^\flat := M in f(u^\flat,N).\qedhere
  \end{align*}
\end{proof*}

Analogous commuting conversions can be proven for the constructors and eliminators of all other ordinary types; we leave them to the reader.
However, there is one conversion we need that is more subtle, relating to $\flat$ itself.
To see how this arises, consider the obvious dual claim to~\eqref{eq:beta-flatf}, that $x_\flat{}^\flat = x$.
Note that $x_\flat{}^\flat$ only makes sense when $x$ is a \emph{crisp} element of $\flat A$, since we can only apply $(\blank)^\flat$ to crisp elements.
In this case we would like to write
\begin{equation}
  x_\flat{}^\flat \jdeq (\flet u^\flat := x in u)^\flat = (\flet u^\flat := x in u^\flat) = x,\label{eq:eta-flatf}
\end{equation}
using the uniqueness principle (\cref{thm:flat-eta}) at the end; but in the middle we need to be able to commute $(\blank)^\flat$ past $\mathsf{let}$.
However, if we try to prove this as we did for the other commuting conversions, we run into the problem that (as remarked above) $x::\flat A$ must be crisp, whereas $\flat$-induction is only stated for types $C$ depending on a \emph{cohesive} variable $x:\flat A$.
Thus, we now take a brief digression to investigate \emph{crisp induction principles}.

\section{Crisp induction principles}
\label{sec:crisp-induction}

% \subsection{Crisp $\flat$-induction}
% \label{sec:crisp-flat-induction}

Here is the lemma we need about $\flat$.
Note that its proof involves $\sharp$ as well!

\begin{lem}[Crisp $\flat$-induction]\label{thm:crisp-flat-ind}
  Let $C$ be a crisp type depending on a crisp variable $x::\flat A$, and $N::C[u^\flat/x]$ a crisp expression depending on a crisp variable $u::A$.
  If we have a crisp element $M::\flat A$, then we have an element of $C[M/x]$, which we denote
  \[(\flet u^\flat ::= M in N) : C[M/x]. \]
  Moreover, if $M::A$ then
  \[(\flet u^\flat ::= M^\flat in N) \jdeq N[M/u].\]
\end{lem}
% Since ordinary $\flat$-induction already makes the hypothesis crisp, the only nontrivial content in this case is that the type $C$ can depend on a crisp variable $x::\flat A$ rather than a cohesive one.
\begin{proof*}
  If we have only a cohesive variable $x:\flat A$, we can still write $\sharp C : \type$, since cohesive variables can become crisp inside $\sharp(\blank)$.
  Similarly, if $u::A$ we have $N^\sharp : \sharp C[u^\flat/x]$.
  Thus, since $M:\flat A$, we have $(\flet u^\flat := M in N^\sharp) : \sharp C[M/x]$.
  And since this expression is crisp, we have $(\flet u^\flat := M in N^\sharp)_\sharp : C[M/x]$, which we define to be ``$\flet u^\flat ::= M in N$''.
  Finally, if $M::A$ then
  \[ (\flet u^\flat := M^\flat in N^\sharp)_\sharp \jdeq N[M/u]^\sharp{}_\sharp \jdeq N[M/u]. \qedhere \]
\end{proof*}
This proof is somewhat tricky, so we include a derivation tree for it; see \cref{fig:crisp-flat-ind}. % on page~\pageref{fig:crisp-flat-ind}.
\begin{figure}
  \centering
  \[
  \inferrule
  {\inferrule
    {\inferrule*{\Delta, x::\flat A\mid\cdot \types  C : \type}
      {\Delta\mid x:\flat A \types  \sharp C : \type} \\
      \Delta\mid\cdot \types M : \flat A \\
      \inferrule*{\Delta,u::A\mid\cdot \types N:C[u^\flat/x]}{\Delta,u::A\mid\cdot \types N^\sharp : \sharp C[u^\flat/x]}}
    {\Delta\mid \cdot \types (\flet u^\flat := M in N^\sharp) : \sharp C[M/x]}}
  {\Delta\mid\Gamma \types (\flet u^\flat := M in N^\sharp)_\sharp:C[M/x]}
  \]
\caption{A derivation tree for \cref{thm:crisp-flat-ind}}
\label{fig:crisp-flat-ind}
\end{figure}
The crispness restriction on $C$ and $N$ can be worked around in the usual ``Frobenius'' way: if $C$ and $N$ depend on some cohesive variable $y:B$, we can form $C' \defeq \prd{y:B} C$ and $N' \defeq \lam{y}N$ and apply \cref{thm:crisp-flat-ind} to them instead.
By induction over metatheoretic natural numbers, we can incorporate an arbitrary cohesive context.

Now we can prove the needed commuting conversion.

\begin{lem}[Commuting \textsf{let} with $(\blank)^\flat$]\label{thm:commute-flat}
  Given crisp expressions $M::\flat A$ and $N::B[u^\flat/x]$, where $B$ depends on a crisp variable $x::\flat A$ and $N$ depends on $u::A$, we have
  \[ (\flet u^\flat := M in N)^\flat = (\flet u^\flat := M in N^\flat) \]
\end{lem}
\begin{proof}
  By \cref{thm:crisp-flat-ind}, we can replace $M$ by $w^\flat$ for some $w::A$.
  Now both sides reduce to $N^\flat[w/u]$.
\end{proof}

At last, the computation~\eqref{eq:eta-flatf} makes sense.
Later on we will also have use for the corresponding type-level commuting conversion:

\begin{lem}[Commuting \textsf{let} with $\flat$]\label{thm:commute-flat-type}
  Given crisp expressions $M::\flat A$ and $N::\type$, where $N$ may depend on $u::A$, we have
  \[ \flat(\flet u^\flat := M in N) = (\flet u^\flat := M in \flat N). \qedhere \]
\end{lem}

\begin{figure}
  \centering
  \begin{mathpar}
    \inferrule*[right=(\ref{thm:crisp-flat-ind})]{\Delta,x::\flat A\mid\cdot \types C:\type \\ \Delta\mid\cdot \types M:\flat A \\ \Delta,u::A\mid\cdot \types N:C[u^\flat/x]}{\Delta\mid\Gamma \types (\flet u^\flat ::= M in N):C[M/x]}
    \and
    % \inferrule*[right=(\ref{thm:crisp-flat-ind})]{\Delta,x::\flat A\mid\cdot \types C:\type \\ \Delta\mid\cdot \types M:A \\ \Delta,u::A\mid\cdot \types N:C[u^\flat/x]}
    % {\Delta\mid\Gamma \types
    (\flet u^\flat ::= M^\flat in N) \jdeq N[M/u]
    % \;:\: C[M^\flat/x]}
    \and
    \inferrule*[right=(\ref{thm:commute-flat})]{\Delta\mid\cdot \types A : \type \\
      \Delta,x::\flat A\mid\cdot \types B : \type \\\\
      \Delta\mid\cdot \types M : \flat A \\
      \Delta,u::A\mid\cdot \types N : B[u^\flat/x]}
    {\Delta\mid\Gamma \types (\flet u^\flat := M in N)^\flat = (\flet u^\flat := M in N^\flat)}
    \and
    \inferrule*[right=(\ref{thm:commute-flat})]{\Delta\mid\cdot \types A : \type \\
      \Delta\mid\cdot \types M : \flat A \\
      \Delta,u::A\mid\cdot \types N : \type}
    {\Delta\mid\Gamma \types \flat(\flet u^\flat := M in N) = (\flet u^\flat := M in \flat N)}
  \end{mathpar}
  \caption{Crisp $\flat$-induction and commuting \textsf{let} with $\flat$}
  \label{fig:crisp-flat}
\end{figure}

% \subsection{Other crisp induction principles}
% \label{sec:other-crisp-induct}

The idea used to prove \cref{thm:crisp-flat-ind} can be applied to other positive types as well.
For instance, since we assert the rules of ordinary type theory only for cohesive variables, the eliminator for the coproduct type is
\[
\inferrule{\Delta\mid\Gamma,z:A+B \types C:\type \\ \Delta\mid\Gamma \types M:A+B\\ \Delta\mid\Gamma,x:A\types c_A : C[\inl(x)/z] \\ \Delta\mid\Gamma,y:B\types c_B : C[\inr(y)/z]}
{\Delta\mid\Gamma \types \mathsf{case}(z.C,x.c_A,y.c_B,M) : C[M/z]}
\]
Thus, the element $M$ being case-analyzed is cohesive, the variables $x$ and $y$ that it is destructed into are cohesive, and the goal $C$ depends on a cohesive $z:A+B$.
The first restriction is not significant since we can always use a crisp $M$ as a cohesive one, but the others are nontrivial.
The method of \cref{thm:crisp-flat-ind} allows us to remove them.

\begin{thm}[Crisp case analysis]\label{thm:crisp-case}
  Let $A$ and $B$ be crisp types, let $C$ be a crisp type depending on a {crisp} variable $z::A+B$, and let $N::C[\inl(u)/z]$ and $P::C[\inr(v)/z]$ be crisp expressions depending on {crisp} variables $u::A$ and $v::B$.
  If we have a crisp $M::A+B$, then we have an element $\mathsf{case}^\flat(z.C,u.N,v.P,M):C[M/z]$, which computes to $N$ and $P$ on $\inl$ and $\inr$.
\end{thm}
\cref{fig:crisp-ind} shows \cref{thm:crisp-case} in type-theoretic syntax.
\begin{proof}
  If we have only a cohesive variable $z:A+B$, we can still write $\sharp C : \type$.
  Similarly, if $u:A$ we have $N^\sharp : \sharp C[\inl(u)/z]$, and if $v:B$ we have $P^\sharp : \sharp C[\inr(v)/z]$.
  Thus, treating the crisp $M::A+B$ as cohesive, we can do ordinary case analysis to yield $\mathsf{case}(z.\sharp C,u.N^\sharp,v.P^\sharp,M):\sharp C[M/z]$.
  But since $C$, $M$, $N$, and $P$ are crisp, so is this expression, so we can write
  \[\mathsf{case}(z.\sharp C,u.N^\sharp,v.P^\sharp,M)_\sharp : C[M/z]. \]
  The computation rule follows from those for $\mathsf{case}$ and $\sharp$ combined.
\end{proof}

\begin{figure}
  \centering
  \begin{mathpar}
    \inferrule*[right=(\ref{thm:crisp-case})]{\Delta,z:A+B\mid\cdot \types C:\type \\
      \Delta,u::A\mid\cdot\types N:C[\inl(u)/z] \\
      \Delta,v::B\mid\cdot\types P:C[\inr(v)/z] \\
      \Delta\mid\cdot \types M:A+B}
    {\Delta\mid\Gamma \types \mathsf{case}^\flat(z.C,u.N,v.P,M) : C[M/z]}
    \and
    \mathsf{case}^\flat(z.C,u.N,v.P,\inl(u)) \jdeq N
    \and
    \mathsf{case}^\flat(z.C,u.N,v.P,\inr(v)) \jdeq P
    % \and
    % \inferrule*[right=(\ref{thm:crisp-nat-ind})]{\Delta,n::\N\mid\cdot\types C:\type \\
    %   \Delta\mid\cdot\types Z:C[0/n]\\
    %   \Delta,m::\N,p::C[m/n]\mid\cdot\types S:C[(m+1)/n]\\
    %   \Delta\mid\cdot \types M:\N}
    % {\Delta\mid\Gamma\types \mathsf{ind}^\flat(n.C,Z,m.p.S,M): C[M/x]}
    % \and
    % \mathsf{ind}^\flat(n.C,Z,m.p.S,0) \jdeq Z
    % \and
    % \mathsf{ind}^\flat(n.C,Z,m.p.S,M+1) \jdeq S[M/m,\mathsf{ind}^\flat(n.C,Z,m.p.S,M)/p]
    \and
    \inferrule*[right=(\ref{thm:crisp-J})]{\Delta,u::B,v::B,p::u=v\mid\cdot \types C : \type \\
      \Delta,u::B \mid\cdot \types d : C[u/v,\refl_u/p]\\
      \Delta\mid\cdot \types b_1 : B\\
      \Delta\mid\cdot \types b_2 : B\\
      \Delta\mid\cdot \types q : b_1=b_2}
    {\Delta\mid\Gamma \types \J^\flat_{u,v,p.C}(u.d; b_1,b_2,q) : C[b_1/u,b_2/v,q/p]}
    \and
    % \inferrule*[right=(\ref{thm:crisp-J})]{\Delta,u::B,v::B,p::u=v\mid\cdot \types C : \type \\
    %   \Delta,u::B \mid\cdot \types d : C[u/v,\refl_u/p]\\
    %   \Delta\mid\cdot \types b : B}
    % {\Delta\mid\Gamma \types
    \J^\flat_{u,v,p.C}(u.d; u,u,\refl_u) \jdeq d
    % }
  \end{mathpar}
  \caption{Some crisp induction principles}
  \label{fig:crisp-ind}
\end{figure}

A more or less equivalent way to state \cref{thm:crisp-case} is that $\flat$ preserves coproducts.
This is sensible from a category-theoretic point of view, since we will soon see that $\flat$ is left adjoint to $\sharp$.

\begin{cor}\label{thm:flat-pres-coprod}
  $\flat A + \flat B \simeq\flat(A+B)$ for any crisp types $A$ and $B$.
\end{cor}
\begin{proof}
  To map from left to right, we do ordinary case analysis and then use the functoriality of $\flat$.
  From right to left, we do $\flat$-induction, then crisp case analysis, then apply $(\blank)^\flat$ in both cases.
  We leave the details to the reader, along with the proof that these are inverses.
\end{proof}

We can apply the same method to other positive types such as $\emptyset$ and $\N$, concluding in particular that $\flat\emptyset \simeq \emptyset$ and $\flat\N\simeq\N$.
It also applies to identity types:

\begin{thm}[Crisp Id-induction]\label{thm:crisp-J}
  Suppose $C::\type$ depends on $u,v::B$ and $p::u=v$, and that we have $d :: C[u/v,\refl_u/p]$ depending on $u::B$.
  Then for any $u,v::B$ and $p::u=v$ we have $\J^{\flat}_{u,v,p.C}(u.d; u,v,p) : C$, such that $\J^{\flat}_{u,v,p.C}(u.d;u,u,\refl_u) \jdeq d$.
\end{thm}
\begin{proof*}
  As before, if we have cohesive variables $x,y:B$ and $q:x=y$ we can write $\sharp C[x/u,y/v,q/p]$. %, since all variables can become crisp inside $\sharp$.
  Now we apply ordinary Id-induction: to inhabit $\sharp C[x/u,y/v,q/p]$ for all $x,y:A$ and $q:x=y$, it suffices to inhabit $\sharp C[x/u,x/v,\refl_x/p]$ for all $x:A$.
  But such an inhabitant is supplied by $d[x/u]^\sharp$ (note that $u$ is crisp in $d$, but inside $\sharpf$ we can assume $x$ to be crisp).

  We have shown that for all $x,y:A$ and $q:x=y$, we have
  \[\J_{x,y,q.\sharp C[x/u,y/v,q/p]}(x.d[x/u]^\sharp; x,y,q) \;:\; \sharp C[x/u,y/v,q/p]. \]
  Now if we in fact have \emph{crisp} assumptions $u,v::A$ and $p::u=v$, then this becomes a crisp conclusion
  \[\J_{x,y,q.\sharp C[x/u,y/v,q/p]}(x.d[x/u]^\sharp; u,v,p) \;::\; \sharp C \]
  (since $C$ and $d$ are also crisp).
  Thus, we can apply $(\blank)_\sharp$ to obtain an element of $C$.
  Finally, if we substitute $(u,u,\refl_u)$ for $(u,v,p)$, we have
  \[ \J_{x,y,q.\sharp C(x,y,q)}(x.d[x/u]^\sharp; u,u,\refl_u)_\sharp
  \jdeq d^\sharp{}_\sharp
  \jdeq d.\qedhere
  \]
\end{proof*}

In particular, we can transport in a crisp type along a crisp identification: if $C::\type$ depends on $z::A$ and we have $a,b::A$ and $p::a=b$ and $c::C[a/z]$, then we have $p_{**}(c) :: C[b/z]$.
We repeat that the important novelty here is that the definition of $C$ might use the crispness of $z$.
Note that since crisp Id-induction is defined using an ordinary Id-induction followed by $(\blank)_\sharp$, the uniqueness rule for $\sharp$ implies $p_{**}(c)^\sharp = p_*(c^\sharp)$.

With ``crisp transport'', we can state and prove crisp induction principles for some higher inductive types as well, and particularly for homotopy colimits.

\begin{thm}[Crisp $\coeq$-induction]\label{thm:crisp-coeq-ind}
  Let $A,B::\type$ and $f,g::A\to B$ with homotopy coequalizer $q::B\to \coeq(f,g)$ and $p::\prd{u:A} q(f(u))=q(g(u))$.
  Let $C::\type$ depend on $z::\coeq(f,g)$, let $N::C[q(v)/z]$ depend on $v::B$, and let $P::p(u)_{**}(N[q(f(u))/z]) = N[q(g(u))/z]$ depend on $u::A$.
  Then for any $M::\coeq(f,g)$ we have an element of $C[M/z]$, which computes to $N$ on $q$ and to $P$ on $f,g$.
\end{thm}
\begin{proof}
  For a cohesive variable $z:\coeq(f,g)$ we can write $\sharp C:\type$, and apply ordinary $\coeq$-induction.
  We also have $N^\sharp : \sharp C[q(v)/z]$ depending on a cohesive $v:B$.
  As remarked above, we have $p(u)_*(N[q(f(u))/z]^\sharp) = p(u)_{**}(N[q(f(u))/z])^\sharp$.
  Now both $p(u)_{**}(N[q(f(u))/z])^\sharp$ and $N[q(g(u))/z]^\sharp$ belong to $\sharp C[q(g(u))/z]$, so by the proof of \cref{thm:path-sharp} we have
  \begin{alignat*}{2}
    && \Big(p(u)_{**}(N[q(f(u))/z])^\sharp &= N[q(g(u))/z]^\sharp\Big)\\
    &\simeq&\quad \sharp\Big(p(u)_{**}(N[q(f(u))/z])^\sharp{}_\sharp &= N[q(g(u))/z]^\sharp{}_\sharp\Big)\\
    &\jdeq &\quad \sharp\Big(p(u)_{**}(N[q(f(u))/z]) &= N[q(g(u))/z]\Big)
  \end{alignat*}
  which is the type of $P^\sharp$.
  Thus we have the inputs to ordinary $\coeq$-induction, after which we can apply $(\blank)_\sharp$ as usual.
\end{proof}

Similarly, we have crisp induction for $\hocirc$, for pushouts, and so on.
The case of truncations requires a little more thought.

\begin{thm}[Crisp truncation-induction]\label{thm:crisp-trunc-ind}
  Let $n::\N$ and $A::\type$, let $C$ be a crisp $n$-type depending on $z::\trunc n A$, and let $N::C[\tproj{n}{x}/z]$ depend on $x::A$.
  For any $M::\trunc n A$, we have an element of $C[M/z]$, which computes to $N$ on $\tprojf{n}$.
\end{thm}
\begin{proof}
  For a cohesive variable $z:\trunc n A$, we have $\sharp C : \type$, to which we apply ordinary truncation-induction.
  We have $N^\sharp : \sharp C[\tproj{n}{x}/z]$ for a cohesive $x:A$; so it remains to show that $\sharp C$ is an $n$-type.
  But in \cref{thm:codiscrete-ntypes} we showed that $\sharp$ preserves $n$-types, and it is straightforward to generalize this (using crisp $\N$-induction) to a derivable rule in which $n$ is crisp and the input $n$-type $C$ depends on crisp variables (that become cohesive in $\sharp C$).
  Thus ordinary truncation-induction applies, after which we use $(\blank)_\sharp$ as usual.
\end{proof}

In the sequent calculus of \textcite{ls:1var-adjoint-logic}, these crisp induction principles are taken as given, rather than derived from $\sharp$.
This produces a better-behaved calculus in a proof-theoretic sense.
However, not \emph{all} (higher) inductive types can consistently admit a crisp induction principle; we will see a counterexample in \cref{thm:no-crisp-shape-ind}.
For this reason, I prefer to derive all such principles from $\sharp$.

\section{Discreteness and $\flat$ revisited}
\label{sec:disc}

\subsection{Identifications in $\flat$}
\label{sec:flat-left-exact}

Just as we showed in \cref{thm:path-sharp} that $\sharp$ is left exact using an encode-decode argument, we can do the same for $\flat$.

\begin{thm}\label{thm:equiv-flat-path}
  For any $x,y::A$ we have $(x^\flat = y^\flat) \simeq \flat(x=y)$.
\end{thm}
\begin{proof}
  We define $\code: \flat A \to \flat A \to \type$ by
  \[ \code(u,v) \defeq (\flet x^\flat := u in (\flet y^\flat := v in \flat(x=y))). \]
  Then for any $x,y::A$, $\code(x^\flat,y^\flat) \jdeq \flat(x=y)$.
  Thus, it will suffice to prove that for any $u,v:\flat A$ we have $\code(u,v) \simeq (u=v)$.

  First note that for any $u:\flat A$, we have
  \[ \mathsf{r}(u) \defeq (\flet x^\flat := u in (\refl_x)^\flat) : \code(u,u). \]
  As usual, we then define $\encode : \prd{u,v:\flat A} (u=v) \to \code(u,v)$ by $\encode(p) \defeq p_* \mathsf{r}(u)$.

  Second, to construct $\decode : \prd{u,v:\flat A} \code(u,v) \to (u=v)$, if we destruct $u$ and $v$ as $x^\flat$ and $y^\flat$ respectively with $x,y::A$, we obtain $c:\flat(x=y)$ with a goal of $(x^\flat = y^\flat)$.
  Then we can further destruct $c$ as a crisp $p::x=y$.
  Now our assumptions of $x$, $y$, and $p$ suffice to apply crisp Id-induction, \cref{thm:crisp-J}; this reduces our goal to $(x^\flat = x^\flat)$ for a given $x::A$.
  But now of course we can use $\refl_{x^\flat}$.
  In symbols, we have
  \begin{multline*}
    \decode_{u,v}(c) \defeq\\ (\flet x^\flat := u in (\flet y^\flat := v in (\flet p^\flat := c in \J^{\flat}_{x,y,p.(x^\flat=y^\flat)}(x.\refl_{x^\flat};x,y,p)))).
  \end{multline*}

  Now as remarked in \cref{thm:path-sharp}, it suffices to show that for any $u,v:\flat A$ and $c:\code(u,v)$ we have $\encode(\decode(c)) = c$.
  Of course, we destruct $u,v,c$ as $x,y::A$ and $p::x=y$, and then apply crisp Id-induction; the goal then becomes $(\refl_{x^\flat})_* (\refl_x{}^\flat) = \refl_x{}^\flat$, which is true.
\end{proof}

\begin{cor}\label{thm:equiv-path-flat}
  For any $u,v::\flat A$, we have $(u=v)\simeq \flat(u_\flat = v_\flat)$.
\end{cor}
\begin{proof}
  We showed above that any crisp element $u::\flat A$ is equal to $u_\flat{}^\flat$, so this follows from \cref{thm:equiv-flat-path}.
\end{proof}

Note that $u$ and $v$ in the statement of \cref{thm:equiv-path-flat} must be crisp for the right-hand side to make sense.
However, inspecting the proof of \cref{thm:equiv-flat-path} we extract a more general statement that applies even to cohesive variables $x,y:\flat A$, namely
\[ (x=y) \simeq (\flet x^\flat := u in (\flet y^\flat := v in \flat(x=y))). \]

\cref{thm:equiv-flat-path} has many useful consequences.
For instance, it implies the functoriality of $\flat$ on homotopies.
That is, suppose that $f,g::A\to B$ and $H::f\sim g$ is a crisp homotopy, i.e.\ $H ::\prd{x:A} f(x)=g(x)$; we would like to conclude $\flat f \sim \flat g$.
Introducing $x:\flat A$ and using $\flat$-induction, this reduces to saying that for any $u::A$ we have $f(u)^\flat = g(u)^\flat$.
This \emph{looks} like it should be immediate from $H(u) :: f(u) = g(u)$, but since $(\blank)^\flat$ is not a \emph{function}, we cannot write ``$\ap_{(\blank)^\flat}(H(u))$''.
Instead we use the following.

\begin{cor}\label{thm:path-flat}
  For any $u,v::B$ and $p::u=v$, we have $p^{=\flat} : u^\flat = v^\flat$ such that $(\refl_u)^{=\flat} = \refl_{u^\flat}$ and $\ap_{\flatf}(p^{=\flat}) = p$.
\end{cor}
\begin{proof}
  We apply the $\decode$ from \cref{thm:equiv-flat-path} to $p^\flat:\flat(u=v)$.
  The first equality $(\refl_u)^{=\flat} = \refl_{u^\flat}$ follows from the definition of $\decode$.
  For the second, we use crisp Id-induction on $p$, and observe that
  \[ \ap_{\flatf}((\refl_u)^{=\flat}) = \ap_{\flatf}(\refl_{u^\flat}) = \refl_u \]
  using the first equality.
\end{proof}

Now we can define the action of $\flat$ on a homotopy $H$ by
\[ \flat H(x) \defeq (\flet u^\flat := x in H(u)^{=\flat}). \]
Functoriality on homotopies also implies that the functor $\flat$ preserves (crisp) equivalences: if $e::A\simeq B$, then we have $\flat e : \flat A \simeq \flat B$.

We can also regard \cref{thm:equiv-flat-path} as the corollary of \cref{thm:crisp-J} saying analogously to \cref{thm:flat-pres-coprod}
that ``$\flat$ preserves identity types''.
It implies analogous corollaries of \cref{thm:crisp-coeq-ind,thm:crisp-trunc-ind}.

\begin{cor}\label{thm:flat-pres-coeq}
  For any $A,B::\type$ and $f,g::A\to B$, we have $\flat(\coeq(f,g)) \simeq \coeq(\flat f, \flat g)$.
\end{cor}
\begin{proof}
  From left to right, we apply $\flat$-induction, crisp $\coeq$-induction, and then $(\blank)^\flat$ and the structure maps of $\coeq(\flat f, \flat g)$.
  From right to left, we apply ordinary $\coeq$-induction first, and then $\flat q : \flat B \to \flat (\coeq(f,g))$; finally we need to show that for $x:\flat A$ we have $\flat q (\flat f(x)) = \flat q (\flat g(x))$; but this follows from the above functoriality of $\flat$ on homotopies.
  That the composites are identities is proven similarly.
\end{proof}

\begin{cor}\label{thm:flat-pres-hocirc}
  $\flat\hocirc \simeq \hocirc$.\qed
\end{cor}

%For truncations, we need to first know that $\flat$ preserves $n$-types:

\begin{thm}\label{thm:flat-set}\label{thm:flat-ntype}
  For any $n::\N$ and $A::\type$, if $A$ is crisply an $n$-type, so is $\flat A$.
\end{thm}
\begin{proof}
  We do crisp induction (as in \cref{sec:crisp-induction}) over $n::\N$, beginning with $n=-2$.
  If $A$ is crisply contractible, we have some $c::A$, hence also $c^\flat : \flat A$.
  Now any $x:\flat A$ can be destructed as $u^\flat$ for $u::A$, and by \cref{thm:equiv-flat-path} $(u^\flat = a^\flat) \simeq \flat (u=a)$.
  But $A$ is crisply contractible, so we have $p::u=a$, hence $p^\flat : \flat (u=a)$.

  For the induction step, suppose $A$ is crisply an $(n+1)$-type.
  Then $x,y:\flat A$ can be destructed as $u^\flat$ and $v^\flat$ for $u,v::A$, and we have $(u^\flat = v^\flat) \simeq \flat (u=v)$.
  But $u=v$ is crisply an $n$-type; hence by the inductive hypothesis so is $\flat (u=v)$.
\end{proof}

\begin{cor}\label{thm:flat-pres-trunc}
  For $n::\N$ and $A::\type$, we have $\flat\trunc n A \simeq \trunc n {\flat A}$.
\end{cor}
\begin{proof}
  From left to right, we apply $\flat$-induction and crisp truncation-induction, followed by $(\blank)^\flat$.
  From right to left, we apply ordinary truncation-induction (using \cref{thm:flat-ntype}) and then functoriality of $\flat$ on $\tprojf{n}$.
  The composites on either side are proven to be identities similarly.
\end{proof}

Thirdly, \cref{thm:equiv-flat-path} implies $\flat$ is left exact, using an analogue of \cref{thm:sharp-sigma}:

\begin{lem}\label{thm:flat-pres-sigma}
  For any $A::\type$ and $B::A\to\type$, we have
  \[ \flat\left(\sm{x:A} B(x) \right) \simeq \sm{x:\flat A} \big(\flet u^\flat := x in \flat B(u)\big).\]
\end{lem}
% Note that the right-hand side makes sense: since $u$ and $B$ are both crisp, so is $B(u)$, so we can apply $\flat$ to it.
\begin{proof}
  From left to right, we first apply $\flat$-induction, then destruct to get a crisp $u::A$ and $v::B(u)$.
  Then we can form $u^\flat$ and $v^\flat : \flat B(u)$, and $(u^\flat,v^\flat)$ belongs to the right-hand side by the computation rule for $\flat$.
  From right to left, we first destruct into $x:\flat A$ and $y:(\flet u^\flat := x in \flat B(u))$.
  Then we apply $\flat$-induction to $x$ to get $u::A$, reducing the type of $y$ to $\flat B(u)$.
  Next we apply $\flat$-induction to $y$ to get $v::B(u)$, so we can form $(u,v)^\flat:\flat\left(\sm{x:A} B(x) \right)$.
  We leave it to the reader to check that these are inverses.
\end{proof}

\begin{thm}\label{thm:flat-pres-prod}
  For any $A,B::\type$ we have $\flat(A\times B) \simeq \flat A \times \flat B$.
\end{thm}
\begin{proof}
  If $B$ is independent of $u:A$, then $(\flet u^\flat := x in \flat B) = \flat B$ (this can be regarded as an instance of the uniqueness principle for $\flat$).
\end{proof}

\begin{thm}\label{thm:flat-pres-pullbacks}
  For any $A,B,C::\type$ with $f::A\to C$ and $g::B\to C$, we have $\flat(A\times_C B) \simeq \flat A \times_{\flat C} \flat B$.
\end{thm}
\begin{proof}
  We compute
  \begin{align*}
    \flat(A\times_C B)
    &\jdeq \flat\Big(\tsm{x:A}{y:B} (f(x)=g(y))\Big)\\
    % &\simeq \tsm{x:\flat A} \flet u^\flat := x in \flat\Big(\tsm{y:B} (f(u)=g(y))\Big)\\
    &\simeq \tsm{x:\flat A} \flet u^\flat := x in \tsm{y:\flat B} \flet v^\flat := y in \flat (f(u)=g(v))\\
    % &\simeq \tsm{x:\flat A}{y:\flat B} \flet u^\flat := x in \flet v^\flat := y in \flat (f(u)=g(v))\\
    % &\simeq \tsm{x:\flat A}{y:\flat B} \flat ((\flet u^\flat := x in f(u)) = (\flet v^\flat := y in g(v)))\\
    &\simeq \tsm{x:\flat A}{y:\flat B} \flat (f(x_\flat) = g(y_\flat))\\
    &\simeq \tsm{x:\flat A}{y:\flat B} \flat (\flat f(x)_\flat = \flat g(y)_\flat)\\
    &\simeq \tsm{x:\flat A}{y:\flat B} (\flat f(x) = \flat g(y))\\
    &\jdeq \flat A \times_{\flat C} \flat B.
  \end{align*}
  The first step is the definition of pullbacks and the second is \cref{thm:flat-pres-sigma}.
  The third step combines several commuting conversions and the fourth is naturality of $\flatf$.
  The fifth step is \cref{thm:equiv-path-flat} and the last is the definition of pullbacks.
\end{proof}

For completeness, we should also record:

\begin{thm}\label{thm:flat-pres-unit}
  $\flat\unit \simeq \unit$.
\end{thm}
\begin{proof}
  It suffices to show that $\flat\unit$ is contractible.
  Since we have $\ttt^\flat : \flat\unit$, it suffices to show that $\flat\unit$ is a \mprop.
  Now if $u,v:\flat\unit$, we can destruct them as $x^\flat$ and $y^\flat$ for $x::\unit$ and $y::\unit$.
  Since $\unit$ is (crisply) contractible we have $p::x=y$, whence $p^\flat : \flat(x=y)$.
  Thus, by \cref{thm:equiv-flat-path}, we have $x^\flat=y^\flat$.
\end{proof}

\subsection{Discrete types}
\label{sec:discrete-types}

We now move on to study discrete types, the duals of the codiscrete types from \cref{sec:codisc}.
The definition should come as no surprise.

\begin{defn}
  A crisp type $A::\type$ is \textbf{discrete} if $\flatf : \flat A \to A$ is an equivalence.
\end{defn}

\begin{rmk}
  To say that a type is \emph{discrete} is very different from saying that it is \emph{crisp}.
  A discrete type is one whose \emph{points} are ``completely unconnected'' topologically, whereas
  a crisp type is one for which we can ignore the topology relating it to \emph{other types}.
  Discreteness is an ``intrinsic'' aspect of a type, whereas crispness is a more ``syntactic'' one indicating what we are doing with that type at the moment.
  Roughly speaking, a type is discrete if we are free to assume that its \emph{elements} are crisp.

  At this point in our development, a type must be crisp before we can even ask whether it is discrete.
  However, in \cref{sec:real-cohesion} we will introduce a more general notion of discreteness that removes this restriction.

  Note that a (crisp and) discrete type $A::\type$ might also be \emph{crisply discrete}, i.e.\ we might have a \emph{crisp} witness $d :: \isdisc(A)$.
  This distinction is admittedly somewhat confusing.
  In practice, pretty much every discrete type is crisply discrete, and whenever we assume a type to be discrete it must be crisply discrete.
  Moreover, in \cref{sec:real-cohesion} we will introduce a further axiom ensuring that \emph{every} discrete type is crisply discrete.
\end{rmk}

Our first goal is to show that the discrete types form a \emph{coreflective} subuniverse in an appropriate sense, with $\flat$ as the coreflector.
The core of that result is that maps $\flat B\to A$ are equivalent to maps $\flat B \to \flat A$; this should appear sensible since any function out of a discrete space is continuous.
We do, however, have to apply a further $\flat$ on both sides, since the type $\flat B\to A$ may inherit a nontrivial topology from $A$.

\begin{thm}\label{thm:discrete-coreflective-0}
  For any $A,B::\type$, postcomposition with $\flatf : \flat A \to A$ induces an equivalence
  \begin{equation}
    \flat (\flat B\to \flat A) \simeq \flat (\flat B\to A).\label{eq:flat-corefl-0}
  \end{equation}
\end{thm}
\begin{proof}
  To be precise, the left-to-right map in~\eqref{eq:flat-corefl-0} is defined as follows: given $h : \flat (\flat B\to \flat A)$, we first destruct it into $k::\flat B \to \flat A$.
  Now since we have only crisp variables, it suffices to construct a function $\flat B \to A$, for which we have $\lam{x:\flat B} k(x)_\flat$.
  In symbols, from $h : \flat (\flat B\to \flat A)$ we construct
  \[ \flet k^\flat := h in (\lam{x:\flat B} k(x)_\flat)^\flat \quad: \flat (\flat B\to A). \]

  To define the right-to-left map in~\eqref{eq:flat-corefl-0}, suppose given an element of $\flat (\flat B\to A)$, which we immediately destruct into $f::\flat B \to A$.
  To construct an element of $\flat (\flat B\to \flat A)$, since we have only crisp variables it suffices to construct a function $\flat B \to \flat A$.
  Thus, let $x:\flat B$, which we immediately destruct into $u::B$.
  But now again it suffices to construct an element of $A$ itself, for which we can take $f(u^\flat)$.
  In symbols, from $g : \flat (\flat B\to A)$ we construct
  \[ \flet f^\flat := g in (\lam{x:\flat B} \flet u^\flat := x in f(u^\flat)^\flat)^\flat \quad: \flat (\flat B \to \flat A). \]
  It is tempting to think that the uniqueness principle for $\flat$ (\cref{thm:flat-eta}) should simplify ``$\flet u^\flat := x in f(u^\flat)^\flat$'' to $f(x)^\flat$.
  However, this rule does not apply because we cannot write $f(x)^\flat$, since $x$ is not a crisp variable.

  In \cref{fig:disc-corefl} we verify that both round-trip composites are the identity.
  We use the fact that since $\flatf$ is defined with $\mathsf{let}$, the commuting conversions allow us to push it inside another $\mathsf{let}$, along with function extensionality to allow this (and the uniqueness principle) to happen inside a $\lambda$-abstraction.
\end{proof}

\begin{figure}
  \centering
    \begin{align*}
    &\quad \flet k^\flat := (\flet f^\flat := g in (\lam{x} \flet u^\flat := x in f(u^\flat)^\flat)^\flat) in (\lam{x} k(x)_\flat)^\flat\\
    &= \flet f^\flat := g in (\flet k^\flat := (\lam{x} \flet u^\flat := x in f(u^\flat)^\flat)^\flat in (\lam{x} k(x)_\flat)^\flat)\\
    &\jdeq \flet f^\flat := g in (\lam{x} (\flet u^\flat := x in f(u^\flat)^\flat)_\flat)^\flat\\
    % &\jdeq \flet f^\flat := g in (\lam{x} (\flet v^\flat := (\flet u^\flat := x in f(u^\flat)^\flat) in v))^\flat\\
    % &\jdeq \flet f^\flat := g in (\lam{x} (\flet u^\flat := x in (\flet v^\flat := f(u^\flat)^\flat in v)))^\flat\\
    &= \flet f^\flat := g in (\lam{x} \flet u^\flat := x in f(u^\flat)^\flat{}_\flat)^\flat\\
    &\jdeq \flet f^\flat := g in (\lam{x} \flet u^\flat := x in f(u^\flat))^\flat\\
    &= \flet f^\flat := g in (\lam{x} f(x))^\flat\\
    &\jdeq \flet f^\flat := g in f^\flat\\
    &= g.
  \end{align*}
  \begin{align*}
    &\quad \flet f^\flat := (\flet k^\flat := h in (\lam{x} k(x)_\flat)^\flat) in (\lam{x} \flet u^\flat := x in f(u^\flat)^\flat)^\flat \\
    &= \flet k^\flat := h in (\flet f^\flat := (\lam{x} k(x)_\flat)^\flat in (\lam{x} \flet u^\flat := x in f(u^\flat)^\flat)^\flat) \\
    % &\jdeq \flet k^\flat := h in (\lam{x} \flet u^\flat := x in (\lam{x} k(x)_\flat)(u^\flat)^\flat)^\flat \\
    &\jdeq \flet k^\flat := h in (\lam{x} \flet u^\flat := x in k(u^\flat)_\flat{}^\flat)^\flat \\
    &= \flet k^\flat := h in (\lam{x} \flet u^\flat := x in k(u^\flat))^\flat \\
    &= \flet k^\flat := h in (\lam{x} k(x))^\flat \\
    &\jdeq \flet k^\flat := h in k^\flat \\
    &= h.
  \end{align*}
  \caption{The round-trip composites in \cref{thm:discrete-coreflective-0}}
  \label{fig:disc-corefl}
\end{figure}

\begin{cor}\label{thm:discrete-coreflective}
  If $B::\type$ is crisply discrete, then for any $A::\type$, postcomposition with $\flatf : \flat A \to A$ induces an equivalence
  \begin{equation}
    \flat (B\to \flat A) \simeq \flat (B\to A).\label{eq:flat-corefl}
  \end{equation}
\end{cor}
\begin{proof}
  As remarked after \cref{thm:path-flat}, the functor $\flat$ preserves equivalences, and precomposition with the equivalence $\flatf : \flat B \simeq B$ is an equivalence on function types.
  Thus, this follows immediately from \cref{thm:discrete-coreflective-0}.
\end{proof}

More generally, we have a dependent universal property:

\begin{thm}\label{thm:discrete-coreflective-D}
  For $B::\type$ and $A::\flat B\to\type$, fiberwise postcomposition with $\flatf : \flat A(x) \to A(x)$ induces an equivalence
  \begin{equation}\label{eq:flat-corefl-0D}
    \flat \left(\prd{x:\flat B} (\flet u^\flat := x in \flat(A(u^\flat))) \right) \simeq
    \flat \left(\prd{x:\flat B} A(x) \right)
  \end{equation}
  Therefore, if $B::\type$ is crisply discrete and $A::B\to\type$, fiberwise postcomposition with $\flatf : \flat A(x) \to A(x)$ induces an equivalence
  \begin{equation}\label{eq:flat-corefl-D}
    \flat \left(\prd{x:B} (\flet u^\flat := s(x) in \flat(A(u))) \right) \simeq
    \flat \left(\prd{x:B} A(x) \right)
  \end{equation}
  where $s:B\to \flat B$ is the inverse of $\flatf:\flat B \to B$.
\end{thm}
\begin{proof}
  The left-to-right map in~\eqref{eq:flat-corefl-0D} is defined like that in~\eqref{eq:flat-corefl-0}, taking $h$ %$h:\flat \left(\prd{x:\flat B} (\flet u^\flat := x in \flat(A(u^\flat))) \right)$
  to
  \[ \flet k^\flat := h in (\lam{x:\flat B} \flet u^\flat := x in k(u^\flat)_\flat)^\flat \quad: \flat \left(\prd{x:\flat B} A(x) \right). \]
  Its inverse looks exactly the same as the inverse of~\eqref{eq:flat-corefl-0}, though now all the functions are dependent:
  \[ \flet f^\flat := g in (\lam{x:\flat B} \flet u^\flat := x in f(u^\flat)^\flat)^\flat \quad: \flat (\flat B \to \flat A). \]
  The calculations in \cref{fig:disc-corefl} generalize immediately, and~\eqref{eq:flat-corefl-D} is again obtained by transporting along the equivalence $\flat B \simeq B$.
\end{proof}

\begin{lem}
  For $A::\type$ to be crisply discrete, it suffices that $\flatf : \flat A \to A$ have a crisp section.
\end{lem}
\begin{proof}
  First we should clarify what is meant by a ``crisp section''.
  The type of sections of $\flatf$ is $\sm{s:A\to \flat A} \prd{a:A} s(a)_\flat = a$.
  Having a crisp element of this type means having a crisp $s::A\to \flat A$ and also a crisp homotopy $H:: \prd{a:A} s(a)_\flat = a$.

  It remains to show that $s(x_\flat) = x$ for all $x:\flat A$.
  By $\flat$-induction, it suffices to show that for any $u::A$ we have $s(u^\flat{}_\flat) = u^\flat$.
  But $u^\flat{}_\flat\jdeq u$, so what we must show is $s(u) = u^\flat$.
  Now $s(u):\flat A$ is crisp, so we also have $s(u) = s(u)_\flat{}^\flat$; thus it suffices to show $s(u)_\flat{}^\flat = u^\flat$.
  But now \cref{thm:path-flat} tells us that this follows from $H(u) :: s(u)_\flat = u$.
\end{proof}

\begin{thm}\label{thm:discrete-flat}
  For any $A::\type$, the type $\flat A$ is crisply discrete.
\end{thm}
\begin{proof}
  We define $s:\flat A \to \flat\flat A$ by $s(x) \defeq (\flet u^\flat := x in u^\flat{}^\flat)$.
  Applying $\flatf$ to $s(x)$, we get
  \begin{equation*}
    \quad (\flet u^\flat := x in u^\flat{}^\flat)_\flat
    = (\flet u^\flat := x in u^\flat{}^\flat{}_\flat)
    \jdeq (\flet u^\flat := x in u^\flat)
    = x
  \end{equation*}
  using a commuting conversion and the uniqueness rule.
\end{proof}

Thus $\flat$ coreflects into the discrete types, dually to how $\sharp$ reflects into the codiscrete types.
However, there are several ways in which the duality breaks down.
Firstly there is the outer $\flat$ on the equivalence in \cref{thm:discrete-coreflective}.

Secondly, we have seen that $\sharp$ and $\flat$ are \emph{both} left exact, whereas a strict duality would interchange left exactness with right exactness.
Note that exactness of reflectors and coreflectors can be a bit subtle and confusing.
For instance, a reflector is a left adjoint, so it always preserves colimits, i.e.\ takes colimits in the ambient category to colimits \emph{in the reflective subcategory} --- but these latter colimits will only coincide with those in the ambient category if the subcategory is closed under them.
This is not generally the case for $\sharp$: the coproduct of codiscrete spaces is not generally codiscrete.

Dually, a coreflector is a right adjoint, so it always takes limits in the ambient category to limits in the coreflective subcategory.
The content of \cref{thm:flat-pres-prod,thm:flat-pres-pullbacks,thm:flat-pres-unit} is thus that the discrete types \emph{are} closed under finite limits.
We record these facts:

\begin{thm}\label{thm:discrete-prod}\label{thm:path-discrete}\label{thm:pullback-discrete}
  If $A,B,C::\type$ are crisply discrete and $f::A\to C$ and $g::B\to C$ and $x,y::A$, then the following types are crisply discrete:
  \begin{mathpar}
    A\times B \and
    A\times_C B \and
    \unit \and
    (x=y)
  \end{mathpar}
\end{thm}
\begin{proof}
  For $A\times B$, we have $A\times B \simeq \flat A \times \flat B \simeq \flat(A\times B)$.
  The case of $A\times_c B$ is similar, and $\unit$ is exactly \cref{thm:flat-pres-unit}.
  For $x=y$, since equivalences lift to identity types, it suffices to show that if $x,y::\flat A$ then $(x=y)$ is discrete;
  but this follows from \cref{thm:discrete-flat,thm:equiv-path-flat}.
\end{proof}

More generally, we have:

\begin{thm}\label{thm:discrete-sigma-0}
  Suppose $A::\type$ is crisply discrete, and that $B::A\to\type$ is such that for every $x:\flat A$, we have $(\flet u^\flat := x in \isdisc(B(u)))$.
  Then $\sm{x:A} B(x)$ is crisply discrete.
\end{thm}
The hypothesis is morally ``$B(u)$ is discrete for every $u::A$'', but stating it that way would require this theorem to be only a ``derivable rule'' (see \cref{rmk:derivable-rule}), since we cannot say ``for every $u::A$'' in an ordinary hypothesis.
\begin{proof*}
  Using \cref{thm:flat-pres-sigma}, we have
  \begin{align*}
    \flat\left(\tsm{x:A} B(x) \right)
    &\simeq \left(\tsm{x:\flat A} \big(\flet u^\flat := x in \flat B(u)\big)\right)\\
    & \simeq \left(\tsm{x:\flat A} \big(\flet u^\flat := x in B(u)\big)\right)\\
    & \simeq \tsm{x:\flat A} B(x_\flat)\\
    & \simeq \tsm{u:A} B(u).\qedhere
  \end{align*}
\end{proof*}

On the other hand, the dual of left-exactness of $\sharp$ would be right-exactness of $\flat$, which is true.
In fact, as remarked in \cref{sec:crisp-induction}, this is essentially the content of the crisp induction principles.
We record these facts:

\begin{thm}\label{thm:empty-discrete}\label{thm:coprod-discrete}\label{thm:nat-discrete}\label{thm:coeq-discrete}\label{thm:s1-discrete}\label{thm:trunc-discrete}\label{thm:setcoeq-discrete}
  If $A,B::\type$ are crisply discrete and $f,g::A\to B$ and $n::\N$, then the following types are crisply discrete:
  \begin{mathpar}
    A+B \and \emptyset \and \N \and \hocirc \and \coeq(f,g)
    \and \trunc n A
  \end{mathpar}
  Combining the last two, set-coequalizers also preserve crisp discreteness.\qed
\end{thm}

\subsection{Adjointness of $\flat$ and $\sharp$}
\label{sec:adjo-flat-sharp}

To complete our study of $\flat$, we consider how it interacts with $\sharp$.
Our first observation is that $\flat$ and $\sharp$ ``eat each other''.

\begin{thm}\label{thm:co-disc-equiv-gamma}
  For any $A::\type$, there are natural equivalences
  \[ \sharp\flat A \simeq \sharp A \qquad\text{and}\qquad \flat A \simeq \flat \sharp A. \]
\end{thm}

This is sensible according to our topological intuition: $\flat$ and $\sharp$ are supposed to re-equip a space with the discrete or codiscrete topology, respectively, so it shouldn't matter whether we have already modified its topology in the other way.

\begin{proof*}
  The first equivalence is defined by
  \[ \lam{y:\sharp\flat A} y_\sharp{}_\flat{}^\sharp \qquad\text{and}\qquad \lam{x:\sharp A} x_\sharp{}^\flat{}^\sharp. \]
  The round-trip composite starting with $x:\sharp A$ yields
  \[ x_\sharp{}^\flat{}^\sharp{}_\sharp{}_\flat{}^\sharp
  \jdeq x_\sharp{}^\flat{}_\flat{}^\sharp
  \jdeq x_\sharp{}^\sharp
  \jdeq x
  \]
  while that starting with $y:\sharp\flat A$ yields
  \[ y_\sharp{}_\flat{}^\sharp{}_\sharp{}^\flat{}^\sharp
  \jdeq y_\sharp{}_\flat{}^\flat{}^\sharp
  = y_\sharp{}^\sharp
  \jdeq y.
  \]
  The second equivalence is defined by
  \[ \lam{x:\flat A} \flet u^\flat := x in u^\sharp{}^\flat \qquad\text{and}\qquad \lam{y:\flat\sharp A} \flet v^\flat := y in v_\sharp{}^\flat. \]
  The round-trip composite starting with $x:\flat A$ yields
  \begin{align*}
    &\quad \flet v^\flat := (\flet u^\flat := x in u^\sharp{}^\flat) in v_\sharp{}^\flat\\
    &= \flet u^\flat := x in (\flet v^\flat := u^\sharp{}^\flat in v_\sharp{}^\flat)\\
    &\jdeq \flet u^\flat := x in u^\sharp{}_\sharp{}^\flat\\
    &\jdeq \flet u^\flat := x in u^\flat\\
    &= u.
  \end{align*}
  And the round-trip composite starting with $y:\flat\sharp A$ yields
  \begin{align*}
    &\quad \flet u^\flat := (\flet v^\flat := y in v_\sharp{}^\flat) in u^\sharp{}^\flat\\
    &= \flet v^\flat := y in (\flet u^\flat := v_\sharp{}^\flat in u^\sharp{}^\flat)\\
    &\jdeq \flet v^\flat := y in v_\sharp{}^\sharp{}^\flat\\
    &\jdeq \flet v^\flat := y in v^\flat\\
    &= v.\qedhere
  \end{align*}
\end{proof*}

\begin{cor}\label{thm:co-disc-equiv}
  For any crisply discrete $A::\type$, the map $A\to \flat\sharp A$ is an equivalence.
  Dually, for any crisply codiscrete $A::\type$, the map $\sharp \flat A \to A$ is an equivalence.\qed
\end{cor}

We can express this more internally using the type $\codisc \defeq \tsm{A:\type} \iscodisc(A)$ from \cref{thm:codiscrete-universe} and an analogous universe of discrete types, defined as:
\[ \disc \defeq \sm{A:\flat\type} (\flet B^\flat := A in \isdisc(B)). \]

\begin{corua}\label{thm:disc-is-flat-codisc}
  $\disc\simeq \flat\codisc$; hence also $\sharp\disc\simeq \codisc$.
\end{corua}
\begin{proof}
  By \cref{thm:flat-pres-sigma}, we have
  \[ \flat\codisc \simeq \sm{A:\flat\type} (\flet B^\flat := A in \iscodisc(B)). \]
  The equivalence $\disc\simeq \flat\codisc$ now comes from the endomaps of $\flat\type$ induced by $\flat$ and $\sharp$.
  These are mutually inverse by \cref{thm:co-disc-equiv,thm:equiv-flat-path} and univalence.
\end{proof}

Since $\sharp$ and $\flat$ preserve $n$-types, this equivalence relativizes to the universes of $n$-types.
In particular, we have $\discprop\simeq \flat\codiscprop$, which uses only propositional univalence.

Furthermore, this equivalence respects all type formers, suitably reflected or coreflected.
For instance, since discrete and codiscrete types are closed under cartesian products, we have product operations $\disc\to\disc\to\disc$ and $\codisc\to\codisc\to\codisc$, and these commute with the equivalence of \cref{thm:disc-is-flat-codisc}.
Discrete objects need not be closed under function types, but we have an operation $\disc\to\disc\to\disc$ defined by
\[\lam{A}{B} \flet C^\flat := A in \flet D^\flat := B in \flat(C\to D) \]
which agrees with $\lam{A}{B} (A\to B):\codisc\to\codisc\to\codisc$ under the equivalence of \cref{thm:disc-is-flat-codisc}; and so on.

Topological intuition also suggests that since all functions between discrete or codiscrete spaces are continuous, both $\flat A \to \flat B$ and $\sharp A \to \sharp B$ ought to be the set of all (discontinuous) functions from $A$ to $B$.
These two spaces may themselves have different topologies, of course (and in fact they do), but if we make them both discrete or codiscrete we should be able to identify them.

\begin{thm}\label{thm:flat-sharp-eqv}
  For any $A,B::\type$, there is an equivalence
  \begin{equation}
    \flat(\flat A\to \flat B) \simeq \flat(\sharp A \to \sharp B).\label{eq:flat-sharp-hom-iso}
  \end{equation}
\end{thm}
\begin{proof*}
  On one hand, we have the functorial action of $\sharp$:
  \[\sharp \;:\; (\flat A \to \flat B) \to (\sharp\flat A \to \sharp\flat B).\]
  We can compose the result with the equivalences $\sharp\flat A \simeq \sharp A$ and $\sharp\flat B\simeq \sharp B$ on either side, and apply then the functor $\flat$, to obtain a map from left to right in~\eqref{eq:flat-sharp-hom-iso}.
  On the other hand, as remarked above, the functorial action of $\flat$ can be expressed (in this case) as
  \[ \flat \;:\; \flat(\sharp A \to \sharp B)\to \flat(\flat\sharp A \to \flat\sharp B) \]
  and we can again compose on either side with the equivalences $\flat\sharp A \simeq \flat A$ and $\flat\sharp B \simeq \flat B$.

  More explicitly, given $f:\flat(\flat A\to \flat B)$ we can write
  \[ \flet h^\flat := f in (\lam{x:\sharp A} h(x_\sharp{}^\flat)_\flat{}^\sharp)^\flat \;: \flat(\sharp A \to \sharp B) \]
  and given $g:\flat(\sharp A \to \sharp B)$ we can write
  \[ \flet k^\flat := g in (\lam{x:\flat A} \flet u^\flat := x in k(u^\sharp)_\sharp{}^\flat )^\flat \;: \flat(\flat A\to \flat B).\]
  Omitting the commuting conversions and reductions that eliminate the variables $h$ and $k$, as in the proof of \cref{thm:discrete-coreflective-0}, the computations that show these to be inverses are
  \begin{gather*}
  (\flet u^\flat := x_\sharp{}^\flat in k(u^\sharp)_\sharp{}^\flat)_\flat{}^\sharp
  \jdeq k(x_\sharp{}^\sharp)_\sharp{}^\flat{}_\flat{}^\sharp
  \jdeq k(x_\sharp{}^\sharp)_\sharp{}^\sharp
  \jdeq k(x)_\sharp{}^\sharp
  \jdeq k(x)
  \\
    (\flet u^\flat := x in h(u^\sharp{}_\sharp{}^\flat)_\flat{}^\sharp{}_\sharp{}^\flat)
    \jdeq (\flet u^\flat := x in h(u^\flat)_\flat{}^\flat)
    = (\flet u^\flat := x in h(u^\flat))
    = h(x).\qedhere
  \end{gather*}
\end{proof*}

\begin{cor}\label{thm:flat-sharp-adj}
  For any $A,B::\type$, there is a natural equivalence
  \[ \flat(\flat A \to B) \simeq \flat(A\to \sharp B). \]
  In other words, $\flat$ is ``crisply left adjoint'' to $\sharp$.
\end{cor}
\begin{proof*}
  We have the following chain of equivalences:
  \[ \flat(\flat A \to B) \simeq
  \flat(\flat A \to \flat B) \simeq
  \flat(\sharp A \to \sharp B) \simeq
  \flat(A \to \sharp B).\qedhere
  \]
\end{proof*}

In fact, more generally we can say:

\begin{thm}\label{thm:flat-sharp-adj-dep}
  For any $A::\type$ and $B::A\to\type$, there is a natural equivalence
  \[ \flat \left(\tprd{u:\flat A} B(u_\flat)\right) \simeq \flat\left(\tprd{x:A} \sharp B(x)\right). \]
\end{thm}
\begin{proof*}
  Given $f:\flat \left(\prd{u:\flat A} B(u_\flat)\right)$, we define $g:\flat\left(\prd{x:A} \sharp B(x)\right)$ by
  \[ \flet h^\flat := f in (\lam{x:A} h(x^\flat)^\sharp)^\flat \]
  and given $g$, we define $f$ by
  \[ \flet k^\flat := g in (\lam{u:\flat A} \flet x^\flat := u in k(x)_\sharp )^\flat. \]
  Omitting the outer commuting conversions and reductions again, the computations that show these to be inverses are
  \[ (\flet x^\flat := u in h(x^\flat)^\sharp{}_\sharp)
  \jdeq (\flet x^\flat := u in h(x^\flat))
  \jdeq h(u) \]
  % Starting from $f$, the round-trip composite is
  % \begin{align*}
  %   &\quad \flet k^\flat := (\flet h^\flat := f in (\lam{x:A} h(x^\flat)^\sharp)^\flat) in (\lam{u:\flat A} \flet x^\flat := u in k(x)_\sharp )^\flat\\
  %   &= \flet h^\flat := f in (\flet k^\flat := (\lam{x:A} h(x^\flat)^\sharp)^\flat in (\lam{u:\flat A} \flet x^\flat := u in k(x)_\sharp )^\flat)\\
  %   &\jdeq \flet h^\flat := f in (\lam{u:\flat A} \flet x^\flat := u in (\lam{x:A} h(x^\flat)^\sharp)(x)_\sharp )^\flat\\
  %   &\jdeq \flet h^\flat := f in (\lam{u:\flat A} \flet x^\flat := u in h(x^\flat)^\sharp{}_\sharp )^\flat\\
  %   &\jdeq \flet h^\flat := f in (\lam{u:\flat A} \flet x^\flat := u in h(x^\flat) )^\flat\\
  %   &= \flet h^\flat := f in (\lam{u:\flat A} h(u))^\flat\\
  %   &\jdeq \flet h^\flat := f in h^\flat\\
  %   &= f.
  % \end{align*}
  and
  \[ (\flet x^\flat := x^\flat in k(x)_\sharp)^\sharp
  \jdeq k(x)_\sharp{}^\sharp
  \jdeq k(x). \qedhere \]
  % and starting from $g$, the round-trip composite is
  % \begin{align*}
  %   &\quad \flet h^\flat := (\flet k^\flat := g in (\lam{u:\flat A} \flet x^\flat := u in k(x)_\sharp )^\flat) in (\lam{x:A} h(x^\flat)^\sharp)^\flat\\
  %   &= \flet k^\flat := g in (\flet h^\flat := (\lam{u:\flat A} \flet x^\flat := u in k(x)_\sharp )^\flat in (\lam{x:A} h(x^\flat)^\sharp)^\flat)\\
  %   &\jdeq \flet k^\flat := g in (\lam{x:A} (\lam{u:\flat A} \flet x^\flat := u in k(x)_\sharp )(x^\flat)^\sharp)^\flat\\
  %   &\jdeq \flet k^\flat := g in (\lam{x:A} (\flet x^\flat := x^\flat in k(x)_\sharp)^\sharp)^\flat\\
  %   &\jdeq \flet k^\flat := g in (\lam{x:A} k(x)_\sharp{}^\sharp)^\flat\\
  %   &\jdeq \flet k^\flat := g in (\lam{x:A} k(x))^\flat\\
  %   &\jdeq \flet k^\flat := g in k^\flat\\
  %   &= g.\qedhere
  % \end{align*}
\end{proof*}

Hence we have both a reflective subcategory (the codiscrete types) and a coreflective subcategory (the discrete types), which by \cref{thm:co-disc-equiv} are abstractly equivalent via the reflection and coreflection.
Moreover, \cref{thm:co-disc-equiv-gamma} tells us that modulo this equivalence, the reflector $\sharp$ agrees with the coreflector $\flat$, which is sensible if we consider that both take the same underlying set (or \oo-groupoid) and equip it with a new topology.

\begin{rmk}\label{rmk:flat-lem}
  We now have another formulation of the law of excluded middle.
  When assumed as an axiom, \cref{thm:sharp-lem} is a \emph{crisp} element of $\prd{P:\prop} \sharp(P\vee\neg P)$, and hence yields an element of $\flat\left(\prd{P:\prop} \sharp(P\vee\neg P)\right)$.
  Thus, by \cref{thm:flat-sharp-adj-dep} we can obtain a (crisp) element of
  \begin{equation}
    \prd{P:\flat \prop} (P_\flat \vee \neg P_\flat),\label{eq:flat-lem}
  \end{equation}
  a statement which we might call the \emph{flat law of excluded middle}.
\end{rmk}

\begin{rmk}\label{rmk:flat-overall}
  The overall conclusion of \crefrange{sec:flat}{sec:disc} can be summarized as ``$\flat$ is a coreflector, and the images of $\flat$ and $\sharp$ are equivalent by an equivalence that identifies $\flat$ with $\sharp$.''
  But in contrast to the analogous statement about $\sharp$ in \cref{rmk:sharp-overall}, this cannot be stated purely cohesively, because of the no-go \cref{thm:no-go}.
  In \textcite{ss:qgftchtt}, using ordinary type theory without crispness, we stated this axiomatically with $\flat$ as an operation on $\sharp\type$; here we instead derive it from the rules governing $\flat$ and $\sharp$ in relation to crisp hypotheses.
\end{rmk}

\addtocontents{toc}{\protect\setcounter{tocdepth}{2}}
\subsection{The axiom of choice revisited}
\label{sec:axiom-choice-2}

Let us now return to the sharp axiom of choice, \ref{ax:shac}.
Restated, this says that for any set $A$ and type family $P:A\to \type$, we have
\begin{equation*}
  \left(\prd{x:A} \sharp\brck{P(x)}\right) \to \sharp\brck{\prd{x:A} \sharp P(x)}.
\end{equation*}
Using discreteness, we can derive from this a statement less encumbered by $\sharp$s.

\begin{thmac}[The discrete axiom of choice]\label{thm:discrete-ac}
  Suppose $A::\type$ is a crisply discrete set, that $P::A\to \type$ is also crisp, and that $\prd{x:A} \brck{P(x)}$ crisply.
  Then $\brck{\prd{x:A} P(x)}$.
\end{thmac}
\begin{proof}
  The hypothesis $\prd{x:A} \brck{P(x)}$ is stronger than that of the sharp axiom of choice, so we can apply the latter, obtaining $\sharp\brck{\prd{x:A} \sharp P(x)}$.
  Since all the parameters are crisp, we can apply $(\blank)_\sharp$ to get $\brck{\prd{x:A} \sharp P(x)}$; whereas what we want is $\brck{\prd{x:A} P(x)}$.
  By functoriality, we can remove the $\brck{\blank}$ from both.
  Now by \cref{thm:flat-sharp-adj-dep}, we get $\prd{u:\flat A} P(u_\flat)$; but $A$ is discrete, so this gives $\prd{x:A} P(x)$ as desired.
\end{proof}

\begin{corac}[The crisp countable axiom of choice]\label{thm:discrete-countable-ac}
  \makeatletter\def\@currentlabel{Axiom AC$_{\mathbb{N}}$}\makeatother\label{ax:cc}
  Suppose $P::\N \to \type$ is crisp and that $\prd{n:\N} \brck{P(n)}$ crisply.
  Then $\brck{\prd{n:\N} P(n)}$.\qed
\end{corac}

We will refer to this statement as \ref{ax:cc}.
As mentioned previously, later on we will assume an axiom ensuring that all \mprop{}s are discrete; thus the crispness of $\prd{n:\N} \brck{P(n)}$ will become automatic.
However, the crispness of $P$ is a real restriction that makes this significantly weaker than the ordinary countable axiom of choice.

For example, let $\R_C$ denote the Cauchy real numbers and $\R$ the Dedekind real numbers.\footnote{For constructive definitions of $\R_C$ and $\R$, see e.g. \textcite[Chapter 11]{hottbook} or \textcite[\S D4.7]{ptj:elephant}.
At the moment, by the ``Cauchy reals'' we mean a simple quotient of the set of Cauchy sequences or approximations; in \cref{thm:RC-cauchy-complete} we will show that under suitable axioms, these coincide with the fancier Cauchy reals of \textcite[\S11.3]{hottbook}.}
There is an injection $i:\R_C \to \R$ which is not constructively an isomorphism.
It is well-known that ordinary countable choice implies that it \emph{is} an isomorphism, i.e.\ that every Dedekind real is a Cauchy real.
But with our restricted \ref{ax:cc} (or \ref{ax:lem}), all we can get is the following.

\begin{thmccorlem}\label{thm:crisp-RD-RC}
  Any \emph{crisp} Dedekind real $x::\R$ is a Cauchy real.
\end{thmccorlem}
\begin{proof}
  By \textcite[Lemma 11.4.1]{hottbook}, a Dedekind real $x:\R$ is a Cauchy real iff there exists $c : \prd{q, r : \Q} (q < r) \to (q < x) + (x < r)$.

  First suppose \ref{ax:lem}.
  Then since $x$ is crisp, and $\Q$ and $(q<r)$ are discrete, for any $q,r$ with $q<r$ we have $(q<x) + \neg(q<x)$.
  This implies $(q < x) + (x < r)$ as in \textcite[Corollary 11.4.3]{hottbook}.

  Now suppose \ref{ax:cc}.
  Then since $x$ is crisp, so is the family $P::S \to \type$ defined in \textcite[Corollary 11.4.3]{hottbook} as
  \[P(q,r) \defeq \sm{b:\bool} (b = \bfalse \to q < x) \land (b = \btrue \to x < r),\]
  where $S \defeq \{ (q,r) : \Q \times \Q \mid q < r \}$ is equivalent to $\N$.
  Thus, we can get a choice function and extract $c$ in the same way.
\end{proof}

\begin{corccorlem}\label{thm:RC-RD-points}
  The induced maps $\flat\R_C \to \flat \R$ and $\sharp\R_C \to\sharp\R$ are equivalences.
\end{corccorlem}
\begin{proof}
  By \cref{thm:co-disc-equiv-gamma}, the two statements are equivalent, so it suffices to prove the second.
  By \cref{thm:emb-sharp-eqv}, it will suffice to show $\prd{x:\R} \sharp \fib_i(x)$.
  But by \cref{thm:flat-sharp-adj-dep}, this follows from \cref{thm:crisp-RD-RC}.
\end{proof}

In other words, $\R_C$ and $\R$ have the same points, but perhaps different topologies.
It is natural to wonder what those topologies are!
We will return to this in \cref{sec:reals,sec:continuity}.

Another important question to ask is what their points \emph{are}.
From the perspective of the cohesive part of the type theory, this is a meaningless question: their points are, by definition, the real numbers (Cauchy or Dedekind, respectively).
However, recall from \cref{sec:codisc} that we can also construct ``the set of real numbers'' entirely in the codiscrete world, yielding the type $\R'$ shown in \cref{fig:R} % on page~\pageref{fig:R}.
and reproduced in \cref{fig:Ragain}.
Thus, we can ask whether the set of \emph{points} of the \emph{space} of real numbers $\R$
% (which, by \cref{thm:RC-RD-points}, is the same as the set of points of $\R_C$, namely $\sharp \R_C \simeq \sharp \R$)
is the same as this ``codiscrete set of real numbers'', i.e.\ is ``the space of real numbers'' $\R$ really a ``topology'' on the set of real numbers?
The answer is yes; this internalizes~\textcite[C3.6.11]{ptj:elephant}.

\begin{thm}\label{thm:sharp-R}
  $\flat\R'\simeq \flat \R$, and hence $\R'\simeq \sharp \R$.
\end{thm}
\begin{figure}
  \centering
  \begin{align*}
  % \R &\defeq \tsm{L,U:\Q \to \prop}
  %  \brck{\tsm{q:\Q} L(q)}\times \brck{\tsm{r:\Q} U(r)}\\
  % &\land \left(\tprd{q:\Q} L(q) \leftrightarrow \brck{\tsm{r : \Q} (q < r) \times L(r)}\right)\\
  % &\land \left(\tprd{r:\Q} U(r) \leftrightarrow \brck{\tsm{q : \Q} U(q) \times (q < r)}\right)\\
  % &\land \left(\tprd{q:\Q} (L(q) \times U(q)) \to \emptyset\right)
  % \land \left(\tprd{q,r:\Q} (q < r) \rightarrow \brck{L(q) + U(r)}\right)\\
  \R' &\defeq \tsm{L,U:\sharp\Q \to \codiscprop}\\
  &\qquad \sharp\brck{\tsm{q:\sharp\Q} L(q)}\times \sharp\brck{\tsm{r:\sharp\Q} U(r)}\\
  &\qquad\times \left(\tprd{q:\sharp\Q} L(q) \leftrightarrow \sharp\brck{\tsm{r : \sharp\Q} (q <' r) \times L(r)}\right)\\
  &\qquad\times \left(\tprd{r:\sharp\Q} U(r) \leftrightarrow \sharp\brck{\tsm{q : \sharp\Q} U(q) \times (q <' r)}\right)\\
  &\qquad\times \left(\tprd{q:\sharp\Q} (L(q) \times U(q)) \to \sharp\emptyset \right)
    \times \left(\tprd{q,r:\sharp\Q} (q <' r) \rightarrow \sharp\brck{\sharp(L(q) + U(r))}\right)
    \\
  &\simeq \tsm{L,U:\Q \to \codiscprop}\\
  &\qquad \sharp\brck{\tsm{q:\sharp\Q} L'(q)}\times \sharp\brck{\tsm{r:\sharp\Q} U'(r)}\\
  &\qquad\times \left(\tprd{q:\Q} L(q) \leftrightarrow \sharp\brck{\tsm{r : \sharp\Q} (q^\sharp <' r) \times L'(r)}\right)\\
  &\qquad\times \left(\tprd{r:\Q} U(r) \leftrightarrow \sharp\brck{\tsm{q : \sharp\Q} U'(q) \times (q <' r^\sharp)}\right)\\
  &\qquad\times \left(\tprd{q:\Q} (L(q) \times U(q)) \to \sharp\emptyset \right)
  \times \left(\tprd{q,r:\Q} (q < r) \rightarrow \sharp\brck{\sharp(L(q) + U(r))}\right)
    \\
  & \simeq \tsm{L,U:\Q \to \codiscprop}\\
  &\qquad \sharp\left(\brck{\tsm{q:\Q} L(q)}\times \brck{\tsm{r:\Q} U(r)}\right)\\
  &\qquad\times \left(\tprd{q:\Q} L(q) \leftrightarrow \sharp\brck{\tsm{r : \Q} \sharp(q < r) \times L(r)}\right)\\
  &\qquad\times \left(\tprd{r:\Q} U(r) \leftrightarrow \sharp\brck{\tsm{q : \Q} U(q) \times \sharp(q < r)}\right)\\
  &\qquad\times \left(\tprd{q:\Q} (L(q) \times U(q)) \to \sharp\emptyset \right)
  \times \left(\tprd{q,r:\Q} (q < r) \rightarrow \sharp\brck{(L(q) + U(r))}\right)
  \\ \\
  \flat\R' &\simeq \tsm{M,N:\flat(\Q \to \discprop)}
  \flet L^\flat := M in \flet U^\flat := N in\\
  &\qquad\flat \sharp\left(\brck{\tsm{q:\Q} \sharp L(q)}\times \brck{\tsm{r:\Q} \sharp U(r)}\right)\\
  &\qquad\times \flat \left(\tprd{q:\Q} \sharp L(q) \leftrightarrow \sharp\brck{\tsm{r : \Q} \sharp(q < r) \times \sharp L(r)}\right)\\
  &\qquad\times \flat \left(\tprd{r:\Q} \sharp U(r) \leftrightarrow \sharp\brck{\tsm{q : \Q} \sharp U(q) \times \sharp(q < r)}\right)\\
  &\qquad\times \flat \left(\tprd{q:\Q} (\sharp L(q) \times \sharp U(q)) \to \sharp\emptyset \right)
  \times \flat \left(\tprd{q,r:\Q} (q < r) \rightarrow \sharp\brck{(\sharp L(q) + \sharp U(r))}\right)
  \\
  &\simeq \tsm{M,N:\flat(\Q \to \discprop)}
  \flet L^\flat := M in \flet U^\flat := N in\\
  & \qquad \flat \left(\brck{\tsm{q:\Q} L(q)}\times \brck{\tsm{r:\Q} U(r)}\right)\\
  &\qquad\times \flat \left(\tprd{q:\Q} \flat\left(L(q) \leftrightarrow \brck{\tsm{r : \Q} (q < r) \times  L(r)}\right)\right)\\
  &\qquad\times \flat \left(\tprd{r:\Q} \flat\left(U(r) \leftrightarrow \brck{\tsm{q : \Q} U(q) \times (q < r)}\right)\right)\\
  &\qquad\times \flat \left(\tprd{q:\Q} \flat\big((L(q) \times U(q)) \to \emptyset\big) \right)
  \times \flat \left(\tprd{q,r:\Q} \flat\big((q < r) \rightarrow \brck{(L(q) + U(r))}\big)\right)
\end{align*}
\caption{The Dedekind reals and their codiscrete version}
\label{fig:Ragain}
\end{figure}
\begin{proof}
  First of all, notice that since $\codiscprop$ is codiscrete, we have
  \[(\sharp\Q \to \codiscprop) \simeq (\Q\to \codiscprop). \]
  Thus, in $\R'$ we can sum over $L,U:\Q\to\codiscprop$ instead; we write $L',U'$ for their unique extensions to $\sharp \Q$.

  Next, by the dependent version of \cref{thm:codiscrete-reflective}, the $\prd{q:\sharp\Q}$'s in $\R'$ can be replaced by $\prd{q:\Q}$, with $q$ replaced by $q^\sharp$ inside.
  Now we can use the facts that $L'(q^\sharp)=L(q)$, $U'(q^\sharp)=U(q)$, and $(q^\sharp <' r^\sharp) = \sharp(q<r)$ (since the inequalities on $\sharp\Q$ extend those of $\Q$), and since $\sharp(q<r)$ is the domain of a function with codiscrete codomain, we can replace it by $(q<r)$.
  Thus we obtain the second version shown in \cref{fig:Ragain}.

  Now note that for $q:\sharp \Q$ we have
  \[L'(q) = \sharp L'(q) = \sharp L'(q_\sharp{}^\sharp) = \sharp L(q_\sharp)\]
  and similarly for $R'$.
  Likewise, for $q:\sharp \Q$ and $r:\Q$ we have
  \[(q <' r^\sharp) = \sharp (q <' r^\sharp) = \sharp (q_\sharp{}^\sharp <' r^\sharp) = \sharp (q_\sharp<r).\]
  Thus, using \cref{thm:sharp-sigma} (and, in the second and third lines, the fact that $\sharp$ preserves products), together with the fact that $\sharp\brck{\sharp A}\simeq \sharp\brck{A}$ (since both are a reflection into codiscrete propositions) we obtain the third version.

  Now, since \Q\ is crisply discrete, we have
  \[\flat(\Q\to\codiscprop) \simeq \flat(\Q\to\flat\codiscprop) \simeq \flat(\Q\to\discprop) \]
  using \cref{thm:disc-is-flat-codisc} in the second step.
  Thus, with \cref{thm:flat-pres-sigma} we can convert $\flat\R'$ into a sum over $M,N:\flat(\Q\to\discprop)$, obtaining the first expression for $\flat\R'$ in \cref{fig:Ragain}.
  (Since $\flat$ preserves products, we have also distributed it over the $\times$s.)

  However, as remarked after \cref{thm:disc-is-flat-codisc}, the equivalence $\disc\simeq\flat\codisc$ respects all suitably-(co)reflected type formers.
  Since discrete types are closed under truncations and contain $\emptyset$, we deduce the second expression for $\flat\R'$ in \cref{fig:Ragain}.
  Here we abuse notation somewhat in applying $\flat$ to types that depend on cohesive variables $q,r:\Q$; the meaning is that since $\Q$ is discrete, we can transfer $q,r$ to elements of $\flat \Q$ and then destruct them into crisp variables first, as in the left-hand side of~\eqref{eq:flat-corefl-D}.
  Now applying \cref{thm:discrete-coreflective-D} gets rid of all the $\flat$s inside the $\prod$s.
  Finally, we can un-distribute the $\flat$s over the $\times$s, and apply \cref{thm:flat-pres-sigma} in reverse, to obtain $\flat \R$ on the right.
\end{proof}

\section{Topos models of spatial type theory}
\label{sec:topos-models}
\addtocontents{toc}{\protect\setcounter{tocdepth}{1}}

So far we have described % (in \cref{sec:spatial-type-theory,sec:codisc,sec:flat,sec:disc})
a \emph{spatial type theory} with crisp variables and two modalities $\flat$ and $\sharp$, and also (in \cref{sec:lem,sec:lem-2,sec:axiom-choice,sec:axiom-choice-2}) some classical axioms that can be added to this theory without (we expect) destroying the intended topological content.
However, nothing we have said excludes the possibility that all types are both discrete and codiscrete, with $\flat$ and $\sharp$ being the identity.
(In this case our classical axioms would of course apply to all types, trivializing the topology.)

Thus, in order to ``do some real topology'', we need a new axiom that gives us a way to ``access the topology'' internally in type theory.
We could motivate such an axiom purely on first principles, but it seems appropriate at this point to instead discuss the various possible models of our theory.

So far, we have been referring mainly to classical topological spaces for intuition.
This is fine, but topological spaces do not \emph{actually} model our theory for several reasons.
One reason is that they are not a topos, lacking universes and even a subobject classifier.
Indeed, they are not even locally cartesian closed, so do not have $\prod$-types.

Another reason is that a classical topology is ``mere structure'' on a set, whereas our theory demands something more.
Consider for instance the following definition.

\begin{defn}\label{defn:concrete}
  A type $A$ is \textbf{concrete} if $\sharpf: A\to \sharp A$ is a embedding (i.e.\ a $(-1)$-truncated map).
\end{defn}

In classical topological spaces, the map from a space to its codiscrete reflection is always injective, so that all spaces would be concrete in this sense.
However, for us this is impossible.

\begin{thm}\label{thm:nonconcrete}
  If $\prop$ is concrete, then all \mprop{}s are codiscrete.
\end{thm}
\begin{proof}
  Suppose $\sharpf : \prop \to \sharp \prop$ is an embedding.
  Then for any $P,Q:\prop$, we have $(P=Q) \simeq (P^\sharp = Q^\sharp)$; but the latter is $\sharp(P=Q)$ by \cref{thm:path-sharp}.
  Thus, $P=Q$ is codiscrete.
  But taking $Q=\unit$, we have $(P=1)=P$, so all \mprop{}s are codiscrete.
  % By \cref{thm:codiscrete-notnot}, this means $\neg\neg P \to P$ for all $P$.
\end{proof}

From a topological point of view, ``all \mprop{}s are codiscrete'' would mean that all continuous injections are subspace embeddings.
This seems quite dubious, and in fact it trivializes almost the entire theory:

\begin{thm}\label{thm:codisc-toploc}
  For any $n:\N$, an $n$-type $A$ is codiscrete if and only if the map $\const:A\to (P\to A)$ is an equivalence for all \mprop{}s $P$ such that $\sharp P$.
  In particular, therefore, if all \mprop{}s are codiscrete, then all $n$-types are codiscrete.
\end{thm}
\begin{proof}
  The ``only if'' direction follows immediately from the universal property of $\sharp$ (\cref{thm:codiscrete-reflective}), since if $P:\prop$ and $\sharp P$, then $(\sharp P\to A) \simeq A$.
  For the ``if'' direction, we induct on $n$.
  When $n=-1$, let $A:\prop$ and suppose $\const:A\to (P\to A)$ is an equivalence for all $P:\prop$ such that $\sharp P$.
  To show $A$ is codiscrete, we assume $\sharp A$ and try to prove $A$.
  But then $A$ is a valid choice of $P$, and of course $A\to A$; hence $A$.

  Before embarking on the inductive step, we note that by \cref{thm:path-sharp}, $A\to \sharp A$ is an embedding (i.e.\ $A$ is concrete) if and only if $x=y$ is codiscrete for all $x,y:A$.
  On the other hand, for $x,y:A$ and $P:\prop$, we have
  \[(\const_x = \const_y) \simeq (P\to (x=y)).\]
  Thus $\const:A\to (P\to A)$ is an embedding if and only if
  \[\const:(x=y)\to (P\to (x=y))\]
  is an equivalence.
  Therefore, if $A$ is an $(n+1)$-type and $\const:A\to (P\to A)$ is an equivalence for all $P:\prop$ such that $\sharp P$, then the inductive hypothesis ensures that $A\to \sharp A$ is an embedding; thus it remains to show that it is surjective.

  Let $z:\sharp A$, and write $\eta \defeq \sharpf : A\to \sharp A$.
  Then $\fib_{\eta}(z)$ is a \mprop{}.
  Moreover, since $\sharp$ is left exact, we have $\sharp \fib_{\eta}(z) \simeq \fib_{\sharp\eta}(z^\sharp)$.
  But $\sharp\eta$ is an equivalence, so the latter type is contractible.
  Thus, $\sharp \fib_{\eta}(z)$, so by assumption $\const:A\to (\fib_{\eta}(z)\to A)$ is an equivalence.
  But we have $\proj_1 : \fib_{\eta}(z)\to A$, so there is an $x:A$ such that $\proj_1 = \const_x$.

  We claim $\eta(x) = z$.
  Since this is a codiscrete goal, and we have $\sharp \fib_{\eta}(z)$, by $\sharp$-induction we may assume that $\fib_{\eta}(z)$, i.e.\ there is a $y:A$ and $p:\eta(y) = z$.
  But $\proj_1 = \const_x$ yields $y=x$, and hence $\eta(x) = z$ as desired.
\end{proof}

\begin{rmk}
  \cref{thm:codisc-toploc} is a general fact about left exact monadic modalities, which generalizes the fact that a subtopos of a 1-topos is determined by its action on monomorphisms \parencite[see e.g.][A4.3.6]{ptj:elephant}.
  A subtopos of an $(\oo,1)$-topos, however, is not completely determined by its action on $n$-types for finite $n$ \parencite[see][\S6.5]{lurie:higher-topoi}.
\end{rmk}

Therefore, we are forced to think of our ``topologies'' as some more contentful ``stuff'' rather than just structure.
A good example to think about is the \emph{topological topos} of \textcite{ptj:topological-topos}.
Its objects are sets $X$ equipped with, for every sequence $x:\N\to X$ and point $y\in X$, a collection of ``witnesses that $(x_n)$ converges to $y$'', equipped with natural operations (e.g.\ there is a specified ``reflexivity'' witness that a constant sequence converges to the point at which it is constant, etc.)
Such an object is codiscrete just when every sequence converges \emph{uniquely} to every point.
Thus, since monomorphisms must be injective on points and witnesses both, such an object is concrete just when there is \emph{at most one witness} that any sequence converges to any point.
But in general, a ``space'' in this topos can have a sequence that converges to one point in many different ways.

The topological topos is a 1-topos rather than a higher topos, so we can expect it to model all of our type theory except for full univalence.
Pending a solution to the general problem of higher-topos-theoretic semantics, we may expect that there should be a corresponding ``topological $(\infty,1)$-topos'' that models the whole theory with univalence.

More generally, the {spatial type theory} of \crefrange{sec:spatial-type-theory}{sec:disc} should admit models in any \emph{local topos}.
A geometric morphism $f:\mathcal{F} \to \mathcal{E}$ is called \emph{local} \parencite[see e.g.][\S C3.6]{ptj:elephant} if its direct image $f_*$ admits a further $\mathcal{E}$-indexed right adjoint $f^!$, which is then necessarily fully faithful (and so is $f^*$).
In this case $\mathcal{F}$ inherits a comonad $f^*f_*$ and a monad $f^! f_*$ which are adjoint, and ought to extend the internal type theory of $\mathcal{F}$ to model our $\flat$ and $\sharp$ (without univalence).
The restricted classical axioms of \cref{sec:lem,sec:lem-2,sec:axiom-choice,sec:axiom-choice-2} will hold if their ordinary versions hold in $\mathcal{E}$.
Moreover, if we enhance $f$ to a local morphism of $(\oo,1)$-toposes, then it should be possible to model homotopy type theory with (at least weak) univalent universes.

\begin{rmk}\label{rmk:missing-pieces}
  At present, the previous paragraph is only conjectural.
  What needs to be done to make it precise is (1) define an appropriate sort of ``category with families'' or ``contextual category'' to handle our two-context type theory with crisp variables and formulate $\sharp$ and $\flat$ as algebraic structure on such a gadget, (2) prove that syntax yields an initial one of these, and (3) construct such algebraic objects from local geometric morphisms.
  Of these, (1) should be straightforward, (2) is still only conjectural even for most ordinary type theories \parencite[the one complete theorem along these lines is in][]{streicher:semtt}, while (3) should be straightforward in the 1-topos case and, I hope, possible in the \oo-case, at least if we only ask for weak universes.
  However, none of this is technically necessary for our actual results, which are simply theorems in our formal system; the categorical semantics is for motivation, for relative consistency, and, eventually, for applications to classical mathematics.
\end{rmk}

As shown by \textcite[C3.6.3(d)]{ptj:elephant}, local toposes are obtained as sheaves on sites with terminal objects admitting no nontrivial covers (``local sites''), and this was generalized to $(\oo,1)$-toposes by \textcite[Proposition 3.4.18]{schreiber:dcct}.
The topological topos arises in this way: its site has two objects $1$ and $\N_\infty$.
More generally, we can consider sheaves on any small full subcategory of topological spaces that contains the one-point space.
Any such topos will model spatial type theory with our classical axioms, and if we generalize to higher sheaves, it should have at least weak univalent universes.

\begin{rmk}\label{rmk:local-toposes}
  A number of further examples of local geometric morphisms of 1-toposes can be found in \textcite[C3.6.3]{ptj:elephant}, all of which generalize appropriately to $(\oo,1)$-toposes.
  A ``purely \oo-categorical'' example is the $(\oo,1)$-topos of parametrized spectra, whose objects are pairs $(X,E)$ where $X$ is a space and $E$ a spectrum parametrized over $X$; this is equivalently the category of 1-excisive functors from pointed spaces to unpointed spaces studied in Goodwille calculus \parencite{goodwillie:calculus-iii,lurie:ha,joyal:logoi}.
  Here the discrete and codiscrete objects coincide, being those for which $E=0$, and moreover we have $\sharp=\flat$.
  Finally, the (co)discrete objects are also exactly the \emph{hypercomplete} ones, and hence include all $n$-types for finite $n$.

  In particular, therefore, our ``spatial type theory'' is in fact significantly more general than the name suggests.
  In the rest of this paper we will add axioms that bring it closer to our intended intuition that ``types have topology''; studying other classes of models would of course lead us to different axioms.
\end{rmk}

\begin{rmk}\label{thm:abs}
  A different sort of internal logic for local geometric morphisms (along with another class of examples arising from realizability) is considered by \textcite{ab:ax-local} and \textcite{abs:lrt-modal-comput-ea}.
  It is a ``logic over a type theory'' in which the types all belong to the base category but the propositions to the local topos, and there are two propositional modalities.
  This theory can be roughly identified with the fragment of spatial type theory in which we allow arbitrary crisp contexts (of sets, i.e.\ 0-truncated types), but require all types in the cohesive context to be propositions.
  (Since dependence on a \mprop{} is always trivial, we can then assume that our cohesive propositions depend only on the crisp context, as is usual for logic-enriched type theories.)
  Unfortunately, the meaning of $\flat$ and $\sharp$ in \textit{ibid.}\ is reversed from ours; but as we will see in \cref{rmk:why-flat} there is a good reason for our choice.

  As remarked in \textit{ibid.}, when restricted to propositions, our $\flat$ (their $\sharp$) satisfies the formal properties of the operator $\Box$ in S4 modal logic.
  Roughly speaking, this is the eventual origin of our term \emph{modality} for $\flat$ and $\sharp$ (and, later on, $\shape$).

  One further remark is that the operations relating types to propositions in the logic of \textit{ibid.}\ are not exactly those of spatial type theory.
  On one hand, since $\prod$ and $\sum$ cannot be applied to crisp dependence directly, we have to wrap them in a $\flat$ first to yield the quantifiers $\exists$ and $\forall$ of \textit{ibid.}
  Specifically, if $u::A\mid \cdot \types P:\prop$, then by ``$\forall u, P$'' we can only mean ``$\prd{x:\flat A} \flet u^\flat := x in P$'', and similarly for $\exists$.

  On the other hand, if $A$ is a set and $u::A$ and $v::A$, we \emph{can} form the proposition $u=v$, but this is not what is meant by ``$u=v$'' in \textit{ibid.}: since that is discrete, it must instead be $\flat(u=v)$ (or equivalently, by \cref{thm:equiv-flat-path}, $(u^\flat=v^\flat)$).
  That this has the correct universal property \parencite[i.e.\ satisfies ``Lawvere's law'' from][]{lawvere:comprehension} can be seen from crisp Id-induction (\cref{thm:crisp-J}), which implies that if $x::A,y::A \mid\cdot \types Q:\prop$ and $u::A\mid\cdot \types Q[u/x,u/y]\istrue$, then $x::A,y::A,p::x=y\mid\cdot \types Q \istrue$, and hence $x::A,y::A\mid q:\flat(x=y) \types Q \istrue$ by $\flat$-induction.
\end{rmk}

Returning to our goal of ``accessing the topology'' internally in type theory, one natural approach would be to pick one local topos with a ``topological'' character, or a class of them, and look for special features of that model that are visible internally.
For Johnstone's topological topos, this approach has been pursued by \textcite{es:universe-indiscrete}, in the following way.
One can define internally the \emph{generic convergent sequence} $\N_\oo$ to be the type of non-increasing functions $\N\to\bool$; in the topological topos this does in fact yield the space $\N_\oo$.
Thus, one can define a \emph{convergent sequence} in an arbitrary type $X$ to be a map $\N_\oo \to X$.
If we wanted to combine this approach with spatial type theory, we could add axioms ensuring that the discrete and/or codiscrete types, as defined using $\flat$ and $\sharp$, can be characterized in terms of $\N_\oo$.%
\footnote{Note that the notion of ``indiscrete'' referred to in the title of \textcite{es:universe-indiscrete} is weaker than our ``codiscrete''; the former asks only that every sequence converges to every point in \emph{some} way.
  \cref{thm:nonconcrete,thm:codisc-toploc} imply that the universe cannot be codiscrete in our stronger sense without largely trivializing the theory.}

This is an interesting direction for future research, but in this paper we take a different route, corresponding to a different choice of a small full subcategory of topological spaces.
We are interested in applications to theorems about ``manifold-like'' spaces locally modeled on the real numbers, both because of their intrinsic interest, because homotopy types presented by such spaces play an important role in classical algebraic topology, and because of their importance in applications (e.g.\ to physics).
While manifold-like spaces are generally well-enough behaved that their topology is detected by convergent sequences, so that they embed fully faithfully into the topological topos, the generating object $\N_\oo$ of this topos is \emph{not} at all manifold-like.
In particular, it is not even \emph{locally connected}, a property that is important for ensuring that a space has a well-behaved fundamental \oo-groupoid.

Thus, our motivating model will instead be the topos of sheaves (or the higher topos of higher sheaves) on the full subcategory of topological spaces whose objects are the \emph{cartesian spaces} $\R^n$.
This higher topos was studied by \textcite[\S4.3]{schreiber:dcct} under the name \emph{Euclidean-topological cohesion}, its objects being called \emph{Euclidean-topological \oo-groupoids} or \emph{continuous \oo-groupoids}; in other places they are called \emph{topological stacks} or \emph{topological \oo-stacks}.
In the same way that the topological topos defines ``topology'' to consist of convergent sequences, this topos defines ``topology'' to consist of \emph{continuous paths} (and, more generally, homotopies and higher homotopies between such paths).
In particular, just as the former allows spaces in which a sequence converges to a limit in many ways, our topos allows paths that ``are continuous'' in many ways; the ``concrete objects'' will be those in which any path is continuous in at most one way.%
\footnote{In general, when interpreted in a local topos of sheaves on a local site, our definition of ``concrete'' reduces to the notion of ``concrete sheaf'' studied by \textcite{dubuc:concrete-qtop,de:quasitopoi,bh:cc-smooth}.
  Its formulation using $\sharp$ was pointed out by Carchedi and appears in \textcite[Proposition 3.7.5]{schreiber:dcct}.}

We emphasize again that these ``paths'' are totally different from the terms of identity type (or ``identifications'') that in homotopy type theory are sometimes called ``paths''.
For instance, the \emph{topological circle} $\topcirc$ contains many continuous paths, but no identifications other than reflexivity (inside type theory, it will be a ``set''), whereas the \emph{homotopical circle} $\hocirc$ contains $\Z$-many identifications but no nontrivial topology (in fact, we have seen already in \cref{thm:s1-discrete} that it is topologically discrete).
A general object of our (higher) topos is a sort of stack, with \emph{both} topological paths and nontrivial identifications, and moreover it can have ``topology on its identifications''.
A good example is the delooping of a topological group such as $O(n)$; this is homotopically a 1-type, but its loop type is $O(n)$ with its topological structure.

Just as $\N_\oo$ plays a central role in the internal theory of the topological topos, we should expect the space of real numbers $\R$ to play an analogously central role for us.
We would like to be able to \emph{define} this space internally, just as $\N_\oo$ can be defined internally using decreasing binary sequences.
As remarked in \cref{sec:axiom-choice-2}, there are at least two obvious ways to define ``the real numbers'' internally, so we should inquire whether either of these definitions automatically inherits the correct topology.

The answer is that the \emph{Dedekind} real numbers do inherit the correct topology, whereas the Cauchy real numbers do not.
In fact, the Cauchy real numbers are \emph{discrete}; we will prove this internally in \cref{thm:RC-discrete}.
The correctness of the topology on the Dedekind reals requires a more involved semantic argument (see \textcite{lin:shtop} and \cref{thm:axt-holds}), but in \cref{thm:continuity} we will prove a corresponding internal approximation.
Both of these proofs use the axiom of \emph{real-cohesion} to be introduced in the second part of the paper, which enhances our spatial type theory by asserting that the ``topology'' of types is determined by maps out of the Dedekind real numbers.

\begin{rmk}\label{thm:local-sections}
  Before continuing, let us note that having moved away from the idea of types as classical topological spaces, we also need to update our subsidiary intuitions.
  For example, we noted before that monomorphisms of topological spaces are continuous injections; but in topological toposes the monomorphisms must be injective not only on points but on ``topological structure'' (e.g.\ witnesses to convergence of a sequence or continuity of a path).
  Such a monomorphism is a subspace inclusion when it is also ``full'' on topological structure (e.g.\ for any path lying in the subspace, the witnesses of its continuity in the subspace map isomorphically to the witnesses in the ambient space).

  The case of (regular) epimorphisms (corresponding to type-theoretic surjections) is somewhat more subtle and involves the Grothendieck topology with which we equip our site.
  In Johnstone's topological topos, a map $f:A \to B$ is epimorphic if it is surjective on points and any (witness of a) convergent sequence $(b_n) \to b_\oo$ in $B$ has a ``subsequence'' $(b_{n_k}) \to b_\oo$ that is the image of some (witness of a) convergent sequence in $A$.
  More importantly for us, in our topos of Euclidean-topological spaces, $f:A \to B$ is epimorphic if for any continuous ``plot'' $\R^n \to B$, we can cover $\R^n$ by open balls such that the restriction of our plot to any such ball is in the image of some plot in $A$.
  If we think of $A$ and $B$ as a sort of manifold locally modeled on $\R^n$, then the epimorphisms are the continuous maps that admit local sections.
\end{rmk}

% \part{Cohesive and real-cohesive type theory}
% \label{sec:rctt}

\section{Cohesion and real-cohesion}
\label{sec:real-cohesion}
\label{sec:reals}

Recall that $\R$ denotes the \emph{Dedekind} real numbers (since it is much more important for us than any other kind of real numbers, we dignify it by omitting any subscript).
We may call a map $\R\to A$ a \emph{(continuous) path} in $A$.
If topology is to be determined by such paths, then a \emph{discrete} type should be one in which every such path is constant.
This is what the following axiom expresses.

\begin{named}{Axiom R$\flat$}\label{ax:r4}\label{ax:r3}
  A crisp type $A::\type$ is discrete if and only if $\const: A \to (\R\to A)$ is an equivalence.
\end{named}

\begin{rmk}
  As stated, this ``axiom'' has to be formulated as an unjustified rule rather than an element of some fixed type, like our original \ref{ax:lem}.
  However, as we did for \ref{ax:lem} in \cref{rmk:flat-lem}, by using $\flat$ we can reformulate it as an assumed (crisp) element of
  \[ \prd{B:\flat\type} \flet A^\flat := B in (\isdisc(A) \leftrightarrow \mathsf{isequiv}(\const_{A,\R})). \]
\end{rmk}

% Note that this condition also implies that maps out of higher powers $\R^n$ into $A$ are constant, by induction.

\subsection{Cohesion}
\label{sec:cohesion}

When \ref{ax:r3} is added to spatial type theory, we call it \emph{real-cohesive type theory} (or real-cohesive homotopy type theory, if homotopical features such as univalence and HITs are included).
We now explore the consequences of \ref{ax:r3} in stages, by introducing a sequence of weaker axioms that use successively more properties of $\R$.
In addition to clarifying the exposition, this makes it clear that many results are true much more generally.
By \emph{cohesive (homotopy) type theory} we mean spatial type theory with one or more of these weaker axioms.

We begin with the following, the weakest of the axioms of cohesion.

\begin{named}{Axiom C0}\label{ax:r0}
  There is a type family $R::I\to\type$ such that a crisp type $A::\type$ is discrete if and only if $\const: A \to (R_i\to A)$ is an equivalence for all $i:I$.
\end{named}

Since this characterization of discreteness is purely internal, we finally have a notion of ``discrete'' that applies also to non-crisp types.

\begin{defnr0}\label{defn:discrete}
  An arbitrary type $A:\type$ (not necessarily crisp) is said to be \textbf{discrete} if $\const: A \to (R_i\to A)$ is an equivalence for all $i:I$.
\end{defnr0}

\begin{lemr0}\label{thm:discrete-pi}
  Discrete types are an exponential ideal, and even a ``dependent exponential ideal'': if $B:A\to \type$ and each $B(x)$ is discrete, then so is $\prd{x:A} B(x)$.
\end{lemr0}
\begin{proof}
  $\left(R_i \to \prd{x:A} B(x)\right) \simeq \left(\prd{x:A} (R_i\to B(x))\right) \simeq \left(\prd{x:A} B(x)\right)$.
\end{proof}

Technically we also need to verify that this equivalence is the map $\const$, but this is usually easy, so we omit it, both here and in the following lemmas.

\begin{lemr0}\label{thm:discrete-sigma}
  Discrete types are closed under $\sum$: if $A$ is discrete and $B:A\to \type$ has each $B(x)$ discrete, then $\sm{x:A} B(x)$ is discrete.
\end{lemr0}
\begin{proof}
  We have $\left(R_i \to \sm{x:A} B(x)\right) \simeq \left(\sm{f:R_i\to A} \prd{r:R_i} B(f(r)) \right)$.
  But since $A \to (R_i\to A)$ is an equivalence, the latter is equivalent to $\sm{x:A} (R_i \to B(x))$, which is equivalent to $\sm{x:A} B(x)$ since each $B(x)$ is discrete.
\end{proof}

\begin{lemr0}\label{thm:discrete-path}
  If $A$ is discrete and $x,y:A$, then $(x=y)$ is discrete.
\end{lemr0}
\begin{proof}
  Since $\const : A\to (R_i\to A)$ is an equivalence, it induces an equivalence on identity types:
  \[ (x=y) \simeq (\const_x = \const_y). \]
  However, $(\const_x = \const_y)$ is equivalent to $R_i\to (x=y)$.
\end{proof}

\begin{lemr0}\label{thm:discrete-pullback}
  Discrete types are closed under pullbacks.
\end{lemr0}
\begin{proof}
  If $A,B,C$ are discrete and $f:A\to C$, $g:B\to C$, then
  \[A\times_C B = \sm{x:C}{y:B}(f(x)=g(y)),\]
  so this follows from \cref{thm:discrete-sigma,thm:discrete-path}.
\end{proof}

\subsection{Punctual cohesion}
\label{sec:pointed-cohesion}

Our next batch of consequences uses the additional assumption that each $R_i$ is inhabited (which is clearly true for $\R$).

\begin{named}{Axiom C1}\label{ax:r1}
  \ref{ax:r0} holds, and moreover we have some $r:\prd{i:I} R_i$.
\end{named}

This makes \cref{defn:discrete} equivalent to a weaker-looking definition.

\begin{lemr1}\label{thm:discrete-by-const}
  A type $A$ is discrete if and only if every function $R_i\to A$ is constant, i.e.\ for all $i:I$ and $f:R_i\to A$ there is an $a:A$ such that $f(x)=a$ for all $x:R_i$.
\end{lemr1}
\begin{proof}
  The given condition says that each $\const: A \to (R_i\to A)$ has a section; but it always has a retraction, namely evaluation at $r_i$.
  Thus, having a section is equivalent to being an equivalence.
\end{proof}

The primary application of \ref{ax:r1} for us is the following fact.

\begin{lemr1}\label{thm:discrete-prop}
  All \mprop{}s are discrete.
\end{lemr1}
\begin{proof}
  Since each $R_i$ is inhabited, if $A$ is a \mprop{} and $R_i\to A$, then $A$ is contractible; hence the condition of \cref{thm:discrete-by-const} holds.
\end{proof}

From the $\flat$-$\sharp$ adjunction, we obtain a partial dual:

\begin{lemr1}\label{thm:codiscrete-crispprop}
  All crisp \mprop{}s are codiscrete.
\end{lemr1}
\begin{proof}
  We are to show that for any $P::\prop$, $P$ is codiscrete, i.e.\ that $\sharp P \to P$ holds.
  Since $P$ and $\sharp P$ are both discrete by \cref{thm:discrete-prop}, it will suffice to prove $\flat \sharp P \to \flat P$.
  But by \cref{thm:co-disc-equiv-gamma}, $\flat\sharp P \simeq \flat P$.
\end{proof}

For instance, this implies that we can dispense with the distinction between ``crisp and discrete'' and ``crisply discrete''.

\begin{lemr1}\label{thm:crisply-discrete}
  If $A::\type$ is crisp and discrete, then we may assume it is crisply discrete.
\end{lemr1}
\begin{proof}
  Suppose $A::\type$ and that we have some $d : \isdisc(A)$.
  Since $\isdisc(A)$ is a \mprop{}, it is discrete, and so we might as well have $d : \flat (\isdisc(A))$.
  Thus, by $\flat$-induction, no matter our goal we are free to assume $d' :: \isdisc(A)$.
\end{proof}

Similarly, if a crisp function is an equivalence, it is automatically a crisp equivalence; if a crisp type is an $n$-type, it is crisply an $n$-type; and so on.

As another particular case of \cref{thm:codiscrete-crispprop}, we have:

\begin{thmr1}\label{thm:emptyset-codiscrete}
  $\emptyset$ is codiscrete, i.e.\ \ref{ax:dense} holds.
\end{thmr1}
\begin{proof}
  $\emptyset$ is a crisp \mprop{}.
\end{proof}

\begin{corlemr1}\label{thm:codiscrete-notnot-2}
  For any \mprop{} $P$ we have $\sharp P \simeq \neg\neg P$.
\end{corlemr1}
\begin{proof}
  This is just \cref{thm:codiscrete-notnot}, but with \ref{ax:dense} replaced by \ref{ax:r1}, which implies it (by \cref{thm:emptyset-codiscrete}).
\end{proof}

In particular, even though \ref{ax:r1} only characterizes discreteness explicitly, in the presence of \ref{ax:lem} it implies an even more explicit characterization of codiscreteness, at least for $n$-types.

\begin{corlemr1}\label{thm:codiscrete-notnot-ntypes}
  An $n$-type $A$ is codiscrete if and only if $\const:A\to (P\to A)$ is an equivalence for all \mprop{}s $P$ such that $\neg\neg P$.
\end{corlemr1}
\begin{proof}
  Combine \cref{thm:codiscrete-notnot-2,thm:codisc-toploc}.
\end{proof}

\begin{rmk}
  In topos-theoretic language, \ref{ax:dense} says that the subtopos defined by $\sharp$ is \emph{dense}, while the stronger \cref{thm:codiscrete-crispprop} says that $\sharp$ is in fact \emph{fiberwise dense} in the sense of \textcite[C1.1.22]{ptj:elephant}.
\end{rmk}

We also obtain a characterization of discrete \emph{sets} (i.e.\ 0-types) that makes no reference to the family $R$, and coincides with that of~\textcite{penon:thesis,dp:compact-obj}: they are the sets with decidable equality.
%(Recall that any type with decidable equality is a set, by Hedberg's theorem.)

\begin{lemlemr1}\label{thm:disc-deceq}
  A set $A$ is discrete if and only if $\forall x,y:A. (x=y \vee x\neq y)$.
\end{lemlemr1}
\begin{proof}
  If $A$ is discrete, then given $x,y:A$ we may assume they are crisp.
  Then since $A$ is a set, $x=y$ is a crisp \mprop{}, hence $x=y \vee x\neq y$ follows from \ref{ax:lem}.

  Conversely, suppose $\forall x,y:A. (x=y \vee x\neq y)$.
  Since $x=y$ and $x\neq y$ are incompatible \mprop{}s, we have $((x=y) \vee (x\neq y)) \simeq ((x=y) + (x\neq y))$, so we can define functions by cases on equality in $A$.
  Now for any $f:R_i \to A$, define $g:R_i \to \bool$ by
  \[ g(r) =
  \begin{cases}
    \btrue &\quad \text{if }f(r)=f(r_i)\\
    \bfalse &\quad \text{if }f(r)\neq f(r_i).
  \end{cases}\]
  Since $\bool$ is discrete, $g$ is constant.
  But $g(r_i) = \btrue$, so $g(r)=\btrue$ for all $r:R_i$, i.e.\ $f(r) = f(r_i)$ for all $r$.
  Hence $f$ is constant.
\end{proof}

We record a few more useful consequences.
Firstly, $\flat$ detects emptiness:

\begin{corr1}\label{thm:flat-detects-empty}
  If $A::\type$ and $\neg\flat A$, then $\neg A$.
\end{corr1}
\begin{proof}
  If $\neg\flat A$, then it may as well hold crisply by \cref{thm:discrete-prop}, i.e.\ we have $f::\flat A \to \emptyset$.
  Then by \cref{thm:flat-sharp-adj}, we also have $A \to \sharp \emptyset$, and hence $\neg A$ since $\sharp\emptyset = \emptyset$.
\end{proof}

Next, any injective continuous function with discrete codomain also has discrete domain, as we would expect from classical topology.

\begin{lemr1}\label{thm:discrete-subobject}
  Any subobject of a discrete type is discrete.
\end{lemr1}
\begin{proof}
  If $A$ is discrete and $P:A\to \prop$ is a predicate, then each $P(x)$ is discrete by \cref{thm:discrete-prop}; hence the subobject $\sm{x:A} P(x)$ is discrete by \cref{thm:discrete-sigma}.
\end{proof}

\begin{corr1}\label{thm:flat-cart-emb}
  If $m::A\to B$ is a crisp embedding, then its naturality square for $\flat$ is a pullback:
  \begin{equation*}
  \vcenter{\xymatrix@-.5pc{
      \flat A\ar[r]\ar[d] &
      A\ar[d]\\
      \flat B \ar[r] &
      B
      }}
  \end{equation*}
\end{corr1}
\begin{proof}
  Since $\flat$ preserves embeddings, $\flat A \to \flat B$ is an embedding, as is the pullback of $A$ to $\flat B$.
  Thus, it will suffice to show that the latter pullback factors through $\flat A$ as a subobject of $\flat B$.
  But since it is a subobject of a discrete type, it is discrete, so this follows from the universal property of $\flat A$.
\end{proof}

We can also characterize the concrete objects (\cref{defn:concrete}) more exactly.

\begin{corlemr1}\label{thm:concrete-negnegsep}
  A set $A$ is concrete if and only if it is $\neg\neg$-separated, i.e.\ $\neg\neg(x=y)\to (x=y)$ for all $x,y:A$.
\end{corlemr1}
\begin{proof}
  By definition, $A$ is concrete if $\sharpf:A\to \sharp A$ is an embedding.
  But by \cref{thm:path-sharp}, $(x^\sharp = y^\sharp) \simeq \sharp(x=y)$, so this holds just when $x=y$ is codiscrete for all $x,y:A$.
  Now apply \cref{thm:codiscrete-notnot-2}.
\end{proof}

For example, since \R\ is always $\neg\neg$-separated \parencite[see for example][D4.7.6]{ptj:elephant}, it follows that \R\ is concrete.
But in fact, this is true even without LEM.

\begin{thmr1}\label{thm:R-concrete}
  \R\ is concrete.
\end{thmr1}
\begin{proof}%[Sketch of proof]
  It will suffice to prove that for any crisp $A::\type$ and $f,g::A\to \R$, if the two composites $A\to \R\to\sharp \R$ are equal, then so are $f$ and $g$.
  (Applying this when $A = \R \times_{\sharp\R} \R$ will then show that $\R\to\sharp \R$ is monic.)
  Moreover, by \cref{thm:flat-sharp-adj}, saying that the two composites $A\to \R\to\sharp \R$ are equal is the same as saying that the two composites $\flat A \to A \to \R$ are equal.

  Now a map $f:A\to\R$ is determined by two \Q-indexed families of subobjects of $A$, say $L^f(q)$ and $U^f(q)$ for $q:\Q$, satisfying the usual axioms, and likewise for $g$.
  Thus, our assumption says that $L^f(q)$ and $U^f(q)$ agree with $L^g(q)$ and $U^g(q)$ when pulled back to $\flat A$, and we want to show that $L^f(q) = L^g(q)$ and $U^f(q) = U^g(q)$ already over $A$.
  By \cref{thm:flat-cart-emb}, the pullback of $L^f(q)$ to $\flat A$ is $\flat L^f(q)$, and likewise for all the others.
  (Note that \Q, being abstractly isomorphic to \N, is discrete, so we can assume any rational number to be crisp).

  Now if $q<r$, we have $L^f(q) \cap U^f(r) = \emptyset$, and thus $\flat L^f(q) \cap \flat U^f(r) = \emptyset$.
  But $\flat U^f(r) = \flat U^g(r)$ by assumption, so $\flat L^f(q) \cap \flat U^g(r) = \emptyset$ as well.
  Since $\flat$ preserves pullbacks, by \cref{thm:flat-detects-empty} we have $L^f(q) \cap U^g(r) = \emptyset$.
  If further $r<s$, then since $U^g(r) \cup L^g(s) = A$, we have $L^f(q) \subseteq L^g(s)$.
  Finally, for any $p:\Q$ we have $L^f(p) = \bigcup_{q<p} L^f(q) \subseteq \bigcup_{q<p} L^g(p) = L^g(p)$ (by taking $s\defeq p$, and $r$ between $q$ and $p$).
  The same argument applies in the other direction and to all the other subsets.
\end{proof}

We can also show the inclusion of the discrete retopologization is injective (indeed, bijective) on points --- but with a truncation restriction.

\begin{thmr1}\label{thm:flat-counit-inj}
  If $B::\type$ is a crisp set, then $\flat B \to B$ is injective.
\end{thmr1}
\begin{proof}
  Note that $\flat B$ is a set by \cref{thm:flat-set}.
  In the proof of \cref{thm:equiv-path-flat}, we showed that for any $u,v:\flat B$ we have $(u=v) \simeq \code(u,v)$, where
  \[ \code(u,v) \defeq \flet x^\flat := u in (\flet y^\flat := v in \flat(x=y)). \]
  However, if $B$ is a set, then $x=y$ is a \mprop{}, hence discrete by \cref{thm:discrete-prop}.
  Thus, by $\flat$-induction we can prove that $\code(u,v)$ is equivalent to
  \[ \flet x^\flat := u in (\flet y^\flat := v in (x=y)) \]
  which by a couple of commuting conversions is equivalent to
  \[ (\flet x^\flat := u in x) = (\flet y^\flat := v in y) \]
  i.e.\ to $(u_\flat=v_\flat)$.
  Thus, $\flatf$ is injective.
\end{proof}

The proof also makes clear why we should \emph{not} expect $\flat B \to B$ to be an embedding if $B$ is not a set: in that case we also have to discretify the identity types.

\begin{corr1}
  If $B::\type$ is a crisp set, then the composite $\flat B \to B \to \sharp B$ is injective.
\end{corr1}
\begin{proof}
  It is equal to the composite $\flat B \cong \flat\sharp B \to \sharp B$, and $\sharp B$ is also a crisp set.
\end{proof}

\begin{corr1}\label{thm:crisp-discrete-concrete}
  Any crisp discrete set is concrete.\qed
\end{corr1}

We can also show that discrete \emph{sets} are closed under surjective quotients \parencite[see also][A4.6.6]{ptj:elephant}.

\begin{lemr1}\label{thm:set-surj-discrete}
  If $A,B::\type$ are crisp, $A$ is discrete, $B$ is a set, and $f::A\to B$ is surjective, then $B$ is discrete.
\end{lemr1}
\begin{proof}
  The kernel-pair of $f$ is a subobject of $A\times A$, and is therefore crisply discrete by \cref{thm:discrete-subobject}.
  But $B$ is the set-coequalizer of that kernel pair, hence discrete by \cref{thm:setcoeq-discrete}.
\end{proof}

\begin{corr1}\label{thm:0trunc-discrete}
  If $A$ is crisp and discrete, then so is $\trunc0A$.\qed
\end{corr1}

Now we can prove the following, which was claimed in \cref{sec:topos-models} to hold in our topos model.

\begin{thmr1}\label{thm:RC-discrete}
  The Cauchy real numbers $\R_C$ are discrete.
\end{thmr1}
\begin{proof}
  Since $\Q$ and $\N$ are discrete, by \cref{thm:discrete-pi}, the type $\N\to \Q$ of sequences of rational numbers is discrete.
  The type $\mathcal{C}$ of Cauchy sequences is a subobject of $\N\to \Q$, so by \cref{thm:discrete-subobject} it is also discrete.
  Finally, the Cauchy reals are a surjective image of the set of Cauchy sequences, so by \cref{thm:set-surj-discrete} they are also discrete.
\end{proof}

\cref{thm:RC-discrete} is about the usual sort of Cauchy real numbers defined as a quotient of the set of Cauchy sequences.
\textcite[\S11.3]{hottbook} constructs a ``better'' set of Cauchy real numbers that is constructively Cauchy complete.
However, our classicality axioms suffice to ensure that the usual $\R_C$ is already Cauchy complete, hence coincides with that of \textcite{hottbook}:

\begin{coracr1}\label{thm:RC-cauchy-complete}
  $\R_C$ is Cauchy-complete.
\end{coracr1}
\begin{proof}[Sketch of proof.]
  Since $\R_C$ is discrete by \cref{thm:RC-discrete} and the quotient map $\mathcal{C} \to \R_C$ is surjective, by \cref{thm:discrete-ac} there exists a section of it.
  Thus, for any Cauchy sequence in $\R_C$ there exists a sequence of sequences of rational numbers, and ``diagonalizing'' this in the usual way we obtain a limit in $\R_C$.
\end{proof}

Finally, we can also nail down the exact relationship between the topology on the two types of real numbers.

\begin{corccorlemr1}\label{thm:RC-flat-R}
  $\R_C \simeq \flat \R$.
\end{corccorlemr1}
\begin{proof}
  By \cref{thm:RC-RD-points}, $\flat \R_C \simeq \flat \R$.
  But $\R_C$ is discrete, so $\flat\R_C \simeq \R_C$.
\end{proof}

We will return to the question of what the topology on $\R$ itself is in \cref{sec:continuity}.

\addtocontents{toc}{\protect\setcounter{tocdepth}{2}}
\subsection{Omniscience principles}
\label{sec:ac-2}

We have seen that the ordinary LEM and AC are inconsistent with spatial interpretations, motivating the introduction of $\sharp$ and $\flat$ in order to state our modified \ref{ax:lem} and \ref{ax:ac}.
However, there are weaker ``classicality'' principles whose ordinary versions \emph{are} true in our motivating model, and some of them can even be proven from our current axioms, such as the following.

\begin{thmlemr1}\label{thm:lpo}
  The \textbf{limited principle of omniscience (LPO)} holds: for any $f:\N\to\bool$, either there exists an $n:\N$ such that $f(n)=\btrue$, or $f(n)=\bfalse$ for all $n$.
\end{thmlemr1}
\begin{proof}
  Since $\N\to\bool$ is discrete by \cref{thm:discrete-pi}, we may assume $f::\N\to\bool$ is crisp.
  Then ``$\exists n. f(n) = \btrue$'' is a crisp \mprop{}, so by \ref{ax:lem} we have either $\exists n. f(n) = \btrue$ or $\neg(\exists n. f(n)=\btrue)$.
  But $\neg (\exists n. f(n) = \btrue)$ is the same as $(\forall n. f(n)=\bfalse)$, so we are done.
\end{proof}

In particular, this implies that if we define $\N_\oo$ as in \textcite{es:universe-indiscrete} to be the type of non-increasing binary sequences $\N\to\bool$, then the canonical map $\N+\unit \to \N_\oo$ is an equivalence.
In particular, \emph{every} type in our theory is ``indiscrete'' in the sense of \textcite{es:universe-indiscrete}.
Thus, the ``topology'' of the types in cohesive type theory is very different from the ``topology'' studied by \textcite{es:universe-indiscrete}.

It is well-known in constructive mathematics that LPO implies the following weaker classicality principles.

\begin{corlemr1}\label{thm:llpo}
  The \textbf{lesser limited principle of omniscience (LLPO)} holds: for any $f,g:\N\to\bool$, if it is not the case that both $\exists n.f(n)=\btrue$ and $\exists n.g(n)=\btrue$, then either $\forall n.f(n)=\bfalse$ or $\forall n.g(n)=\bfalse$.\qed
\end{corlemr1}

\begin{corlemr1}\label{thm:markov}
  \textbf{Markov's principle (MP)} holds: if $f:\N\to\bool$ and $\neg \forall n. f(n)=\bfalse$, then $\exists n. f(n)=\btrue$.\qed
\end{corlemr1}

It is also well-known that LPO, LLPO, and MP are equivalent to statements about the order and equality of \emph{Cauchy} real numbers.
Specifically:
\begin{enumerate}
\item LPO is equivalent to saying that $\R_C$ has decidable equality.
  (Thus, \cref{thm:lpo,thm:disc-deceq} give an alternative proof of \cref{thm:RC-discrete}.)
\item LLPO is equivalent to $\forall x:\R_C.(x\le 0 \vee x\ge 0)$, hence that $\le$ is a total order on $\R_C$.
\item MP is equivalent to $\forall x:\R_C.(\neg(x\le 0) \to x>0)$.
\end{enumerate}

Traditionally in constructive mathematics, two real numbers $x,y$ are said to be \textbf{apart}, written $x\apart y$, if $|x-y|>0$.
We have $\neg(x\apart y) \leftrightarrow (x=y)$, but in general, $x\apart y$ is stronger than $\neg(x=y)$.
However, since $x=y$ is equivalent to $|x-y|\le 0$, Markov's Principle implies that $x\apart y$ and $\neg(x=y)$ coincide for \emph{Cauchy} reals.

Of course, we are generally more interested in the Dedekind reals $\R$ than the Cauchy ones $\R_C$.
Toby Bartels has suggested the following terminology:
\begin{enumerate}
\item The \textbf{analytic LPO} claims that $\R$ has decidable equality.
\item The \textbf{analytic LLPO} claims that $\forall x:\R.(x\le 0 \vee x\ge 0)$.
\item The \textbf{analytic MP} claims that $\forall x:\R.(\neg(x\le 0) \to x>0)$.
\end{enumerate}

Since (assuming \ref{ax:r3}) $\R$ is not discrete (its identity map is not constant), the analytic LPO is false.
Somewhat more surprisingly, we will prove in \cref{thm:no-allpo} that the analytic LLPO is also false.
The analytic Markov's principle, however, is actually \emph{true} in our motivating model!
I conjecture that it can even be proven from \ref{ax:lem} and \ref{ax:r3}, but I have not managed to show this yet.
Thus, we assume it as an additional axiom.
Actually, I prefer the following axiom, which is equivalent in the presence of \ref{ax:lem}.

\begin{named}{Axiom T}\label{ax:t}
  For any $x:\R$, the \mprop{} $x>0$ is codiscrete.
\end{named}

\begin{thmlemr1t}[Analytic Markov's Principle]\label{ax:amp}
  If $x,y:\R$ satisfy $\neg(x=y)$, then $x\apart y$.
  In particular, if $\neg(x\le 0)$, then $x>0$.
\end{thmlemr1t}
\begin{proof}
  By \ref{ax:t}, $x\apart y$ is codiscrete, so by \cref{thm:codiscrete-notnot-2} it is $\neg\neg$-stable.
  Thus, we can prove it by contradiction.
  However, $\neg(x\apart y)$ means $x=y$, which contradicts the assumption of $\neg(x=y)$.
\end{proof}

\ref{ax:t} is arguably the most mysterious part of the theory.
Topologically, it says that the open subset $(0,\oo)$ of $\R$ is a sub\emph{space}, i.e.\ has the induced topology.
In other words, it relates the intrinsic ``topology'' of $\R$, arising from its definition as a type, to the internal \emph{ordering} relation defined on it.
Another way to look at it is that it ensures that the intrinsic ``topology'' of $\R$ makes it into a \emph{topological field}, i.e.\ that the reciprocal is continuous on the \emph{subspace} of invertible elements.

Some considerations of models may also help to understand \ref{ax:t}.

\begin{eg}\label{thm:axt-holds}
  \ref{ax:t} holds in our motivating model of sheaves on the category $\C$ of cartesian spaces $\R^n$.
\end{eg}
\begin{proof}[Sketch of proof.]
  This requires recalling a bit about the standard proof that $\R$ inherits the correct topology in this model, which goes as follows \parencites[see][\S VI.9]{mm:shv-gl}{lin:shtop}.
  For any space $X\in\C$, we have $\Sh(\C)/X \simeq \Sh(\C/X)$, which admits a local geometric morphism to $\Sh(X)$.
  Since local geometric morphisms are orthogonal to grouplike morphisms \parencite[C3.6.10]{ptj:elephant}, and the space of real numbers is a grouplike locale, to give a real number in $\Sh(\C)/X$ is equivalent to giving a real number in $\Sh(X)$, or equivalently a map $X\to \R$.
  Thus, the sections of the Dedekind real number object over $X$ are the continuous maps $X\to \R$.

  Repeating this argument with $\R$ replaced by the space of \emph{positive} real numbers, which is also grouplike, we find that the object of positive reals in $\Sh(\C)$ is the sheaf of positive real-valued functions.
  To say that this is a codiscrete subobject of the Dedekind reals is to say that a continuous real-valued function factors through the space of positive real numbers just when it is positive at every point, which is obviously true.
\end{proof}

The preceding argument depends crucially on the classicality of the base topos, and in particular the fact that the locale of formal real numbers (which classifies internal Dedekind real numbers) is spatial.
Indeed, we have:

\begin{eg}\label{thm:axt-fails}
  If we start from an arbitrary base topos, and construct a similar topos of sheaves on the category of cartesian \emph{locales} $R_f^n$ (where $R_f$ denotes the locale of formal reals), then \ref{ax:t} can fail.
\end{eg}
\begin{proof}[Sketch of proof.]
  I am indebted to Bas Spitters for this argument.
  There exist toposes (particularly, recursive ones such as the effective topos) in which there are uniformly continuous functions $f:\R\to\R$ such that $f(x)>0$ for all $x$ but $f$ does not have uniform lower bounds on all finite intervals.
  Now since the metric space $\R$ is ``locally compact'' in the sense of Bishop, by \textcite{palmgren:lcmet-loc}, $f$ extends to a locale morphism $\bar{f} : R_f \to R_f$, and since $\R$ is the space of points of $R_f$, we still have $\bar{f}(x)>0$ for all $x$.
  However, $\bar{f}$ does not factor through the locale of formal positive reals, since by \textcite{palmgren:ulb} if it did then it would have local uniform lower bounds.
  Thus, we can repeat the proof of \cref{thm:axt-holds} until the last sentence, at which point we find that the ``obvious'' fact is now false.
\end{proof}

Note that the topos of \cref{thm:axt-fails} is local over its base and satisfies \ref{ax:r3}, by the same argument as before.
%(This does actually require that we use locales rather than spaces, since the \emph{space} of real numbers may not be locally connected constructively.)
Thus, (modulo \cref{rmk:missing-pieces}) \ref{ax:t} does not follow from \ref{ax:r3} alone; but I do not know a countermodel to it that satisfies both \ref{ax:r3} and \ref{ax:lem}.
% However, the topos of \cref{thm:axt-fails} fails to satisfy \ref{ax:lem}, since its base is not classical.
% I conjecture that \ref{ax:t} does follow from \ref{ax:r3} combined with \ref{ax:lem}, but I have not yet been able to prove this.

\section{Shape}
\label{sec:shape}
\addtocontents{toc}{\protect\setcounter{tocdepth}{1}}

Codiscrete types are defined by the fact that functions \emph{into} them need not be continuous; while discrete types are defined by the fact that functions \emph{out} of them need not be continuous.
In a sense, this means that all the \emph{non-tautological} information about codiscrete types is carried by the maps \emph{out} of them, and likewise the non-trivial information about discrete types is carried by the maps \emph{into} them.
Let us leave the first for another day\footnote{But see \cref{thm:CC,thm:ECC}.} and concentrate on maps $A\to B$, where $B$ is discrete.

In the world of classical topological spaces, if $B$ is discrete, then (at least when $A$ is well-behaved, which at the moment means ``locally connected'') continuous maps $A\to B$ are the same as functions $\pi_0(A) \to B$, where $\pi_0(A)$ is the set of connected components of $A$.
In other words, $\pi_0$ is a \emph{reflection} into the subcategory of discrete spaces.
Since $\pi_0(A)$ is also the ``first layer'' of the homotopy type, or fundamental \oo-groupoid, of $A$, this suggests that mapping into discrete spaces can carry information about the latter construction.

In fact, it turns out that the single generalization from sets to \oo-groupoids ensures that mapping into discrete types carries \emph{all} the information about the fundamental \oo-groupoid.
At least in good situations, such as our \oo-topos of sheaves on the category of cartesian spaces $\R^n$, the fundamental \oo-groupoid can be \emph{defined} as a left adjoint to the inclusion of discrete spaces \parencites[see][Proposition 4.3.32]{schreiber:dcct}[\S3]{carchedi:hotyorb}.
With \ref{ax:r4} in our theory, we can get an inkling of how this works internally by considering the topological circle $\topcirc$.

\begin{defn}\label{defn:topcirc}
  The \textbf{topological circle} \topcirc\ is the (homotopy) coequalizer of the pair of maps
  \[\xymatrix{\bR \ar@<1mm>[r]^{\id} \ar@<-1mm>[r]_{+1} & \bR }.\]
\end{defn}

We might write this as $\topcirc \defeq \bR/\Z$.
In \cref{thm:bdry-disc} we will compare this definition to some other possible definitions of the topological circle.
For now let us simply observe that it is at least \emph{one} reasonable definition --- with one caveat, namely that $\topcirc$ ought to be a set (i.e.\ have no higher identifications), and with this definition it is not obvious that this is so.
We could, of course, 0-truncate it, but that would defeat the purpose of what we are about to do, and fortunately it turns out to be unnecessary:

\begin{thmua}\label{thm:free-int-quotient}
  Suppose $R$ is a set and $f:R\simeq R$ an equivalence such that for all $r:R$ and $m,n:\Z$, if $f^n(r) = f^m(r)$, then $m=n$.
  (In other words, the induced action of \Z\ on $R$ is free.)
  Then the homotopy coequalizer of $\xymatrix{R \ar@<1mm>[r]^{\id} \ar@<-1mm>[r]_{f} & R }$ is also a set.
\end{thmua}
\begin{proof}
  Recall that $\hocirc$ denotes the \emph{homotopical} circle, a HIT with two constructors $\base:\hocirc$ and $\floop : \base=\base$.
  We define $P:\hocirc \to \type$ by $P(\base)\defeq R$ and $P(\floop) = f$ (modulo univalence).
  Note that $\hocirc$ is the homotopy coequalizer of $\unit \toto \unit$.
  Thus, by the flattening lemma \parencite[\S6.12]{hottbook}, the desired homotopy coequalizer of $\id_R$ and $f$ is equivalent to $\sm{x:\hocirc} P(x)$; so it will suffice to prove that the latter is a set.

  Let $x,y:\hocirc$ and $u:P(x)$, $v:P(y)$; we must prove that $(x,u)=(y,v)$ is a \mprop{}.
  Now $(x,u)=(y,v)$ is equivalent to $\sm{p:x=y} (p_* u = v)$.
  Thus, suppose $p,q:x=y$ and $r:p_*u=v$ and $s:q_*u=v$; we must show $(p,r)=(q,s)$.
  Since $r$ and $s$ are equalities in a set, it suffices to show $p=q$.
  But this is also an equality in a set (since $\hocirc$ is a 1-type), hence a \mprop{}; thus, since $\hocirc$ is connected, we may assume $x=\base$ and $y=\base$.

  Now under the isomorphism $(\base=\base)\simeq \Z$, our $p$ and $q$ get identified with $n,m:\Z$ respectively.
  Moreover, by definition of $P$, transporting along $p$ or $q$ gets identified with application of $f^n$ or $f^m$ respectively.
  Thus, from $r$ and $s$ we obtain $p_* u = q_* u$ and hence $f^n(u) = f^m(u)$, with $u:R$, and so the freeness assumption gives $n=m$ as needed.
\end{proof}

\begin{rmk}
  It should be possible to prove \cref{thm:free-int-quotient} using univalence only for \mprop{}s \parencite[e.g.\ with][Theorem 7.2.2]{hottbook}.
  But the preceding proof is easier.
\end{rmk}

\begin{corua}
  The topological circle $\topcirc$ is a set.
\end{corua}
\begin{proof}
  Since $\bR$ is an abelian group, $(+1):\bR\to\bR$ is an equivalence.
  The freeness condition in \cref{thm:free-int-quotient} is equivalent to saying that the unique ring homomorphism $\Z\to \bR$ is injective, which just says that $\bR$ has characteristic 0.
\end{proof}

The following proof contains the crucial idea involved in the adjoint characterization of fundamental \oo-groupoids.

\begin{thmr3}\label{thm:circ-circ}
  For any discrete type $X$, we have $(\topcirc \to X) \simeq (\hocirc \to X)$.
\end{thmr3}
\begin{proof}
  The universal property of $\hocirc$ means that $(\hocirc \to X)$ is equivalent to
  \begin{equation}
    \sm{x:X} (x=x).\label{eq:circ-circ-0}
  \end{equation}
  But by the universal property of a coequalizer, $(\topcirc\to X)$ is equivalent to
  \begin{equation}
    \sm{g:\bR\to X} (g\circ (+1) = g).\label{eq:circ-circ-1}
  \end{equation}
  Now $g\circ (+1) = g$ is an equality in the type $\bR\to X$, which is equivalent to $X$ since $X$ is discrete.
  The equivalence can be implemented by evaluating at any $r:\bR$, such as $0$; thus~\eqref{eq:circ-circ-1} is equivalent to
  \begin{equation}
    \sm{g:\bR\to X} (g(1) = g(0)).\label{eq:circ-circ-2}
  \end{equation}
  But since $(\bR\to X)\simeq X$, we may assume $g$ is constant at some $x:X$, giving~\eqref{eq:circ-circ-0}.
\end{proof}

In particular, the identity map of $\hocirc$ corresponds to a nontrivial map $\topcirc \to\hocirc$, which may be said informally to ``wrap the topological circle around the homotopical one''.
Since \cref{thm:circ-circ} is easily shown to be natural, this map exhibits $\hocirc$ as a \emph{reflection} of $\topcirc$ into the discrete types.
(Recall from \cref{thm:s1-discrete} that $\hocirc$ is \emph{topologically} discrete.)

Since $\hocirc$ is, or should be, the fundamental \oo-groupoid of $\topcirc$, this leads us to ask: does every type have a reflection into the discrete types?
The answer is yes, and it requires only the much weaker \ref{ax:r0}.

\begin{defnr0}
  For any type $A$, its \textbf{shape} $\shape A$ is the higher inductive type with the following five constructors.
  \begin{enumerate}
  \item $\sigma_A:A\to \shape A$
  \item $\kappa_A : \prd{i:I} (R_i\to \shape A) \to \shape A$
  \item For all $i:I$ and $g:R_i\to\shape A$ and $x:R_i$, an equality $g(x) = \kappa_A(i,g)$.
  \item $\kappa'_A : \prd{i:I} (R_i\to \shape A) \to \shape A$
  \item For all $i:I$ and $x:\shape A$, an equality $x = \kappa'_A(i,\const_x)$.
  \end{enumerate}
\end{defnr0}

\begin{rmk}\label{rmk:shape}
We use the word ``shape'' not only because it is nine syllables shorter than ``fundamental \oo-groupoid'', but because the latter has various connotations that we want to be free of.
In particular, the shape is a purely internal construction in our ``real-cohesive type theory'', although it generally \emph{behaves like} the fundamental \oo-groupoid.
Moreover, the fundamental \oo-groupoid is usually notated $\Pi_\oo$, but since $\Pi$ is rather overworked in type theory already, a different symbol is preferable.
The word ``shape'' comes from \emph{shape theory}, which also studies generalizations of the fundamental \oo-groupoid that make sense for less well-behaved spaces.
The symbol $\shape$ is not an integral sign ($\int$) but an ``esh'', the IPA sign for a voiceless postalveolar fricative (English \textit{sh}).
\end{rmk}

The definition of $\shape A$ is cooked up precisely to admit a map from $A$ and to be discrete:

\begin{lemr0}\label{thm:discrete-shape}
  $\shape A$ is discrete.
\end{lemr0}
\begin{proof}
  The last four constructors say exactly that $\shape A\to (R_i\to \shape A)$ has both a left and a right inverse.
\end{proof}

It is unsurprising, therefore, that $\shape A$ has the desired universal property.
We express this first as a \textbf{$\shape$-induction} principle.

\begin{thmr0}\label{thm:shape-induction}
  If $P:\shape A \to \type$ is a family of discrete types, and we have $d:\prd{x:A} P(\sigma(x))$, then we have $f:\prd{y:\shape A} P(y)$ such that $f(\sigma(x))=d(x)$ for all $x:A$.
\end{thmr0}
\begin{proof}
  The ``basic'' induction principle of $\shape A$ arising from its higher inductive definition says that given $Q:\shape A \to \type$ together with
  \begin{enumerate}
  \item $d_\sigma:\prd{x:A} Q(\sigma(x))$
  \item $d_\kappa:\prd{i:I}{g:R_i\to\shape A}{h:\prd{x:R_i} Q(g(x))} Q(\kappa(i,g))$
  \item $d_= : \prd{g:\R\to\shape A}{h:\prd{x:\R} Q(g(x))}{x:\R} (h(x) =_{\xi} d_\kappa(i,g,h))$, where $\xi$ is the third constructor of $\shape A$,
  \item $d_{\kappa'}:\prd{i:I}{g:R_i\to\shape A}{h:\prd{x:R_i} Q(g(x))} Q(\kappa'(i,g))$
  \item $d_{='} : \prd{i:I}{x:\shape A}{y:Q(x)} (y =_{\zeta} d_{\kappa'}(i,\const_x,\const_y))$, where $\zeta$ is the last constructor of $\shape A$,
  \end{enumerate}
  we have $f:\prd{y:\shape A} P(y)$ such that $f(\sigma(x))\jdeq d_\sigma(x)$ (plus four other equalities).
  Thus, given our current hypotheses, it remains to construct $d_\kappa$, $d_=$, $d_{\kappa'}$, and $d_{='}$.

  First suppose ${g:R_i\to\shape A}$ and ${h:\prd{x:R_i} Q(g(x))}$, for some $i$.
  By the third constructor of $\shape A$, we have $g(x) = \kappa(i,g)$ for all $x:R_i$.
  Thus, transporting $h(x)$ along these equalities, we get $h' : R_i \to Q(\kappa(i,g))$.
  Thus, since $Q(\kappa(i,g))$ is discrete, there is a $d_\kappa(i,g,h):Q(\kappa(i,))$ such that $d_=(i,g,h,x)  :h'(x) = d_\kappa(i,g,h)$ for all $x:\R$, as desired.

  Now suppose $x:\shape A$ and $y:Q(x)$.
  By the last constructor of $\shape A$, we have $x = \kappa(i,\const_x)$.
  Thus, transporting $y$ along this equality, we obtain $y' : Q(\kappa(i,\const_x))$.
  Since $Q(\kappa(i,\const_x))$ is discrete, we have $d_{='}(i,x,y) : y' = d_{\kappa'}(i,\const_x,\const_y)$ as desired.
\end{proof}

\begin{corr0}
  $\shape R_i$ is contractible for any $i:I$.
\end{corr0}
\begin{proof}
  It is inhabited by $\kappa_{R_i}(i,\sigma_{R_i})$, so it remains to show $x=\kappa_{R_i}(i,\sigma_{R_i})$ for all $x:\shape R$.
  Since $\shape R_i$ is discrete, so is this equality type; so by $\shape$-induction, it suffices to prove $\sigma_{R_i}(x) = \kappa_{R_i}(i,\sigma_{R_i})$ for all $x:R_i$.
  But this follows from the third constructor of $\shape R_i$.
\end{proof}

\begin{corr3}\label{thm:shape-R-contr}
  $\shape\R$ is contractible.\qed
\end{corr3}

\begin{corr1}\label{thm:shape-surj}
  For any $A$, the map $\sigma:A\to \shape A$ is surjective.
  In particular, if $\brck{\shape A}$ then $\brck{A}$.
  % we have $\brck{\shape A} \to \brck{A}$ (hence $\brck{\shape A} = \brck{A}$).
  % That is, $\shape$ reflects inhabitedness.
\end{corr1}
\begin{proof}
  Given $y:\shape A$, define $P(y) \defeq \brck{\sm{x:A} \sigma(x)=y}$.
  By \cref{thm:discrete-prop}, each $P(y)$ is discrete.
  Thus, by \cref{thm:shape-induction}, to prove $\prd{y:\shape A} P(y)$ (i.e.\ that $\sigma$ is surjective), it suffices to prove $\prd{x:A} P(\sigma(x))$; but this is obvious.
\end{proof}

\begin{rmk}\label{thm:no-crisp-shape-ind}
  Recall from \cref{sec:flat} that many positive type formers have a ``crisp'' variant of their induction principle, so that we can do case analysis on a $u::A+B$ or induction on a $u::\N$.
  This is \emph{not} the case for $\shape$, however.
  If it were, then we could do $\shape$-induction on $u::\shape A$, so that if $B$ were discrete then any $M:B$ depending on $u::A$ would factor through $\shape A$.
  Specializing to $A\defeq\R$ and $B\defeq \bool$, any $M:\bool$ depending on $u::\R$ would be constant.
  However, since $u$ is crisp, by using the crisp LEM we can define $M:\bool$ to be $\bfalse$ if $u<0$ and $\btrue$ otherwise, which is not at all constant.
  % In particular, this means that $\flat\shape A$ does not admit a ``combined induction principle'', despite the apparent ``positivity'' of both $\flat$ and $\shape$.
\end{rmk}

Now we have the reflection property:

\begin{thmr0}\label{thm:discrete-reflective}
  For any $A:\type$ and any discrete $B:\type$, composition with $\sigma : A \to \shape A$ induces an equivalence
  \[ (\shape A \to B) \simeq (A\to B). \]
\end{thmr0}
\begin{proof}
  $\shape$-induction into the constant family $\lam{x} B$ yields a section of $(\blank\circ \sigma)$.
  Thus, it suffices to show that given $h,k:\shape A \to B$, if $h\circ\sigma =k\circ \sigma$, then $h=k$.
  This follows from another $\shape$-induction with $P(x) \defeq (h(x)=k(x))$, which is discrete by \cref{thm:discrete-path}.
\end{proof}

Moreover, analogously to how $\flat\dashv \sharp$ crisply (\cref{thm:flat-sharp-adj}), we have $\shape\dashv\flat$ crisply.

\begin{thmr0}\label{thm:shape-flat-adj}
  For any $A,B::\type$, there is a natural equivalence
  \[ \flat(\shape A \to B) \simeq \flat(A\to \flat B). \]
\end{thmr0}
\begin{proof}
  Since $\flat B$ is discrete, by \cref{thm:discrete-reflective} we have a natural equivalence
  \[ (\shape A \to \flat B) \simeq (A \to \flat B) \]
  and this equivalence is preserved by $\flat$.
  On the other hand, since $\shape A$ is discrete, by \cref{thm:discrete-coreflective} we have a natural equivalence
  \[ \flat (\shape A\to \flat B) \simeq \flat (\shape A\to B). \]
  Composing these two equivalences yields the conclusion.
\end{proof}

\cref{thm:discrete-reflective} makes $\shape$ into a monadic modality, like $\sharp$, and
its construction as a localization makes it an \emph{accessible} one \parencite[see][]{rss:modalities}.
Unlike $\sharp$, however, $\shape$ is not left exact.
To show this, we introduce our final weakening of \ref{ax:r4}.

\begin{named}{Axiom C2}\label{ax:r2}
  \ref{ax:r1} holds, and moreover there exist $i_0:I$ and $r_0,r_1:R_{i_0}$ such that $R_{i_0}$ is a set and $r_0 \neq r_1$.
\end{named}

Note that \ref{ax:r3} implies \ref{ax:r2}, since $0\neq 1$ in \R.

\begin{thmr2}\label{thm:shape-not-lex}
  There exists a pullback that is not preserved by $\shape$.
\end{thmr2}
\begin{proof}
  The pullback of the inclusions $r_0, r_1 : \unit \to R_{i_0}$ is $\emptyset$.
  But $\shape\unit$ and $\shape R_{i_0}$ are both contractible, whereas $\shape\emptyset=\emptyset$ since it is already discrete.
\end{proof}

% (We will return to \ref{ax:r2} in \cref{sec:pieces-have-points}.)

(However, like any monadic modality, $\shape$ does preserve products.)

Viewing $\shape$ as a reflection gives another, more category-theoretic, way to phrase and prove \cref{thm:circ-circ}.

\begin{thmr3}\label{thm:circ-circ-2}
  $\shape\topcirc = \hocirc$.
\end{thmr3}
\begin{proof}
  By \cref{thm:discrete-reflective}, $\shape$ is a left adjoint to the inclusion of discrete types in all types.
  Thus, it preserves all colimits, which is to say it takes colimits in the category of all types to colimits in the category of discrete types.
  Thus, since $\topcirc$ is the (homotopy) coequalizer of $\R \toto \R$, it follows that $\shape \topcirc$ is the coequalizer \emph{in the category of discrete types} of $\shape\R \toto \shape \R$, and hence by \cref{thm:shape-R-contr} of $\unit \toto \unit$.
  Now $\hocirc$ is the coequalizer of $\unit \toto \unit$ in the category of \emph{all} types; but by \cref{thm:s1-discrete} it lies in the subcategory of discrete types and hence is also the coequalizer there.
\end{proof}

Similar methods may be used in many other examples.
That is, given a ``cell complex'' presentation of a classical topological space, if we can convert it into both a specification for a HIT and a colimit decomposition of that space that is sufficiently ``cofibrant'', then $\shape$ will preserve that colimit and take the space to the HIT.
We do have to be careful to avoid the sort of cell complex used in classical algebraic topology where discs are glued along their \emph{boundaries}, because such gluing in a constructive world tends to produce inhomogeneous ``cusps''.
For instance, identifying the two endpoints of the topological interval $[0,1] = \{ x:\R \mid 0\le x\le 1 \}$ would not\footnote{Although, amusingly, its shape would still be $\hocirc$: it would be the coequalizer of $1\toto [0,1]$, while $\shape [0,1]$ is also contractible since $[0,1]$ is a retract of \R.} produce a space equivalent to $\topcirc$.
This is also a known ``defect'' (if one regards it so) of the topos model we have in mind; see also the remarks after \cref{thm:no-allpo}.
(It is ``fixed'' in the topological topos of \textcite[\S6]{ptj:topological-topos}, but in that model no left adjoint $\shape$ can exist due to a lack of local connectedness.)

However, in many cases it should be possible to remedy this by gluing along \emph{open overlaps} instead.
What makes this work is the fact that although $\R$ fails to satisfy the trichotomy principle constructively, we do have $(x<z) \to (x<y)\lor (y<z)$ for all $x,y,z:\R$.
Combined with the following observation, this tells us that gluing along open overlaps produces no cusps.

\begin{lem}\label{thm:join-or}
  If $f:A\to C$ and $g:B\to C$ are embeddings, then their union as sub-types of $C$ is their pushout under their intersection: $A\cup B \simeq A \sqcup^{A\cap B} B$.
\end{lem}
\begin{proof}
  By definition, $f\cup g : A\cup B\to C$ is the smallest embedding through which both $f$ and $g$ factor, while $A\cap B$ is the pullback $A\times_C B$.
  If $h:D\to C$ is an embedding through which both $f$ and $g$ factor, then the square
  \[ \xymatrix@-.5pc{ A\cap B \ar[r] \ar[d] & B \ar[d] \\ A \ar[r] & D } \]
  commutes since $h$ is an embedding, so there is an induced map $A \sqcup^{A\cap B} B \to D$.
  It remains to note that $A \sqcup^{A\cap B} B\to C$ is an embedding by \textcite[Lemma 2.4]{rijke:join}.
\end{proof}

As an example, we give two other definitions of $\topcirc$.

\begin{thm}\label{thm:topcircles}\label{thm:bdry-disc}
  The following types are equivalent.
  \begin{enumerate}
  \item $\topcirc$ as defined in \cref{defn:topcirc}, i.e.\ the coequalizer of the identity and $+1:\R\to\R$.\label{item:tc1}
  \item The subset $\{ (x,y) \mid x^2+y^2=1\}$ of $\R\times \R$ (the boundary of the topological disc $\topdisc$).\label{item:tc2}
  \item The coequalizer of the inclusion of open intervals $(0,\varepsilon) \to (0,1)$ and the translated inclusion $(\lam{x} 1-\varepsilon+x):(0,\varepsilon) \to (0,1)$, for any $0<\varepsilon<\frac12$.\label{item:tc3}
  \end{enumerate}
\end{thm}
\begin{proof}
  We can define the functions $\sin$ and $\cos$ on the Dedekind reals constructively, and prove that $\sin^2 x + \cos^2 x = 1$.
  Thus, $\lam{x} (\cos (2\pi x),\sin (2\pi x))$ defines a map from $\R$ to $\{ (x,y) \mid x^2+y^2=1\}$, which by the periodicity of $\sin$ and $\cos$ descends to $\topcirc$.
  It is then straightforward to verify that this map is both surjective and injective, hence an equivalence since both types are sets.
  Thus~\ref{item:tc1} and~\ref{item:tc2} are equivalent.

  Now let $0<\varepsilon<\frac12$.
  Every real number is either $<\varepsilon$ or $>-\varepsilon$, so the set $\{ (x,y) \mid x^2+y^2=1\}$ in~\ref{item:tc2} is the union of its subsets $U$ and $V$ consisting of the points where $y<\varepsilon$ and where $y>-\varepsilon$.
  Thus, by \cref{thm:join-or} it is the pushout of $U$ and $V$, each of which is isomorphic to an open interval, under their intersection, which is isomorphic to the disjoint union of two open intervals, included into $U$ and $V$ at their ends.
  Since (again by \cref{thm:join-or}) the pushout of two open intervals under their intersection is their union, we can rearrange this pushout to become~\ref{item:tc3}.
\end{proof}

Similarly, we can obtain the topological sphere $\topsphere$ by gluing two topological discs $\topdisc$ (isomorphic to $\R^2$, hence with contractible shape) along a strip $\topcirc \times \R$ (with shape $\hocirc$).
Thus $\shape\topsphere$ is the homotopical suspension of $\hocirc$, i.e.\ the homotopical sphere $\hosphere$.
We can proceed inductively for higher spheres, and so on.
In this way it may be possible to access homotopy types that would be quite difficult to present as HITs (due, for example, to their having infinitely many constructors), by first constructing their topological versions as sets using the Dedekind reals and then applying $\shape$.

\begin{rmk}\label{rmk:shape-overall}
  The overall conclusions drawn from \ref{ax:r0} in this section and \cref{sec:real-cohesion} can be summarized as ``$\shape$ is a reflector into the same subcategory that $\flat$ coreflects into, and it preserves finite products''.
  As with the analogous conclusions about $\sharp$ and $\flat$ noted in \cref{rmk:sharp-overall,rmk:flat-overall}, in \textcite{ss:qgftchtt} we assumed this axiomatically, while here we have deduced it from \ref{ax:r0} and a higher inductive definition of $\shape$.
  This difference from \textcite{ss:qgftchtt} is roughly orthogonal to our introduction of $\flat$ and $\sharp$ using crisp variables; it is not necessitated by anything, but it is a convenient starting point for expressing the stronger \ref{ax:r1}, \ref{ax:r2}, and \ref{ax:r3}.
  These axioms were not considered in \textcite{ss:qgftchtt}; but in \cref{sec:pieces-have-points} we will relate the first two to corresponding categorical properties studied by \textcite{lawvere:cohesion,johnstone:punctual-lc,lm:ac-cohesive}.
\end{rmk}

\begin{rmk}\label{rmk:why-flat}
  The crisp adjointness $\shape\dashv\flat$ from \cref{thm:shape-flat-adj} finally allows us to justify the choice of the notations $\flat$ and $\sharp$ (the notation and justification are both due to \textcite{schreiber:dcct}).
  Suppose $G$ is a group, meaning a set (0-type) with a group structure (but, like all types in spatial type theory, also possessing intrinsic topology).
  We can construct a delooping $\mathbf{B}G$ (a.k.a.\ $K(G,1)$) by the methods of \textcite{lf:emspaces}.

  Now if $X$ is any space, a \emph{principal $G$-bundle over $X$ with a flat connection} assigns to every (topological) path in $X$ an element of $G$ (``transport'' along that path), such that concatenating paths corresponds to multiplication in $G$, and so forth.
  Since paths in $X$ and elements of $G$ are the 1-morphisms in the fundamental \oo-groupoid of $X$ and in $\mathbf{B}G$, respectively, such a bundle with flat connection can be described categorically as a map $\shape X \to\mathbf{B}G$.
  But by \cref{thm:shape-flat-adj}, this is equivalent to a map $X\to \flat\mathbf{B}G$.
  Thus, $\flat\mathbf{B}G$ is the \emph{moduli space of flat $G$-connections}, justifying the notation $\flat$.
  The notation $\sharp$ is chosen simply as dual to flat; although one can argue that if $Z$ classifies bundles of some sort then $\sharp Z$ classifies analogous bundles that are ``sharp'' in that their fibers can vary discontinuously.
  (Unfortunately, as noted in \cref{thm:abs}, in \textcite{ab:ax-local} and \textcite{abs:lrt-modal-comput-ea} the symbols $\flat$ and $\sharp$ were used for the propositional restriction of the same operations, but with reversed meanings.)
\end{rmk}

\section{Axioms of cohesion}
\label{sec:pieces-have-points}

The axioms introduced in \cref{sec:real-cohesion,sec:shape} form a linear string of implications:
\begin{center}
  R$\flat$ $\Rightarrow$ C2 $\Rightarrow$ C1 $\Rightarrow$ C0.
\end{center}
Our main goal (the Brouwer fixed point theorem) requires \ref{ax:r4} (or at least its consequence \cref{thm:shape-R-contr}).
However, the weaker versions are much more general assumptions, which are not very specific to the ``topological'' situation and are satisfied in many other models, so it is interesting that they suffice for most of the general theory.
In this somewhat digressive section, we investigate some of their further consequences and their topos-theoretic meaning.

\subsection{\ref{ax:r0} means stable local connectedness}
\label{sec:c0}

Combining \cref{thm:flat-sharp-adj,thm:shape-flat-adj}, we see that \ref{ax:r0} gives us a crisp ``adjoint triple'' $\shape \dashv \flat\dashv \sharp$, in which $\shape$ preserves finite products.
In topos-theoretic language, the category of types is ``local and locally connected'' over the category of discrete (or, equivalently, codiscrete) types, and in addition the left adjoint preserves binary products.
\textcite{johnstone:punctual-lc} called this additional condition \emph{stable local connectedness}.

Conversely, if a topos is local and stably locally connected, then by combining \textcite[C3.6.3(d)]{ptj:elephant} with \textcite[Proposition 1.3]{johnstone:punctual-lc}, it has a site of definition that is local and locally connected and closed under finite products.
We should therefore be able to show that \ref{ax:r0} holds in its internal spatial type theory by taking $I$ to be the discrete set of objects of the site and $R_i$ the corresponding representable.
The condition that $A\to (R_i\to A)$ be an equivalence then says, in terms of sheaves, that $A(U) \to A(R_i\times U)$ is an isomorphism (or, in the higher-topos case, an equivalence) for any object $U$ of the site (here we use the fact that the site has finite products).
The special case $U=1$ tells us that $A$ is a constant (pre)sheaf, and the rest of the conditions follow automatically.
(We could also take $I$ to be a set of objects that generate the site under finite products, as \R\ does for our site of cartesian spaces $\R^n$.)

Thus, \ref{ax:r0} is exactly a type-theoretic incarnation of stable local connectedness (as an additional property added to a local topos).
\textcite{schreiber:dcct} calls a local and stably locally connected \oo-topos \emph{cohesive} (although for \textcite{lawvere:cohesion} the word ``cohesive'' also includes the categorical version of \ref{ax:r1}, below, and an additional axiom called ``continuity'' that we do not consider here).

In addition to our motivating example (sheaves on cartesian spaces), several other similar examples of local and stably locally connected toposes and \oo-toposes can be found in \textcite{schreiber:dcct}.
Here are a few more:
\begin{itemize}
\item By \textcite[Lemma 1.1]{johnstone:punctual-lc} and \textcite[C3.6.3(a) and C3.6.17(a)]{ptj:elephant}, the topos $\Sh(X)$ of sheaves on a space (or locale) is local and stably locally connected if and only if $X$ has both a focal point and a dense point.
  For example, in the Sierpinski space, the closed point is focal and the open point is dense.
\item The local \oo-topos of parametrized spectra mentioned in \cref{rmk:local-toposes} is stably locally connected, with $\shape=\sharp$.
  In fact, it is also punctually locally connected (see below).
\item \textcite{rezk:global-cohesion} shows that ``global equivariant homotopy theory'' forms a local and stably locally connected \oo-topos.
\end{itemize}

\subsection{\ref{ax:r1} means punctual local connectedness}
\label{sec:c1}

If we additionally assume \ref{ax:r1}, then by \cref{thm:discrete-subobject} the discrete objects are closed under subobjects.
\textcite{johnstone:punctual-lc} shows that this is equivalent to the following condition, there called  \emph{punctual local connectedness} (\textcite{lawvere:cohesion} calls it the \emph{Nullstellensatz}; and \textcite{schreiber:dcct} calls it \emph{pieces have points}).

\begin{thmr1}[{\cite[Lemma 2.3]{johnstone:punctual-lc}}]\label{thm:pieces-have-points}
  For any $A::\type$, the composite $\flat A \to A \to \shape A$ is surjective.
\end{thmr1}
\begin{proof}
  As in \textcite[Theorem 8.8.1]{hottbook}, it suffices to show that $\trunc0{\flat A} \to \trunc0{\shape A}$ is surjective.
  Thus, since epimorphisms of sets are surjections \parencite[Lemma 10.1.4]{hottbook}, it suffices to show (invoking the universal property of $\trunc0{\blank}$) that $(\shape A\to C) \to (\flat A\to C)$ is injective for any set $C$.
  Moreover, since in the proof of Lemma 10.1.4 from \textcite{hottbook} we only need to instantiate this property with one particular set $C$, in fact it suffices to show that
  \begin{equation}
    \flat (\shape A\to C) \to \flat (\flat A\to C)\label{eq:php1}
  \end{equation}
  is injective for any \emph{crisp} set $C$.
  And by \cref{thm:discrete-coreflective}, we may as well assume $C$ is also discrete (otherwise we could apply $\flat$ to it).

  Now we can essentially copy the proof of Lemma 2.3 from \textcite{johnstone:punctual-lc}.
  Firstly, note that~\eqref{eq:php1} is equal to the composite
  \[ \flat (\shape A\to C) \xto{\simeq} \flat (A\to C) \xto{\;\flat\;} \flat(\flat A \to \flat C) \xto{\simeq} \flat(\flat A \to C) \]
  where the first equivalence is \cref{thm:discrete-reflective} and the last is because $C$ is discrete.
  Thus, it will suffice to show that the middle map is injective.
  But this map is equal to the composite
  \[ \flat (A\to C) \longrightarrow \flat(A\to \sharp C) \xto{\simeq} \flat(\flat A \to C) \xto{\simeq} \flat(\flat A \to \flat C) \]
  where the middle equivalence is \cref{thm:flat-sharp-adj} and the last is because $C$ is discrete.
  But $C\to \sharp C$ is injective by \cref{thm:crisp-discrete-concrete}.
\end{proof}

Thus, \ref{ax:r1} implies punctual local connectedness.
Conversely, by \textcite[Proposition 1.4]{johnstone:punctual-lc}, any punctually locally connected topos has a site of definition that is local and locally connected and in which every object has a global point.
Clearly in this case if we take $I$ and $R_i$ as described above, then \ref{ax:r1} will hold; thus it is exactly a type-theoretic incarnation of punctual local connectedness.

% Of course, the map of \cref{thm:pieces-have-points} is not generally an equivalence.
% For instance, under \ref{ax:r2}, $\flat R_{i_0}$ has at least two points whereas $\shape R_{i_0}$ is contractible (see also~\cite[Proposition 3.7]{johnstone:punctual-lc}).

In the different terminology of \textcite{lm:ac-cohesive}, \cref{thm:pieces-have-points} (along with our previous results) means that the category of types is ``pre-cohesive'' over that of (co)discrete types.
In fact, it is ``stably pre-cohesive'', because all of our theorems admit an arbitrary crisp context.

\subsection{\ref{ax:r2} means contractible codiscreteness}
\label{sec:c2}

Following \textcite{lm:ac-cohesive} further, we investigate the shapes of codiscrete types, for which purpose we bring back \ref{ax:r2} (page~\pageref{ax:r2}).

\begin{thmlemr2}\label{thm:CC}
  If $A$ is codiscrete, then $\shape A$ is a \mprop{}.
\end{thmlemr2}
\begin{proof}
  Let $x,y:\shape A$; we must show $x=y$.
  By $\shape$-induction, we may assume that $x=\eta(u)$ and $y=\eta(v)$ for $u,v:A$.
  Let $i_0$ and $r_0,r_1:R_{i_0}$ be as in \ref{ax:r2}.
  Since they are axioms, they are crisp, and so we have $r_0{}^\flat : \flat R_{i_0}$ and $r_1{}^\flat : \flat R_{i_0}$.
  By \cref{thm:equiv-flat-path}, $(r_0{}^\flat = r_1{}^\flat)$ is equivalent to $\flat(r_0 = r_1)$, i.e.\ to $\flat\emptyset$ and hence to $\emptyset$.
  In other words, we have $(r_0{}^\flat \neq r_1{}^\flat)$.

  Now since $\flat R_{i_0}$ is a crisp discrete set, by the flat LEM it has decidable equality.
  Thus, there is a map $f:\flat R_{i_0} \to A$ with $f(r_0^\flat) = u$ and $f(r_1{}^\flat) =v$.
  But the inclusion $\flat R_{i_0} \to R_{i_0}$ is inverted by $\flat$ and hence also by $\sharp$, so since $A$ is codiscrete our $f$ extends to a map $g:R_{i_0}\to A$ with $g(r_0) = u$ and $g(r_1) =v$.
  Now the composite $g\circ \eta : R_{i_0} \to \shape A$ has $g(r_0)=x$ and $g(r_1)=y$; but as $\shape A$ is discrete, $g$ is constant, so $x=y$.
\end{proof}

\begin{corlemr2}\label{thm:ECC}
  For any crisp $B$ we have $\shape \sharp B = \brck{B}$.
\end{corlemr2}
\begin{proof}
  For any $B$ we have $B \to \shape \sharp B$, hence a map $\brck{B} \to \shape \sharp B$.
  But $\flat B \simeq \flat \sharp B \to \shape \sharp B$ is surjective by \cref{thm:pieces-have-points}, and factors through $\brck{B}$.
  Hence our map $\brck{B} \to \shape \sharp B$ is also surjective, thus an equivalence.
\end{proof}

\textcite{lm:ac-cohesive} call the conclusion of \cref{thm:CC} \emph{connected codiscreteness}, and \cref{thm:ECC} (proven there as Lemma 7.3) \emph{explicit connected codiscreteness}.
In a homotopical context, however, it seems better to say \emph{contractible} codiscreteness.
% , but fortunately the abbreviations CC and ECC will do for both.
As before, since we can work in an arbitrary crisp context, we automatically have ``stable contractible codiscreteness''.

\begin{rmk}
  Unfortunately, the phrase ``$X$ is contractible'' is used by \textcite{lm:ac-cohesive} to mean ``$\trunc0{\shape (Y\to X)}\simeq \unit$ for all $Y$'', which conflicts with its standard usage in homotopy type theory to mean ``$X\simeq \unit$''.
  (They don't mention the 0-truncation since they are working with 1-toposes.)
  The intution behind their definition is that the sets $\trunc0{\shape (Y\to X)}$ are the hom-sets of a ``strong homotopy category'' of types, analogous to the strong homotopy category of topological spaces obtained by identifying homotopic maps, so this definition says that $X$ is terminal in that category.
  Since $\shape$ is a functor from this strong homotopy category to the category of discrete spaces and their 0-truncated hom-sets, and it preserves the terminal object, if $X$ is contractible in the sense of \textit{ibid.}\ then $\shape X \simeq \unit$, i.e.\ $\shape X$ is contractible in the sense of homotopy type theory.
  Thus, if we want to keep ``$X$ is contractible'' to mean $X\simeq \unit$, we could use a phrase like \emph{strongly spatially contractible} for the ``contractibility'' of \textcite{lm:ac-cohesive}.
\end{rmk}

Conversely, if a punctually locally connected topos satisfies \cref{thm:CC}, then we have an object, namely $\sharp \bool$, for which $\shape\sharp\bool$ is contractible, but which has two unequal elements $\bfalse^\sharp$ and $\btrue^\sharp$.
(By \cref{thm:path-sharp} we have $(\bfalse^\sharp = \btrue^\sharp) \simeq \sharp(\bfalse=\btrue) \simeq \sharp \emptyset$, which is empty by \cref{thm:emptyset-codiscrete}.)
Thus, we can make \ref{ax:r2} hold by adding $\sharp\bool$ to the collection of $R_i$'s; so \ref{ax:r2} is exactly a type-theoretic incarnation of (punctual local connectedness and) contractible codiscreteness.

We could now copy the proof of \textcite[Corollary 6.5]{lm:ac-cohesive} to show that every type $A$ embeds into a type $B$ such that $\shape B$ is contractible, there called \emph{sufficient cohesion}.
I conjecture that \ref{ax:r3} and \ref{ax:ac} should together also imply the property called \emph{continuity} in \textcite{lawvere:cohesion,menni:contcoh} --- that the canonical map $\shape(A^S) \to (\shape A)^S$ is an equivalence for any crisp set $A$ and any crisp discrete set $S$ --- but I have not managed to prove this yet.

\ref{ax:r3} does not seem to have been studied yet in a topos-theoretic context.
Of course, it is much more restrictive even than \ref{ax:r2}; rather than defining a general class of toposes for investigation, it singles out an important characteristic of a single topos (or a small group of toposes) that we are already interested in.

\section{Synthetic topology}
\label{sec:continuity}

As a warm-up for Brouwer's fixed-point theorem, in this section we develop some aspects of synthetic topology that stay in the world of sets (0-types).
Here we can already see the power of \ref{ax:r3}, even before any homotopy theory enters: it forces $\R$ to be \emph{connected} in several senses.
One such sense is the following.

\begin{thmr3}\label{thm:compact-connected} %[\R\ is compact connected]
  If $P:\R\to \prop$ is a ``detachable subset'' (i.e.\ for all $x:\R$ we have $P(x) \vee \neg P(x)$), then either $\forall x,\, P(x)$ or $\forall x,\, \neg P(x)$.
\end{thmr3}
\begin{proof}
  A detachable $P:\R\to \prop$ yields a function $f:\R\to\bool$, where $f(x)=\btrue$ if $P(x)$ and $f(x)=\bfalse$ if $\neg P(x)$.
  But $\bool$ is discrete, so by \ref{ax:r3} $f$ is constant.
\end{proof}

Put differently, if $U$ and $V$ are subsets of \R\ such that $U\cup V = \R$ and $U\cap V = \emptyset$, then either $U=\R$ or $V=\R$.
In other words, \R\ cannot be ``broken into two pieces'' nontrivially.
\textcite{taylor:lamcra} calls this property \emph{compact connectedness}.
Combining this with \ref{ax:lem}, and also \ref{ax:t} from \cref{sec:ac-2}, we can prove a version of the Intermediate Value Theorem (which is a sort of one-dimensional analogue of Brouwer's fixed-point theorem).

\begin{thmlemr3t}[Discontinuous IVT]\label{thm:discont-ivt} %; {\cite[Proposition 13.5]{taylor:lamcra}}]
  Let $f::\R\to\R$ be a crisp function and $c::\R$ a crisp real number.
  If there exist $a,b:\R$ such that $f(a)<c<f(b)$, then there exists an $x:\R$ such that $f(x)=c$.
\end{thmlemr3t}
\begin{proof}
  Since the hypothesis of $a$ and $b$ is a crisp \mprop{}, it is discrete.
  Thus all our hypotheses may as well be crisp, so we may prove $\sharp\brck{\sm{x:\R} f(x)=c}$ instead, which by \cref{thm:codiscrete-notnot-2} is equivalent to $\neg\neg{\sm{x:\R} f(x)=c}$.
  Thus, suppose for contradiction $\neg{\sm{x:\R} f(x)=c}$, i.e.\ for all $x:\R$ we have $f(x)\neq c$.
  By \cref{ax:amp}, for all $x$ we have $f(x)\apart c$, hence either $f(x)<c$ or $f(x)>c$.
  Let $U = \{ x \mid f(x)<c\}$ and $V = \{ x \mid f(x)>c\}$; then $U\cup V = \R$ and $U\cap V = \emptyset$.
  By \cref{thm:compact-connected}, then, either $U=\R$ or $V=\R$, which is a contradiction since $a\in U$ and $b\in V$.
\end{proof}

The crispness of the hypotheses $f$ and $c$ means that the point $x$ does not vary continuously with them.
It is well-known that a continuous choice of $x$ is impossible, and that this is what prevents the classical version of IVT from holding constructively
\parencite[see, for instance,][]{taylor:lamcra}.
Our technology of spatial type theory enables us to state a ``discontinuous'' IVT, even in a world where all types are spaces, and our axioms of classicality and real-cohesion allow us to prove it.

A more common way to constructivize the IVT is to strengthen the hypothesis (e.g.\ that $f$ ``doesn't hover'' or is ``locally non-constant'') or weaken the conclusion (e.g.\ we only find an approximate solution).
We can also use \ref{ax:r3} to prove such a form of IVT, without the need for \ref{ax:lem} or \ref{ax:t}.

First we show that \R\ is connected in a different sense.
The following proof also introduces an idea that we will use repeatedly.
Since $\shape$ is a left adjoint, it takes colimits of types to colimits in the subcategory of discrete types.
Moreover, since \emph{crisp} discrete types are also coreflective, they are closed under colimits, so $\shape$ preserves crisp colimits of types.
Similarly, using \cref{thm:0trunc-discrete}, crisp discrete sets are also closed under set-colimits (i.e.\ 0-truncations of homotopy colimits).
Thus, if we write $\shape_0 A$ for $\trunc0{\shape A}$, then the functor $\shape_0$ reflects sets into discrete sets, so it preserves crisp set-colimits.
We can then use facts about (set-)colimits of the shapes of types to conclude things about the types themselves.

\begin{thmr3}\label{thm:overt-connected} % [\R\ is overt connected]
  Let $U,V::\R\to\prop$ be crisp subsets of \R\ with $U\cup V = \R$.
 % whose union is all of \R, i.e.\ for all $x:\R$ we have $U(x)\vee V(x)$.
  If $\brck{U}$ and $\brck{V}$, % merely
  % (i.e.\ there exists an $x:\R$ such that $U(x)$ and also an $x:\R$ such that $V(x)$),
  then also $\brck{U\cap V}$.
  %(i.e.\ there exists a single $x:\R$ such that $U(x)\wedge V(x)$).
\end{thmr3}
\begin{proof}
  By \cref{thm:join-or}, the assumption ensures that $\R$ is the pushout of $U$ and $V$ over $U\cap V$.
  Thus, the contractible type $\shape \R$ is the pushout of $\shape U$ and $\shape V$ under $\shape(U\cap V)$, and hence the contractible $\shape_0 \R$ is the set-pushout of $\shape_0 U$ and $\shape_0 V$ under $\shape_0(U\cap V)$.
  But $\brck{U}$ and $\brck{V}$, hence also $\brck{\shape_0 U}$ and $\brck{\shape_0 V}$.

  Now if $x:\shape_0 U$ and $y:\shape_0 V$, their images in the contractible type $\shape_0 \R$ must be equal.
  Using the explicit construction of this set-pushout as the set-quotient of an equivalence relation of zigzags on $\shape_0 U + \shape_0 V$, we see that we must in particular have $\brck{\shape_0(U\cap V)}$.
  Finally, since $\shape(U\cap V) \to \shape_0(U\cap V)$ is 0-connected and hence surjective, while $(U\cap V) \to \shape(U\cap V)$ is surjective by \cref{thm:shape-surj}, we also have $\brck{U\cap V}$.
\end{proof}

\textcite{taylor:lamcra} calls the above property \emph{overt connectedness}.
Note that we have only proven it for crisp subsets.

The 0-truncations in the above proof could be omitted, but if we left them out, then instead of the set-based construction of set-pushouts we would need to use the van Kampen theorem \parencite[\S8.7]{hottbook}.
The former needs only the fact that the sets form a $\Pi W$-pretopos \parencite[Theorem 10.1.11]{hottbook}, which requires only function extensionality and propositional extensionality; whereas the latter in general requires full univalence.
Similar remarks apply to the proof of \cref{thm:continuity} below, but in \cref{thm:abfp} we will need to resort to univalence and van Kampen.
(Actually, based on intuition from classical topology, it is natural to conjecture that $\shape U$ is automatically a set for any subset $U\subseteq \R$, but I do not know how to prove this.)

\begin{thmr3}[Approximate IVT]\label{thm:approx-ivt}%; {\cite[Proposition 13.4]{taylor:lamcra}}]
  Let $f::\R\to\R$ be a crisp function and $c::\R$ a crisp real number.
  Then for any $\varepsilon>0$, if there exist $a,b:\R$ such that $f(a)<c<f(b)$, then there exists an $x:\R$ such that $|f(x)-c|<\varepsilon$.
\end{thmr3}
\begin{proof}
  It suffices to consider the case when $\varepsilon:\Q^+$ is a positive rational number.
  Then since $\Q^+$ is discrete, we may assume $\varepsilon$ is also crisp.
  Therefore the sets $U \defeq \{ x:\R \mid f(x) < c+\varepsilon \}$ and $V \defeq \{ x:\R \mid f(x) > c-\varepsilon \}$ are also crisp.
  We have $U\cup V = \R$ by locatedness of Dedekind reals, while by assumption $a\in U$ and $b\in V$.
  Thus, \cref{thm:overt-connected} supplies the desired $x$.
\end{proof}

Note that although \cref{thm:approx-ivt} eliminates classicality axioms, we still have only crisp dependence on $f$ and $a$.
In fact, the approximate IVT (for continuous functions) is contructively provable without any additional assumptions; the most common proofs use countable choice or uniform continuity of $f$, but \textcite{mattf:aivt} has recently given a proof avoiding both of these.
So there is no topological reason for this crispness restriction, but I do not know whether it can be removed using our current methods.

On the other hand, \cref{thm:discont-ivt,thm:approx-ivt} are statements about \emph{all} functions $\R\to\R$.
No explicit continuity hypothesis is required, because in our synthetic world ``all functions are continuous''.
In fact, with a little more work we can actually prove explicitly that all functions are continuous, thereby showing that the axioms of real-cohesion really do nail down precisely the ``topology'' of $\R$ to be the intended one.

\begin{thmr3t}\label{thm:continuity}
  Every crisp function $f::\R\to\R$ is $\varepsilon$-$\delta$ continuous at every crisp real number $a::\R$.
\end{thmr3t}
\begin{proof}
  We will need to assume given crisp $b_1<a$ and $b_2>a$ such that $f(b_1) \apart  f(a)$ and $f(b_2) \apart  f(a)$.
  This can always be ensured by modifying $f$ without changing it near $a$.
  For instance, $f(a+2)-f(a)$ is either $>\frac12$ or $<1$ (and this is crisp since it is a \mprop{}); if the former we can take $b_2=a+2$, while if the latter we can subtract $\max(0,x-a-1)$ from $f$ before taking $b_2 = a+2$.

  As before, we may assume $\varepsilon::\Q^+$ is a crisp positive rational number.
  We may also assume $\varepsilon$ is less than both $|f(b_1)-f(a)|$ and $|f(b_2)-f(a)|$.
  Define
  \begin{align*}
    U &= \{ x:\R \mid |f(x)-f(a)|<\varepsilon/2 \}\\
    V &= \{ x:\R \mid |f(x)-f(a)|>\varepsilon/4 \}
  \end{align*}
  Then $U\cup V=\R$, so that $\R$ is the pushout of $U$ and $V$ under $U\cap V$.
  Let $R_1 = (-\oo,a]$ and $R_2 = [a,\oo)$, and $U_i = U \cap R_i$, $V_i = V \cap R_i$ for $i=1,2$, so that $R_i = U_i \cup V_i$ and $R_i$ is the pushout of $U_i$ and $V_i$ under $U_i\cap V_i$.
  Note that $a\in U_i$ and $b_i\in V_i$ for $i=1,2$.

  Consider the fibers of the map $U\to \shape_0 U$; these are the ``spatially connected components'' of $U$.
  Let $w$ be the image of $a$ in $\shape_0 U$, and $W$ the preimage of $w$ in $U$; thus $W$ is the ``component of $U$ containing $a$''.
  Similarly, define $w_i$ and $W_i$ for $i=1,2$ to represent the component of $U_i$ containing $a$.
  Then the function $\shape_0 U_i \to \shape_0 U$ maps $w_i$ to $w$, so that $W_i$ is a subset of $W$.
  Similarly, let $Z_i$ be the component of $V_i$ containing $b_i$, with image $z_i$ in $\shape_0 V_i$, for $i=1,2$.

  Now as remarked above, $\shape_0$ prerves crisp set-colimits.
  Thus, the contractible $\shape_0 R$ is the set-pushout of $\shape_0 U$ and $\shape_0 V$ under $\shape_0(U\cap V)$, and similarly for $i=1,2$.
  In particular, $w_i$ and $z_i$ are identified in the set-pushout of $\shape_0 U_i$ and $\shape_0 V_i$ under $\shape_0(U_i\cap V_i)$, and so $w_i$ must be equal to the image of some element of $\shape_0(U_i\cap V_i)$.
  Since $X \to \shape_0 X$ is surjective by \cref{thm:shape-surj}, $W_i$ must contain an element of $U_i\cap V_i$.
  Let $y_i$ be such an element.
  Since $y_i \in V_i$, we have in particular $f(y_i) \neq f(a)$, and hence $y_i\neq a$.
  By \cref{ax:amp}, $y_i \apart  a$; so since $y_i \in R_i$ we have $y_1 < a < y_2$.
  Finally, since $y_i \in U_i$, also $y_i \in U$.

  Let $\delta = \min(y_2-a,a-y_1)$; I claim this is the $\delta$ to our $\varepsilon$.
  Thus, suppose $x:\R$ and $|x-a|<\delta$, hence $x \in (y_1,y_2)$.
  (Note that $x$ is \emph{not} crisp.)
  Either $|f(x)-f(a)|<\varepsilon$ or $|f(x)-f(a)|>\varepsilon/2$; the first is what we want, so it suffices to show that the second leads to a contradiction.

  Thus, suppose $|f(x)-f(a)|>\varepsilon/2$.
  Then since any $p\in U$ satisfies $|f(p)-f(a)|<\varepsilon/2$ by definition of $U$, we have $f(p)\neq f(x)$, hence $p\neq x$.
  By \cref{ax:amp}, therefore, $p \apart x$, and thus either $p <  x$ or $x < p$.
  In other words, $\prd{p:U} (p<x)\vee (x<p)$, which is equivalent to $\prd{p:U} (p<x)+(x<p)$ since $y$ cannot be both greater and less than $x$.
  It follows that we can define a map $q : U \to \bool$ by $q(p)=\bfalse$ if $p<x$ and $q(p)=\btrue$ if $x<p$.
  In particular, since $y_1 < x$ and $x < y_2$, we have $q(y_1) = \bfalse$ and $q(y_2) = \btrue$.
  But $y_1$ and $y_2$ are also both in $W$, i.e. they are identified in $\shape_0 U$.
  Since $\bool$ is a discrete set, $q : U \to \bool$ must factor through $\shape_0 U$, and so $q(y_1) = q(y_2)$, a contradiction.
  So it must be that the other case holds: $|f(x)-f(a)|<\varepsilon$, as desired.
\end{proof}

The statement that all functions $\R\to\R$ are $\varepsilon$-$\delta$ continuous is sometimes called ``Brouwer's theorem'', because Brouwer proved it in his ``intuitionistic mathematics'' using choice sequences.
Here we have proved a version of it using instead our \ref{ax:r3} and \ref{ax:t}.
(Both uses of \ref{ax:t} could also be replaced by the assumption that $f$ is ``strongly extensional'', i.e.\ $f(x)\apart f(y)$ implies $x\apart y$.)

\cref{thm:continuity} provides an almost complete answer to the question of ``what is the topology of $\R$'', with ``topology'' interpreted in the sense relevant to our intended model, namely ``functions out of $\R$''.

\begin{corlemr3t}
  The canonical function $\flat(\R\to\R) \to (\flat \R \to \flat \R)$ is an injection, and coincides with the inclusion of the $\varepsilon$-$\delta$ continuous functions.
\end{corlemr3t}
\begin{proof}
  To show it is an injection, suppose given $f,g::\R\to\R$ such that $\flat f = \flat g$ as functions $\flat \R \to\flat\R$, or equivalently $\sharp f = \sharp g$ as functions $\sharp \R \to \sharp \R$.
  We must show that for any $x:\R$ we have $f(x)=g(x)$; but since $\R$ is concrete by \cref{thm:R-concrete}, it suffices to show $f(x)^\sharp = g(x)^\sharp$, which follows from $\sharp f = \sharp g$.

  Now \cref{thm:continuity} implies that this injection lands inside the $\varepsilon$-$\delta$ continuous functions.
  Conversely, suppose $f:\flat\R \to\flat\R$ is $\varepsilon$-$\delta$ continuous.
  By \ref{ax:lem} applied in the universe of discrete types, $f$ is locally uniformly continuous therein (meaning that it is uniformly continuous on all finite intervals).
  Since it suffices to consider intervals with rational endpoints, while $\Q$ is discrete and discrete types are closed under all logical operations, $f$ is also locally uniformly continuous in the world of all types.
  Therefore, so is the composite $\Q \to \flat \R \xto{f} \flat\R \to \R$.
  But it can be proved constructively that any locally uniformly continuous function $\Q\to\R$ extends uniquely to $\R$; see~\textcite[Theorem 5.6.2]{tvd:constructivism-i}.
  And such an extension restricts to $f$ on $\flat \R$, since the two agree on $\Q$ and both are continuous.
\end{proof}

It would be even better to characterize the ``topology'' of $\R\to\R$ itself and so on, or equivalently the functions $\R^n\to\R$ for all $n:\N$, but I don't know how to do this without additional axioms.
However, the general method of proof of \cref{thm:continuity} does have various further applications.
Here we mention only one more: the promised falsity of the analytic LLPO from \cref{sec:ac-2}.

\begin{thmlemr3}\label{thm:no-allpo}
  It is not the case that for all $x:\R$ we have $x\le 0 \vee x\ge 0$.
\end{thmlemr3}
\begin{proof}
  By \cref{thm:flat-detects-empty}, it suffices to show $\neg\flat(\forall x:\R.(x\le 0 \vee x\ge 0))$.
  For contradiction, suppose $\flat(\forall x:\R.(x\le 0 \vee x\ge 0))$, hence $\forall x:\R.(x\le 0 \vee x\ge 0)$ crisply.

  Let $f:\R\to\R$ be such that $f(0)=0$ and $f(x) = x \sin(\frac \pi x)$ for any $x\apart 0$.
  Such a function can be defined constructively without too much trouble, and since it needs no hypotheses it is crisp.
  Let $U = \{ x:\R \mid x\ge 0 \wedge f(x)\ge 0\}$ and $V = \{ x:\R \mid x\ge 0 \wedge f(x)\le 0\}$; then $U$ and $V$ are also crisp.
  Our assumption implies that $U\cup V = [0,\oo)$.

  Classically, $U$ would be the disjoint union of $\{0\}$ with the intervals $[\frac1{n+1},\frac1n]$ for even $n$, and $V$ the similar union for odd $n$.
  (Here and subsequently, we adopt the convention that when $n=0$, the interval $[\frac1{n+1},\frac1n]$ means $[1,\oo)$.)
  In fact, \ref{ax:lem} and \ref{ax:r1} suffice to make this true, by the following argument.

  Let $\N_\oo$ be the type of functions $\N\to \bool$ that take the value $\btrue$ at most once (this is equivalent to the type of non-increasing sequences).
  Now, for any even $n$ and any $x\in U$, we have $(x<\frac1{n-1}) \vee (x>\frac1n)$; but since $x\notin (\frac1n,\frac1{n-1})$, it follows that $(x\le \frac1n) \vee (x > \frac1n)$.
  Similarly, $(x\ge \frac1{n+1}) \vee (x < \frac1{n+1})$, and so the \mprop{} $x\in [\frac1{n+1},\frac1n]$ is decidable (given that $x\in U$).
  %i.e.\ the interval $[\frac1{n+1},\frac1n]$ is a ``detachable'' subset of $U$.
  Thus, we can define $p : U \to \N_\oo$ by
  \[ p(x)_n =
  \begin{cases}
    \btrue &\quad x \in [\frac1{2n+1},\frac1{2n}]\\
    \bfalse &\quad x \notin [\frac1{2n+1},\frac1{2n}].
  \end{cases}\]
  If we write $\underline n:\N_\oo$ for the characteristic sequence of $n:\N$, and $\oo:\N_\oo$ for the constantly-$\bfalse$ sequence, then
  $p^{-1}(\underline n) =[\frac1{2n+1},\frac1{2n}]$ and $p^{-1}(\oo) = \{0\}$.
  On the other hand, as for any function, $U \simeq \sm{n:\N_\oo} p^{-1}(n)$.
  But by LPO (\cref{thm:lpo}), $\N_\oo = \N+\unit$, so we have
  \begin{alignat*}{2}
    U &\simeq \{ 0 \} + \tsm{n:\N} \textstyle[\frac1{2n+1},\frac1{2n}] &&\quad \text{and similarly}\\
    V &\simeq \{ 0 \} + \tsm{n:\N} \textstyle[\frac1{2n+2},\frac1{2n+1}].
  \end{alignat*}
  Now as remarked above, $\shape$ preserves coproducts of crisp types (such as the summands of $U$ and $V$).
  Similarly, for any $P:\N\to\type$, by the generalization of~\cite[Theorem 7.3.9]{hottbook} to monadic modalities we have
  \begin{alignat*}{2}
    \shape \bigl(\tsm{n:\N} P(n)\bigr)
    &\simeq \shape \bigl(\tsm{n:\N} \shape P(n)\bigr)\\
    &\simeq \tsm{n:\N} \shape P(n)
    &&\quad \text{(since this is already discrete, as $\N$ is).}
  \end{alignat*}
  Thus, $\shape U \simeq \unit + \tsm{n:\N} \shape [\frac1{2n+1},\frac1{2n}]$ and similarly for $V$.
  But any closed interval $[a,b]$ is a retract of $\R$, hence $\shape[a,b]$ is a retract of $\shape\R$ and thus contractible.
  Therefore, $\shape U\simeq \unit+\N$ and also $\shape V\simeq \unit+\N$.

  Recalling that $[0,\oo) = U \cup V$, it follows that $[0,\oo)$ is the pushout of $U$ and $V$ under their intersection.
  However, this intersection is just $\{0\} \cup \{ \frac 1n \mid n:\N \}$, which is another copy of $\unit+\N$.
  Since $\shape$ preserves this pushout, the contractible type $\shape [0,\oo)$ is the pushout of $\unit+\N$ and $\unit+\N$ under $\unit+\N$.
  Moreover, the two maps $\unit+\N \toto \unit+\N$ in this pushout can be identified with the identity and with ``$+1$'' acting on $\N$, respectively (arising from the inclusions of $\frac1n$ as the left endpoint of one interval and the right endpoint of another).
  But the pushout of these two maps is $\bool$ rather than $\unit$, since there is nothing to identify the elements of $\unit$ (which correspond to $\{0\}$ in $[0,\oo)$) with anything else.
  This is a contradiction, so our assumption $\flat(\forall x:\R.(x\le 0 \vee x\ge 0))$ is false.
\end{proof}

\cref{thm:no-allpo} explains why classical ``closed'' cell complexes must be avoided in real-cohesion, as mentioned in \cref{sec:shape}.
Namely, the analytic LLPO claims that the evident map $(-\oo,0] + [0,\oo) \to \R$ is surjective, which would be exactly what we need in order to be able to glue (say) the endpoints of a closed interval together seamlessly; its failure means that such a gluing will not give the desired answer.
(Conversely, therefore, if one wants a ``topological topos'' in which classical cell complexes \emph{do} work, one has to give up on local connectedness, as done by~\textcite{ptj:topological-topos}.)

\section{The Brouwer fixed-point theorem}
\label{sec:bfp}

Finally we are ready to attack Brouwer's fixed-point theorem.
We confine ourselves to the simplest version, which is about self-maps of the disc:
\[ \topdisc \defeq \{ (x,y) : \R\times\R \mid x^2 + y^2 \le 1 \}. \]
Recall from \cref{thm:topcircles} that the boundary of the disc
\[ \partial\topdisc \defeq \{ (x,y) : \R\times\R \mid x^2 + y^2 = 1 \} \]
is an equivalent definition of the topological circle $\topcirc$; it is the only one we will use in this section.
Recall also from \cref{thm:circ-circ-2} that $\shape\topcirc \simeq \hocirc$.
We also have:

\begin{lemr3}
  $\shape \topdisc$ is contractible.
\end{lemr3}
\begin{proof}
  Since $\shape\R$ is contractible and $\shape$ preserves finite products, $\shape(\R\times\R)$ is also contractible.
  But $\topdisc$ is a retract of $\R\times \R$, by the function
  \[ (x,y) \mapsto \left(\frac{x}{\max(1,\sqrt{x^2+y^2})},\frac{y}{\max(1,\sqrt{x^2+y^2})}\right). \]
  Hence $\shape \topdisc$ is a retract of $\shape(\R\times\R)$, and thus also contractible.
\end{proof}

In particular, therefore, we have:

\begin{lemuar3}\label{thm:topcirc-not-retract}
  $\topcirc$ is not a retract of $\topdisc$.
\end{lemuar3}
\begin{proof}
  If it were, $\shape\topcirc$ would be a retract of $\shape\topdisc$ and hence contractible.
  But $\shape\topcirc = \hocirc$, which is not contractible since its loop space is $\Z$~\parencite{ls:pi1s1}.
\end{proof}

So much for the homotopy-theoretic part of the proof; we proceed to the topological part.
This might be done in many ways, but we choose to be explicit and calculational to make it obvious that the argument is constructive.
Here is also where the mysterious \ref{ax:t} from \cref{sec:ac-2} reappears.

\begin{lemlemr1t}\label{thm:disc-fpfree}
  If there is a map $f:\topdisc \to \topdisc$ with no fixed point, then $\topcirc$ is a retract of $\topdisc$.
\end{lemlemr1t}
\begin{proof}
  Our assumption is that for all $z:\topdisc$ we have $f(z)\neq z$.
  Write $z=(x,y)$ and $f(z)=(u,v)$; then $f(z)\neq z$ means $(x-u)^2 + (y-v)^2 \neq 0$, hence by \cref{ax:amp} (using \ref{ax:t}) $(x-u)^2 + (y-v)^2 >0$.

  Now the line through $f(z)$ and $z$ can be parametrized as
  \begin{equation}
    (t x + (1-t)u , t y + (1-t)v)\label{eq:topdisc-retract-line}
  \end{equation}
  with $t=0$ being $f(z)$ and $t=1$ being $z$.
  We are interested in the intersections of this line with $\topcirc$, which are given by solving a quadratic equation:
  \begin{align*}
    (t x + (1-t)u)^2 + (t y + (1-t)v)^2 &= 1 \quad\iff\\
    ((x-u)t + u)^2 + ((y-v)t + v)^2 &= 1 \quad\iff\\
    % (x-u)^2 t^2 + 2(x-u)u t + u^2 + (y-v)^2t^2 + 2(y-v)v t + v^2 &= 1\\
    ((x-u)^2 + (y-v)^2)t^2 + 2((x-u)u + (y-v)v)t + (u^2+v^2-1) &= 0
  \end{align*}
  The standard quadratic formula yields the solutions
  \begin{align*}
    t
    % &= \frac{-2((x-u)u + (y-v)v) \pm \sqrt{4((x-u)u + (y-v)v)^2 - 4((x-u)^2 + (y-v)^2)(u^2+v^2-1)}}{2((x-u)^2 + (y-v)^2)}\\
    % &= \frac{-((x-u)u + (y-v)v) \pm \sqrt{((x-u)u + (y-v)v)^2 - ((x-u)^2 + (y-v)^2)(u^2+v^2-1)}}{(x-u)^2 + (y-v)^2}.
    &= \frac{-B \pm \sqrt{D}}{A}
  \end{align*}
  where
  \begin{align*}
    B &= ((x-u)u + (y-v)v)\\
    D &= {((x-u)u + (y-v)v)^2 - ((x-u)^2 + (y-v)^2)(u^2+v^2-1)}\\
    A &= {(x-u)^2 + (y-v)^2}.
  \end{align*}
  To ensure that this yields real solutions constructively, we need to know that $A\apart 0$ and $D\ge 0$.
  As we saw above, the former follows from our assumption that $f(z)\neq z$.
  For the latter, we can compute
  \begin{equation*}
    D = (x-u)^2 + (y-v)^2 - (v x - u y)^2.
  \end{equation*}
  Now $(x-u)^2 + (y-v)^2$ is the squared length of one diagonal of the parallelogram spanned by $z$ and $f(z)$ regarded as vectors originating at the origin, while $(v x - u y)^2$ is the squared area of that parallelogram.
  The area of a parallelogram is at most half the product of the lengths of its diagonals; thus $D\ge 0$ as long as the other diagonal of the parallelogram in question is $\le 2$.
  But this other diagonal is the vector sum $z + f(z)$, which has magnitude $\le 2$ since both $z$ and $f(z)$ lie in $\topdisc$.

  Thus we constructively have two real solutions to our quadratic equation for $t$.
  We choose the one in which the $\pm$ sign is $+$.
  Substituting this value of $t$ into the equation~\eqref{eq:topdisc-retract-line} of our line, we obtain a formula for a point in $\R^2$ that is a function of $(x,y)$ and $(u,v)$, hence (recalling that $(u,v) = f(x,y)$) of $z = (x,y) : \topdisc$ only.
  Moreover, by construction, this point always lies in $\topcirc$.
  Let us denote it $r(z)$; thus we have $r:\topdisc \to \topcirc$.

  It remains to show that $r$ is a retraction of the obvious inclusion, i.e.\ that if $z = (x,y)$ lies in $\topcirc$, then $r(z)=z$.
  For this, it suffices to show that if $x^2+y^2 = 1$, then our chosen value of $t$ becomes equal to $1$.
  This is equivalent to asking that $A+B = \sqrt{D}$, and this in turn is equivalent to asking that $A+B\ge 0$ and that $(A+B)^2-D=0$.
  For the first, we compute
  \begin{align*}
    A + B &= x^2 + y^2 - u x - v y\\
    &= 1 - (ux + vy)
  \end{align*}
  But $ux+vy$ is the dot product of $z$ and $f(z)$, which is $\le 1$ since both have magnitude $\le 1$.
  Finally, for the second we compute
  \[ (A+B)^2-D = ((x-u)^2 + (y-v)^2) (x^2 + y^2 - 1) \]
  which clearly vanishes if $x^2+y^2=1$.
\end{proof}

Combining \cref{thm:topcirc-not-retract,thm:disc-fpfree}, we see that an arbitrary $f:\topdisc \to \topdisc$ cannot fail to have a fixed point.
To deduce a positive statement from this, we invoke \ref{ax:lem}.

\begin{thmualemr3t}[Sharp Brouwer Fixed-Point Theorem]\label{thm:bfp}
  For any function $f:\topdisc \to \topdisc$, we have $\sharp\brck{\sm{x:\topdisc} f(x)=x}$.
  That is, any self-map of $\topdisc$ sharply has a fixed point.
\end{thmualemr3t}
\begin{proof}
  By \cref{thm:codiscrete-notnot-2}, the claim is equivalent to $\neg\neg \sm{x:\topdisc} f(x)=x$.
  But this is what we have just proven.
\end{proof}

Equivalently, we can say:

\begin{corualemr3t}[Crisp Brouwer Fixed-Point Theorem]
  For any crisp function $f::\topdisc \to \topdisc$, there exists $x:\topdisc$ such that $f(x) = x$.
\end{corualemr3t}
\begin{proof}
  If $f$ is crisp, then so is the fixed point asserted by \cref{thm:bfp}, so we can strip the $\sharp$ from it.
\end{proof}

Topologically, these formulations both say the same thing: every continuous self-map of $\topdisc$ has a fixed point, but such a fixed point cannot be selected continuously.
Since the existential in both theorems is propositionally truncated, even if the $\sharp$ or crispness could be omitted we would not be asserting the existence of a single continuous \emph{function} selecting fixed points.
However, according to \cref{thm:local-sections}, we would be asserting that such a function exists \emph{locally}, i.e.\ that any given function has a neighborhood of functions on which fixed points can be selected continuously; and this is just as impossible.
Thus, the versions of the theorem that we have proven really are the strongest we can expect.

On the other hand, just as we did with the Intermediate Value Theorem in \cref{sec:continuity}, we can eliminate \ref{ax:lem} and \ref{ax:t} by constructing an \emph{approximate} fixed point.
Our proof of this will also involve more synthetic homotopy theory: in addition to $\pi_1(\hocirc) = \mathbb{Z}$, we will use the van Kampen theorem for pushouts.

\begin{thmuar3}[Approximate Brouwer Fixed-Point Theorem]\label{thm:abfp}
  For any crisp function $f::\topdisc \to \topdisc$ and $\varepsilon>0$, there exists an $x:\topdisc$ such that $|f(x)-x|<\varepsilon$.
\end{thmuar3}
\begin{proof}
  As before, we may assume that $\varepsilon$ is rational, and hence that it is crisp.
  We may also assume $\varepsilon <1$.
  Define
  \[ g(x) = \begin{cases}
    f(x) \cdot \frac{\min(|f(x)|,1-\varepsilon/2)}{|f(x)|} &\qquad \text{if}\; |f(x)|>1-\varepsilon\\
    f(x) &\qquad\text{if}\; |f(x)|< 1-\varepsilon/2
  \end{cases}
  \]
  The case division is constructively valid by \cref{thm:join-or}, since one of the two cases must be true, and if both are true then the two definitions agree.
  And since $\varepsilon<1$, if $|f(x)|>1-\varepsilon$ then $|f(x)|>0$, so that $\frac{\min(|f(x)|,1-\varepsilon/2)}{|f(x)|}$ makes sense.

  Now this definition ensures that $|f(x)-g(x)|<\varepsilon/2$, so it suffices to find an $x:\topdisc$ such that $|g(x)-x|<\varepsilon/2$.
  Moreover, we have $|g(x)-x| > \varepsilon/4$ for all $x:\topcirc$.

  Let $U = \{ x:\topdisc \mid |g(x)-x| > \varepsilon/4 \}$ and $V = \{ x:\topdisc \mid |g(x)-x| < \varepsilon/2 \}$.
  Then $U \cup V = \topdisc$, exhibiting $\topdisc$ as a crisp pushout of $U$ and $V$ under their intersection.
  Therefore the contractible type $\shape \topdisc$ is also the pushout of $\shape U$ and $\shape V$ under $\shape (U\cap V)$.

  However, $\topcirc$ is contained in $U$ by the construction of $g$, and the construction of the retraction $r$ in \cref{thm:disc-fpfree} works for all $x:U$.
  Thus, $\topcirc$ is a retract of $U$, and so $\shape \topcirc$ is a retract of $\shape U$.
  But since $\shape\topcirc = \hocirc$, it contains a nonidentity identification, and hence so does $\shape U$.
  On the other hand, this identification must be mapped to the identity in the contractible type $\shape\topdisc$.

  Now recall that the van Kampen theorem \parencite[\S8.7]{hottbook} identifies the 0-truncated equality types of a (homotopy) pushout as a certain set-quotient.
  Inspecting this construction we see that the only way two identifications in $\shape U$ can become equal in $\shape U \sqcup^{\shape(U\cap V)} \shape V$ is if we have $\brck{\shape V}$.
  Hence, by \cref{thm:shape-surj}, we also have $\brck{V}$.
\end{proof}

As in \cref{thm:approx-ivt}, there seems no topological reason for the crispness of $f$, but I do not know how to remove it.
Note that although the technique is similar to that in \cref{thm:overt-connected,thm:continuity}, here we need to use $\shape$ rather than $\shape_0$, and likewise we need the van Kampen theorem for pushouts that are not sets (whose proof also involves the full univalence axiom).

\printbibliography

\end{document}